[1994/12/01]
\documentclass[reqno,13pt]{amsart}
\usepackage{bbm}
\def\1{\mathbbm{1}}

\usepackage{amsmath,amsthm}
\usepackage[mathscr]{euscript}
\usepackage{graphicx}
\usepackage{amsmath,amssymb}
\usepackage{circuitikz}

\usepackage{amsmath}

\usepackage{amsfonts}

\usepackage{hyperref}

\hypersetup{
	colorlinks=true,                         
	linkcolor=red, 
	citecolor=blue, 
	urlcolor=red  } 

\usepackage{float}

\usepackage[percent]{overpic}

\usepackage{multicol}

\usepackage{subfloat}

\usepackage{subcaption}
\usepackage{graphicx}
\usepackage{graphics}
\usepackage{dblfloatfix}

\usepackage{fixltx2e}

\usepackage{latexsym}

\usepackage{amsfonts}

\usepackage{fancybox}

\usepackage[utf8]{inputenc}

\usepackage{amssymb}

\usepackage{fancyhdr}

\usepackage{amsthm}

\usepackage{cases}

\usepackage{tikz}

\newtheorem{theorem}{Theorem}[section]

\newtheorem{lemma}{Lemma}[section]

\newtheorem{remark}{Remark}[section]

\newtheorem{definition}{Definition}[section]

\newtheorem{proposition}{Proposition}[section]

\numberwithin{equation}{section}

\begin{document}
	\title{ \textbf{Controllability and Positivity Constraints in Population Dynamics with age, size Structuring and Diffusion } }
	\date{}
	\author{Yacouba Simpore* and Umberto Biccari}
	
	\thanks{This work has been supported by the European Research Council (ERC) under the European Union’s Horizon 2020 research and innovation programme (grant agreement No 694126-DYCON).}
	
	\address{Yacouba Simpore, Universit\'e de Fada N'Gourma-Burkina Faso\\Laboratoire LAMI, Universit\'e Joseph Ki zerbo\\ Laboratoire LAMIA, Universit\'e Norbert Zongo; simplesaint@gmail.com}
	\address{Umberto Biccari, Chair of Computational Mathematics, Fundación Deusto, Universityof Deusto, 48007 Bilbao, Basque Country, Spain}
	
	
	\thanks{This project has received funding from the European Research Council (ERC) under the European Union’s Horizon 2020 research and innovation programme (grant agreement NO: 694126-DyCon). The
		work of U.B. is supported by the Grant PID2020-112617GB-C22 KILEARN of MINECO (Spain). }
	
	\maketitle
	\begin{abstract}
		In this article, we consider the infinite dimensional linear control system describing the Population Models Structured by Age, Size, and Spatial Position. The control is localized in the space variable as well as with respect to the age and size. For each control  support, we give an estimate of the time needed to control the system to zero. We prove the null controllability of the model, using a technique avoids the explicit use of parabolic Carleman estimates. Indeed, this method combines final-state observability estimates with the use of characteristics and with $L^{\infty}$ estimates of the associated semigroup.
	\end{abstract}
	\section{Introduction and mains results}
	It has been recognized that age structure alone is not adequate to explain the population dynamics of some species (\cite{b9,o,kk,z}). The size of individuals could also be used to distinguish cohorts. In principle there are many ways to differentiate individuals addition to age, such as body size and dietary requirements or some other physiological 
	variables and behavioral parameters. For the sake of simplicity and the reason of similarity of mathematical treatment we assume here that only one internal variable is involved. Meanwhile, we consider the velocity of internal variable to be constant. Note that this assumption is not restrictive since in  it was pointed out that the problem in which the growth of an internal variable does not increase at the same rate as age can be converted to the constant velocity case.\\
	
	In these models going back to Glenn Webb, the state space of the system is $K=L^2((0,A)\times(0,S);L^2(\Omega)),$ where $A$ and $S$ denote respectively the maximal age and the maximal size an individual can attain; $\Omega\subset \mathbb{R}^n$ (with $n\in \mathbb{N}$ in general but $n=3$ for real life application) is an open bounded which represents the space variable.\\
	Let $y(x,a,s,t)$ be the distribution density of individuals with respect to age, size and position $x\in \Omega$ and some time $t\in(0,T)\text{, }T>0.$\\ 
	According to Webb \cite{b9} the function $y$ satisfies the degenerate parabolic partial differential equation.
	\begin{equation}
		\left\lbrace\begin{array}{ll}
			\dfrac{\partial y}{\partial t}+\dfrac{\partial y}{\partial a}+\dfrac{\partial y }{\partial s}-\Delta y+\mu(a,s)y=mu &\hbox{ in }Q ,\\ 
			\dfrac{\partial y}{\partial\nu}=0&\hbox{ on }\Sigma,\\ 
			y\left( x,0,s,t\right) =\displaystyle\int\limits_{0}^{A}\int\limits_{0}^{S}\beta(a,\hat{s},s)y(x,a,\hat{s},t)da d\hat{s}&\hbox{ in } Q_{S,T} \\
			y\left(x,a,s,0\right)=y_{0}\left(x,a,s\right)&
			\hbox{ in }Q_{A,S};\\
			y(x,a,0,t)=0& \hbox{ in } Q_{A,T}.
		\end{array}\right.
		\label{2}
	\end{equation}
	where,  $y_{0}$ is given in $K$, 
	$Q=\Omega\times (0,A)\times (0,S)\times(0,T)$,
	$\Sigma=\partial\Omega\times(0,A)\times (0,S)\times(0,T)$, $ Q_{S,T}=\Omega\times (0,S)\times (0,T)$, $Q_{A,S}=\Omega\times (0,A)\times (0,S)$ and $Q_{A,T}=\Omega\times (0,A)\times (0,T).$\\
	Moreover, the positive function $\mu$ denotes the natural mortality rate of individuals of age $a$ and size $s,$ supposed to be independent of the spatial position $x$ and of time $t$. The control function is $u$, depending on $x$, $a$, $s$ and $t$, where as $m$ is the characteristic function. \\
	We denote by $\beta$ the positive function describing whereas the fertility rate age-size depending on $(a,\hat{s})$ and also depending on the size $s$ of the newborns. The fertility rate $\beta$ is supposed to be independent of the spatial position $x$ and of time $t,$ so that the density of newly born individuals at the point $x$ at time $t$ of size $s$ is given by
	\[\displaystyle\int\limits_{0}^{A}\int\limits_{0}^{S}\beta(a,\hat{s},s)y(x,a,\hat{s},t)da d\hat{s}.\]
	In these models, size is considered to be an individual-specific variable, and may include size, volume, length, maturity, bacterial or viral load, or other physiological or demographic property. It is assumed that size increases in the same way for everyone in the population and is controlled by the growth function $g(s),$ then $\dfrac{\partial (g(s)y) }{\partial s}$ is the growth. Here $g(s)=1$. For more details about the modeling of such system and the biological significance of the hypotheses, we refer to Glenn Webb \cite{b9}.\\
	
	In the sequel we assume that the fertility rate $\beta$ and the mortality rate $\mu(a,s)=\mu_1(a)+\mu_2(s)$ satisfy the demographic property:
	\begin{align*}
		(H_1)=
		\left\lbrace
		\begin{array}{l}
			\mu_1\left(a\right)\geq 0 \text{ for every } (a)\in (0,A)\\
			\mu_1\in L^{1}\left([0,A^*]\right)\hbox{ for every }\hbox{ } A^*\in [0,A) \\
			\displaystyle\int\limits_{0}^{A}\mu_1(a)da=+\infty
		\end{array} \right..
	\end{align*}
	,
	\begin{align*}
		(h_1)=
		\left\lbrace\begin{array}{l}
			\mu_2\left(s\right)\geq 0 \hbox{ for every } (a)\in (0,A)\\
			\mu_2\in L^{1}\left([0,S^*]\right)\hbox{ for every }\hbox{ } S^*\in [0,A) \\
			\displaystyle\int\limits_{0}^{S}\mu_2(s)ds=+\infty
		\end{array} \right..
	\end{align*}
	,
	\begin{align*}
		(H_{2})=
		\left\lbrace\begin{array}{l}
			\beta(a,\hat{s},s)\in C\left([0,A]\times [0,S]^2\right)\\
			\beta\left(a,\hat{s},s\right)\geq0 \hbox{ in } \left([0,A]\times [0,S]^2\right)
		\end{array}
		\right..
	\end{align*}
	In probabilistic terms, fertility $\beta$ can mean the probability of an individual of age $a$ and of size $\hat{s}$ giving birth to an individual of size $s.$\\
	We consider the following hypotheses:
	\[
	(H_3):
	\begin{array}{c}
		\beta(a,\hat{s},s)=0 \text{ } \forall a\in (0,\hat{a}) \text{ for some } \hat{a}\in (0,A). 
	\end{array}
	\]
	\[
	(H_4):
	\begin{array}{c}
		\beta(a,\hat{s},s)=0 \text{ } \forall s\in (s_e,S) \text{ for some } s_e\in (0,S). 
	\end{array}
	\]
	Also we denote by \[Q_{1}=\omega\times (a_1,a_2)\times (s_1,s_2),\]with $0\leq a_1<a_2\leq A\text{ and }0\leq s_1<s_2\leq S.$ and that $\omega\subset \Omega$ is an open set; \[Q_{2}=\{(x,a,s)\hbox{ such that } (x,a,s)\in \omega\times(a_1,a_2)\times(0,S)\hbox{ and } a-a_0<s<a+s_e\}.\] where $s_e\in (0,S],$ and $a_0\in [a_1,\hat{a}].$\\
	We denote by $m_i$ the characteristic function of $Q_{i}$ where $i\in \{1,2\}$; and $m_i$ the support of the control $u_i$.\\
	\textbf{NB :} All the figures on this paper represented in dimension two are in reality in dimension three; They are cubic in shape seen from above. The height represents the time $t.$ The characteristics evolve from top to bottom but in the positive direction with respect to the age variable $a$ and the size variable $s.$
	\begin{figure}[H]
		\begin{subfigure}{0.55\textwidth}
			\begin{overpic}[scale=0.25]{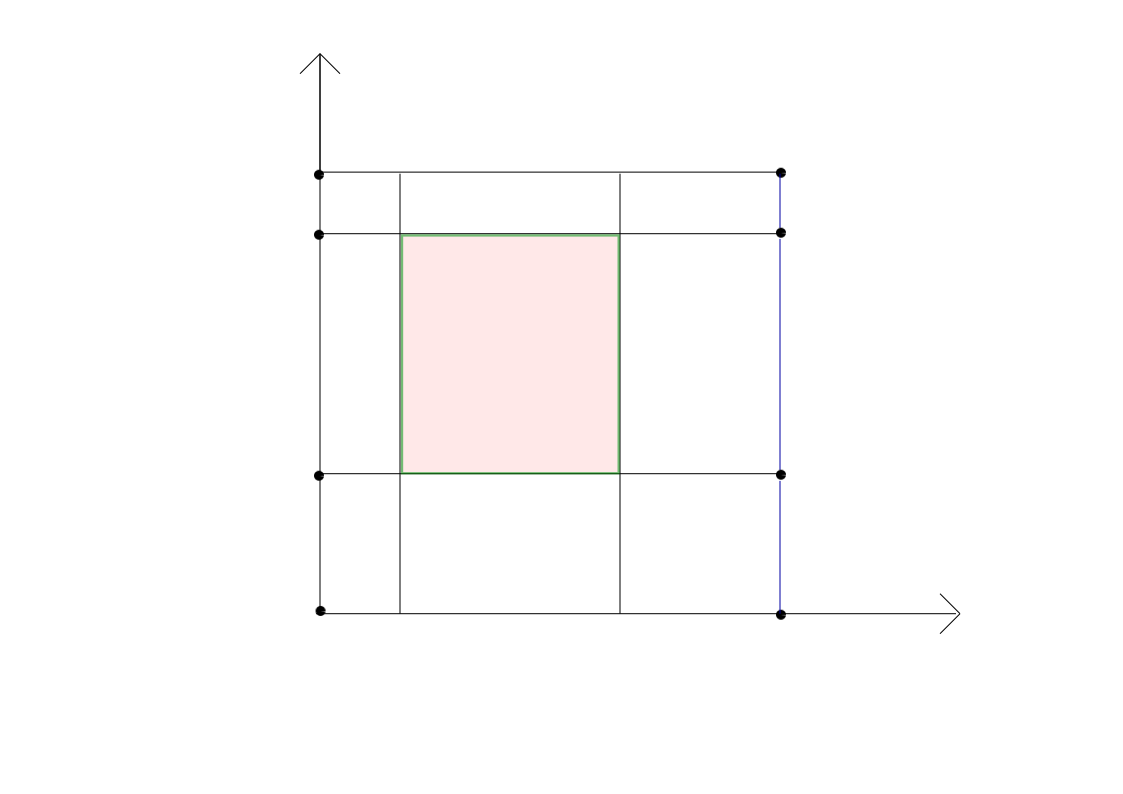}
				\put (21,53) {$S$}
				\put (22,49.5) {$s_2$}
				\put (40,34.5) {$Q_1$}	
				\put (22,27.5) {$s_1$}
				\put (35,12) {$a_1$}
				\put (52,12) {$a_2$}
				\put (67,10.5) {$A$}
			\end{overpic}
			\caption{Here is the support $m_1$ without the space variable see from above}
		\end{subfigure}\quad
		\begin{subfigure}{0.5\textwidth}
			\begin{overpic}[scale=0.25]{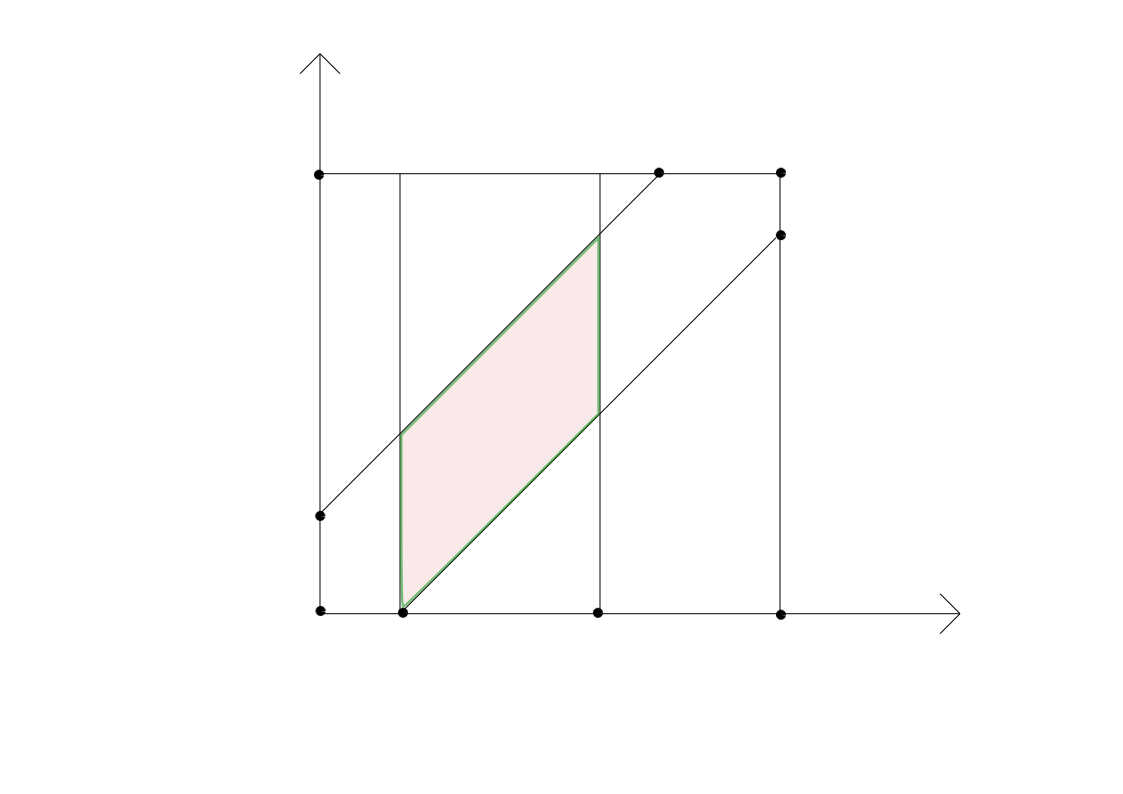}
				\put (20,53) {$S$}	
				\put (22,23) {$s_e$}
				\put (42,32.5) {$Q_2$}
				\put (35,12) {$a_1$}
				\put (50,12) {$a_2$}
				\put (67,10.5) {$A$}	
			\end{overpic}
			\caption{Here is the support $m_2$ without the space variable see from above}
		\end{subfigure}
		\caption{The Figures above illustrate  the different control supports}
	\end{figure}
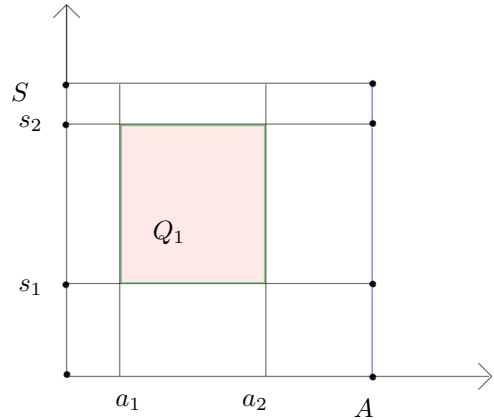
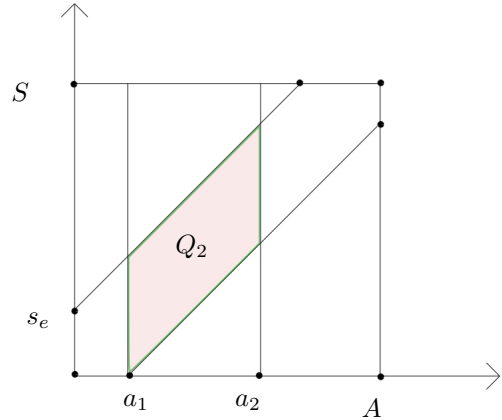
	\begin{definition}
		$\bullet$ The real $s_e$ is the maximal size expectancy of the newborn.\\
		$\bullet$ The real $\hat{a}$ is the minimal age from which individuals become fertile. We will call it "the minimal age of fertility".\\
	\end{definition}
	From the previous data, we denote by $T_0=\max\{s_1,S-s_2\}\hbox{ and } T_1=\max\{a_1+S-s_2,s_1\}.$ \\
	Our main results are the following:
	\begin{theorem}
		Assume that $\beta$ and $\mu$ satisfy the conditions ($H1)-(H2)-(h_1)$ above.\\  Assume that $a_1< \hat{a}$, $T_0<\min\{a_2-a_1,\hat{a}-a_1\},$ and the assumption $(H3)$ hold.\\
		Then for every $T>A-a_2+T_1+T_0$ and for every $y_0\in L^2(\Omega\times(0,A)\times (0,S)) $
		there exists a control $u_1\in L^2(Q_1)$ such that the solution $y$ of $(\ref{2})$ satisfies 
		\[ y(x,a,s,T)=0\hbox{ a.e }x \in \Omega\hbox{ } a \in (0,A)\hbox{ and } s \in (0,S).\]	
	\end{theorem}
	One can also have a result of null controllability by acting on the same individuals of age $a_1$ until the age $a_2$, with additional constraints on the kernel $\beta$ but without the condition $T_0<\min\{a_2,\hat{a}\}-a_1.$
	\begin{theorem}
		Assume that $\beta$ and $\mu$ satisfy the conditions ($H1)-(H2)$ above.\\ Assume that the hypotheses $(H3)-(h1)$ and $(H4)$ above.
		Then for every $T>\sup\{S-s_e,A-a_0+a_1\}$ and for every $y_0\in L^2(\Omega\times(0,A)\times (0,S)),$ there exists a control $u_2\in L^2(Q_{2})$ such that the solution $y$ of $(\ref{2})$ satisfies 
		\[ y(x,a,s,T)=0\hbox{ a.e. }  x \in \Omega\hbox{ } a \in (0,A)\hbox{ and } s \in (0,S).\]
	\end{theorem}
	\begin{remark}
		The proof of the theorem comes down to the observability of the adjoint system of $(\ref{2}).$
	\end{remark}
	\begin{remark}
		\begin{itemize}
			\item [1-] The assumptions $(H3)$ and $a_1<\hat{a}$ in the case where the control support is $Q_1$ are fundamental for the proof of the theorem 1.1. For the moment, we only manage to prove this result when these conditions are verified.
			\item [3-] In the case where the control support is $Q_2,$ the condition $a_0\in [a_1,\hat{a}]$ associated with the hypotheses $(H4)$  allows us to establish the observability inequality of the adjoint system.		
			\item [4-] In all cases the assumptions $ (H1) $ and $ (H2) $ are necessary for the problem to be well posed.	
		\end{itemize}
	\end{remark}
	Many versions of the age-size structured model, both linear and nonlinear, have been investigated, and seminal treatments of such models are given by Metz and Diekmann \cite{j} and Tucker and Zimmerman \cite{z}.\\
	Spatial structure in linear and nonlinear age or size structured models has also been investigated by many researchers, including \cite{j,b9,z} in the context of the existence of solutions by the semi-group method.\\
	Also, many results of null controllability of the Lotka-McKendrick system without or with spatial diffusion have been obtained by several authors \cite{yac3,dy,b1,B1,abst,b6,B}. The first result was obtained by Ainseba and Anita \cite{b1}. They proved that the Lotka-Mckendrick system with spatial diffusion can be driven to a steady state in any arbitrary time  $T>0$ keeping the positivity of the trajectory, provided the initial data is close to the steady state and the control acts in a spatial subdomain $\omega\subset\Omega$ but for all ages.\\
	Recently, Maity \cite{b6} proved that null controllability can be achieved by controls supported in any subinterval of $(0,A)$, provided we control before the individuals start to reproduce.\\
	In \cite{b2}, Hegoburu and Tucsnak proved that the same system is null controllable for all ages and in any time by controls localized with respect to the spatial variable but active for all ages.\\
	In \cite{Genni} the authors study the null controllability property for a single population model in which the population $y$ depends on time $t$ space $x,$ age $a$ and size $\tau$ using essentially the Carleman estimates. It should be noted that in their work the growth function $g$ is variable and depends only on the size $s.$\\
	The authors D. Maity, M. Tucsnak and E. Zuazua in \cite{dy} solved a problem of null controllability of the system Lotka-Mckendrick with spatial diffusion where the control is localized in the space variable as well as with respect to the age. The method combines final-state observability estimates with the use of characteristics and the associated semigroup. The same method is used in \cite{yac2}, to prove the null controllability of the nonlinear Lotka-mckendrick system where the control is also localized in the space variable as well as with respect to the age.\\
	The author Y.Simpor\'e and O.Traor\'e have recently in \cite{yac1} solved a problem of null controllability of a nonlinear age, space and two-sex structured population dynamics model; the system being a cascade system and the coupling being at the level of renewal terms, the technique introduced in \cite{dy} was the most suitable for resolution. \\
	This paper is devoted to introduce a null controllability and null controllability with positivity constraints of population dynamics structuring with age size and spatial diffusion.\\
	Indeed, it has been recognized that age structure alone is not adequate to explain the population dynamics of some species. The size of individuals could also be used to distinguish cohorts. In principle there are many ways to differentiate individuals in addition to age, such as body size and dietary requirements or some other physiological variables and behavioural parameters. Hence the interest of this work.\\
	The main novelties brought in by our paper are:\\
	$\bullet$ The first novelty in this work is the introduction of the null controllability for this model.
	For each control support $m_i,$ we establish an estimate of the time $T$ depending on $m_i$, necessary to control the system at zero using some assumptions on the kernel $\beta.$\\
	$\bullet$ Here, we use  the technique introduce by D. Maity, M. Tucsnak and E. Zuazua (\cite{dy}) combining final-state observability estimates with the use of characteristics and with $L^\infty$ estimates of the associated semigroup. This technique is interesting because it avoids the establishment of Carleman's inequality for the model that seems to be very expensive.\\
	With this technique, we show that our global controllability result applies to individuals of all ages and sizes, without needing to exclude ages and sizes in a neighborhood of zero; which will probably not be the case with Carleman's estimates.\\
	$\bullet$ Controllability with positivity constraints is proved, as far as we know for the first time, with a control which is localised both in age and with respect to the space variable. The methodology employed to obtain this result is based on duality and $L^\infty$ estimates for parabolic PDEs.\\
	The remaining part of this work is organized as follows:\\
	In Section 2, we first recall some basic facts about the semigroup of the population dynamics structured by age, size and spatial position. We next formulate our control problem in a semigroup setting and we define the associated adjoint semigroup.\\ Section 3 is divided in two parts. In first part, we prove the final state observability on space-age for the adjoint system and, as a consequence, we obtain the proof of the result of the Theorem 1.1 and, in the second part, we  prove the final state observability space-size for the adjoint system and, as a consequence, we obtain the proof of the result of the Theorem 1.2.\\
	Section 4 is first devoted to the preliminary study of the existence of a stationary solution for the model, then to the proof that the controllability between positive stationary states can be reached in a sufficiently large time.\\ We give the description of possible extensions and open questions in the section 5.\\
	\section{Population Models Structured by Age, Size, and Spatial Position Semigroup}
	In this section, we provide some basic results on the population semigroup for the linear age, size structured model with diffusion and its adjoint operator. We give the existence of the semigroup in  the Hilbert space $K=L^2((0,A)\times(0,S);L^2(\Omega)).$\\		
	For it, we define the operator: $\mathcal{A}:K\longrightarrow K$ as follows:
	\[\mathcal{A}\varphi=\Delta \varphi-\dfrac{\partial \varphi}{\partial a}-\dfrac{\partial \varphi}{\partial s}-\left(\mu_1(a)+\mu_2(s)\right)\varphi\]
	where
	\begin{equation*}
		D(\mathcal{A})=	\left\lbrace\begin{array}{l}
			\varphi\in L^2\left((0,A)\times(0,S);H^2(\Omega)\right)/ \mathcal{A}\varphi\in K\hbox{, }\\
			\dfrac{\partial \varphi}{\partial a}\hbox{, }\dfrac{\partial \varphi}{\partial s}\in L^2\left((0,A)\times(0,S);H^1(\Omega)\right)\\
			\dfrac{\partial\varphi}{\partial\nu}|_{\partial\Omega}=0\hbox{, }\\ \varphi(a,0,s)=0\hbox{, }\varphi(x,0,s)=\int\limits_{0}^{A}\int\limits_{0}^{S}\beta(a,\hat{s},s)\varphi(x,a,\hat{s})da d\hat{s}.
		\end{array}\right.\label{ywe}
	\end{equation*}
	\begin{theorem}
		Under the assumption $(H1)-(H2)-(h1),$ the population (age-size and diffusion) operator with diffusion $(\mathcal{A},D(\mathcal{A}))$ is the infinitesimal generator of a strongly continuous semigroup $\textbf{U}$ on $K.$	
	\end{theorem}
	To give a clear idea on the semi-group, let $\phi\in K,$ we consider the following system:
	\begin{equation}
		\left\lbrace\begin{array}{ll}
			\dfrac{\partial y}{\partial t}+\dfrac{\partial y}{\partial a}+\dfrac{\partial y }{\partial s}-\Delta y+\mu_1(a)y=u &\hbox{ in }\Omega\times(0,A)\times(0,S)\times(0,+\infty) ,\\ 
			\dfrac{\partial y}{\partial\nu}=0&\hbox{ on }\partial\Omega\times(0,A)\times(0,S)\times(0,+\infty),\\ 
			y\left( x,0,s,t\right) =\displaystyle\int\limits_{0}^{A}\int\limits_{0}^{S}\beta(a,\hat{s},s)y(x,a,\hat{s},t)da d\hat{s}&\hbox{ in }\Omega\times(0,S)\times(0,+\infty) \\
			y\left(x,a,s,0\right)=y_{0}\left(x,a,s\right)&
			\hbox{ in }\Omega\times(0,A)\times(0,S);\\
			y(x,a,0,t)=0& \hbox{ in }\Omega\times(0,A)\times(0,+\infty).
		\end{array}\right..
		\label{xdc}
	\end{equation}
	We denote by $w(x,\lambda)=y(x,a+\lambda,s+\lambda,\lambda).$ The function $w$ verifies the following system:
	\begin{equation}
		\left\lbrace
		\begin{array}{l}
			w'(x,\lambda)=(\Delta -\mu(a-\lambda)-\mu_2(s-\lambda))w\\
			w(x,0)=\phi(x,a,s).
		\end{array}
		\right.
	\end{equation}
	The method of characteristics yields the following formula for the density $y$ :
	\begin{equation}
		y(x,a,s,t)=\left\lbrace
		\begin{array}{l}
			X(t)\phi(x,a-t, s-t)\text{ if } 0<t<a\text{ and }  0<t<s,\\
			X(a)y(x,0,s-a,t-a) \text{ if }0<a<t \text{ and } 0<a<s,\\
			0 \text{ otherwise }.
		\end{array}
		\right.\label{mop}
	\end{equation}
	where $X$ is a  semigroup (see \cite[\hbox{ pp 33-35 }]{b9}) and \cite{kk}.\\
	To obtain $y(x,a,s,t)$ for $0<a<t \text{ and } 0<a<s$  we must solve the boundary condition using \[y(x,0,s,t)=\int\limits_{0}^{A}\int\limits_{0}^{S}\beta(a,\hat{s},s)y(x,a,\hat{s},t)da d\hat{s}\] and the representation $(\ref{mop}).$ \\
	We denote by: \[b_{\phi}(s,t)=y(x,0,s,t)=\int\limits_{0}^{A}\int\limits_{0}^{S}\beta(a,\hat{s},s)y(x,a,\hat{s},t)da d\hat{s},\] where 
	\[b_{\phi}:[0,S)\times[0,\infty)\longrightarrow L^{2}(\Omega)\]
	satisfies the following integral equation :
	\[b_{\phi}(s,t)=\int\limits_{0}^{t}\int\limits_{a}^{S}\beta(a,\hat{s},s)X(a)y(x,0,\hat{s}-a,t-a)da d\hat{s}+X(t)\int\limits_{t}^{A}\int\limits_{t}^{S}\beta(a,\hat{s},s)\phi(x,a-t,\hat{s}-t)da d\hat{s}\]
	\begin{equation}
		=\int\limits_{0}^{t}\int\limits_{a}^{S}\beta(a,\hat{s},s)X(a)b_{\phi}(\hat{s}-a,t-a)da d\hat{s}+X(t)\int\limits_{0}^{A-t}\int\limits_{0}^{S-t}\beta(a+t,\hat{s}+t,s)\phi(x,a,\hat{s})da d\hat{s}.\label{semi1}
	\end{equation}
	Let \[V=\{f\in C([0,S), L^2(\Omega)):  \lim_{s\longrightarrow S}f(s)=0\},\] let $W=C([0,T_1],V)$ where $T_1>0.$\\ Let
	\begin{align*}
		c_{\phi}(s,t)=\left\lbrace\begin{array}{l}
			X(t)\int\limits_{0}^{A-t}\int\limits_{0}^{S-t}\beta(a+t,\hat{s}+t,s)\phi(x,a,\hat{s})da d\hat{s}	\hbox{ if } A>t \hbox{ and } S>t	\\
			0 \hbox{ otherwise }
		\end{array}\right.
	\end{align*}
	Let $C_{\phi}(t)(s)=c_{\phi}(s,t) \text { if } 0\leq t\leq T_1\text{, }s\in(0,S),$ and then $C_{\phi}\in W.$\\
	Define $P:[0,+\infty )\longrightarrow B(V)$ (the space of bounded linear operators in V) by \[(P(a)f)(s)=X(a)\int_{a}^{S}\beta(a,\hat{s},s)f(\hat{s}-a)d\hat{s} \text{,  }f\in V \text{ }a\in (0,A)\]
	Then, $P(a)$ is well-defined, since $X(t), t>0$ is uniformly strongly continuous on  $L^2(\Omega)$, and $\beta$ and $f$ are continuous. Equation $(\ref{semi1})$ may now be written as an abstract linear Volterra integral equation in $V$ :
	\begin{equation}
		B_{\phi}(t)=\int\limits_{0}^{t}P(a)B_{\phi}(t-a)da+C_{\phi}(t)\label{vi}
	\end{equation}
	where $B_{\phi}\in W$ and $B_{\phi}(t)(s)=b_{\phi}(t,s).$\\
	The Volterra equation (\ref{vi}) has a unique solution ( see (\cite{cv})). \\
	We thus define the family of linear operators $\textbf{U}(t)\text{ }t\geq 0$ in $K$ by the following formula:
	\begin{align*}
		(\textbf{U}(t)\phi)(x,a,s)=\left\lbrace
		\begin{array}{ll}
			X(t)\phi(x,a-t, s-t)\hbox{ if } 0<t<a\hbox{ and }  0<t<s,\\
			X(a)b_{\phi}(x,s-a,t-a) \hbox{ if }0<a<t \hbox{ and } 0<a<s,\\
			0 \hbox{ otherwise }.
		\end{array},
		\right.
	\end{align*}
	The operator $\textbf{U}$ defines a strongly continuous semigroup on $K$ (see (\cite{b9,o}).\\
	We also introduce the input space $\mathcal{U}=K$ and the control operator $\mathcal{B}\in \mathcal{L}(\mathcal{U},K)$ defined by 
	\begin{equation} 
		\mathcal{B}u=mu \text{ } (u\in \mathcal{U}).\label{op1}
	\end{equation}
	
	With above notation, we rewrite the system (\ref{2}) by:
	\begin{equation}
		\left\lbrace
		\begin{array}{ll}
			\dot{y}=\mathcal{A}y+\mathcal{B}u(t)\\
			y(0)=y_{0}.
		\end{array}\right.,
	\end{equation}
	It is already established that, the null controllability in time $T$ of $(\mathcal{A},\mathcal{B})$ is equivalent to the final-state observability in time $T$ of the pair $(\mathcal{A}^{*},\mathcal{B}^{*})$, where $\mathcal{A}^*$ and $\mathcal{B}^*$ are the adjoint operators of $\mathcal{A}$ and $\mathcal{B}$, respectively (see, for instance, \cite[\hbox{ Section } 11.2]{b8}. For that it is important to determine the adjoint of the operator $\mathcal{A}$. We obtain by proceeding as in \cite{dy} the following result:		
	\begin{proposition}
		The operator $(\mathcal{A}^*,D(\mathcal{A}^*))$ in $K$ is defined by 
		\begin{align*}
			D(\mathcal{A}^*):= \begin{cases}
				\varphi\in K,\quad \varphi\in L^2((0,A)\times (0,S);H^{2}(\Omega)),\cr
				\dfrac{\partial y}{\partial a}+\dfrac{\partial y}{\partial s}+L\varphi-\left(\mu_1(a)+\mu_2(s)\right)\varphi \in K,\cr
				\dfrac{\partial\varphi}{\partial\nu}=0,\cr
				\varphi(\cdot,A,\cdot)=0,\quad \varphi(\cdot,S,\cdot)=0,
			\end{cases}
		\end{align*}
		
		and we have, for every $\varphi\in D(\mathcal{A}^*)$,
		
		\begin{align*}
			\mathcal{A}^*\varphi=\Delta \varphi+\dfrac{\partial \varphi}{\partial a}+\dfrac{\partial \varphi}{\partial s}-\left(\mu_1(a)+\mu_2(s)\right)\varphi+\int\limits_{0}^{S}\beta(a,s,\hat{s})\varphi(x,0,\hat{s})ds,
		\end{align*}
	\end{proposition}
	\begin{proof}
		See \cite{dy} (the same technique allows to prove the Proposition).
	\end{proof}
	According to the Proposition 2.1 the adjoint system of $(\ref{2})$ is given by: 
	\begin{equation}
		\left\lbrace\begin{array}{ll}
			\dfrac{\partial q}{\partial t}-\dfrac{\partial q}{\partial a}-\dfrac{\partial q}{\partial s}-\Delta q+(\mu_1(a)+\mu_2(s))q=\displaystyle\int\limits_{0}^{S}\beta(a,s,\hat{s})q(x,0,\hat{s},t)d\hat{s}&\hbox{ in }Q ,\\ 
			\dfrac{\partial q}{\partial\nu}=0&\hbox{ on }\Sigma,\\ 
			q\left( x,A,s,t\right) =0&\hbox{ in } Q_{S,T} \\
			q(x,a,S,t)=0& \hbox{ in } Q_{S,T}\\
			q\left(x,a,s,0\right)=q_0(x,a,s)&
			\hbox{ in }Q_{A,S};
		\end{array}\right.
		\label{112}
	\end{equation}
	Let $Q'$ be the domain $Q$ without the space variable.\\
	We split the domain $Q'$ as follow:
	\[A_1=\{(a,s,t)\in Q' \text{ such that } 0< t< A-a \text{ and } 0< t< S-s\},\] \[A'_1=\{(a,s,t)\in Q' \text{ such that } S-s>t>A-a>0\text{ or } t>S-s>A-a>0\}\] and
	\[A'_2=\{(a,s,t)\in Q' \text{ such that } A-a>t>S-s>0\text{ or } t>A-a>S-s>0\}.\] To simplify, we denote by $A_2=A'_1\cup A'_2,$ then $Q'=A_1\cup A_2$.\\
	The figures below illustrate $Q'$:
	\begin{figure}[H]
		\begin{subfigure}{0.4\textwidth}
			\begin{overpic}[scale=0.25]{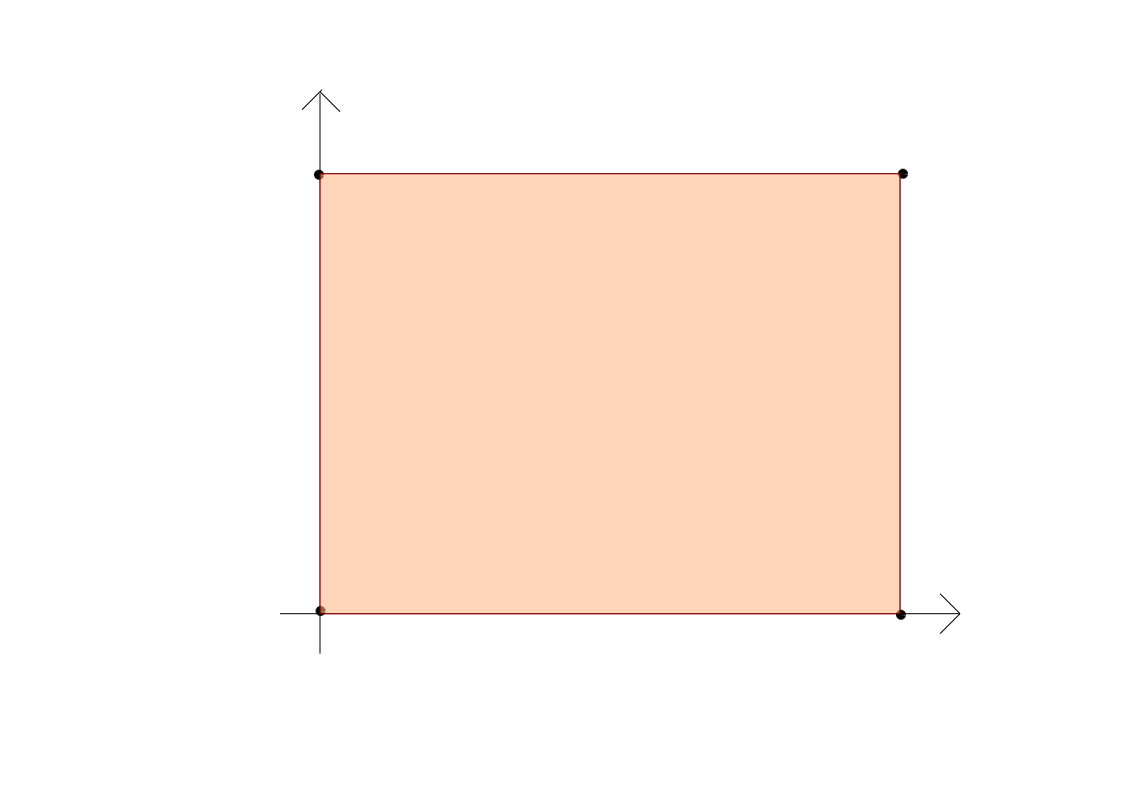}
				\put (21,53) {$S$}	
				\put (50,35) {$A_1$}
				\put (78.5,10.5) {$A$}
			\end{overpic}
			\subcaption{Here is the section $(t=0)$ of $Q'$}
		\end{subfigure}\quad
		\begin{subfigure}{0.6\textwidth}
			\begin{overpic}[scale=0.25]{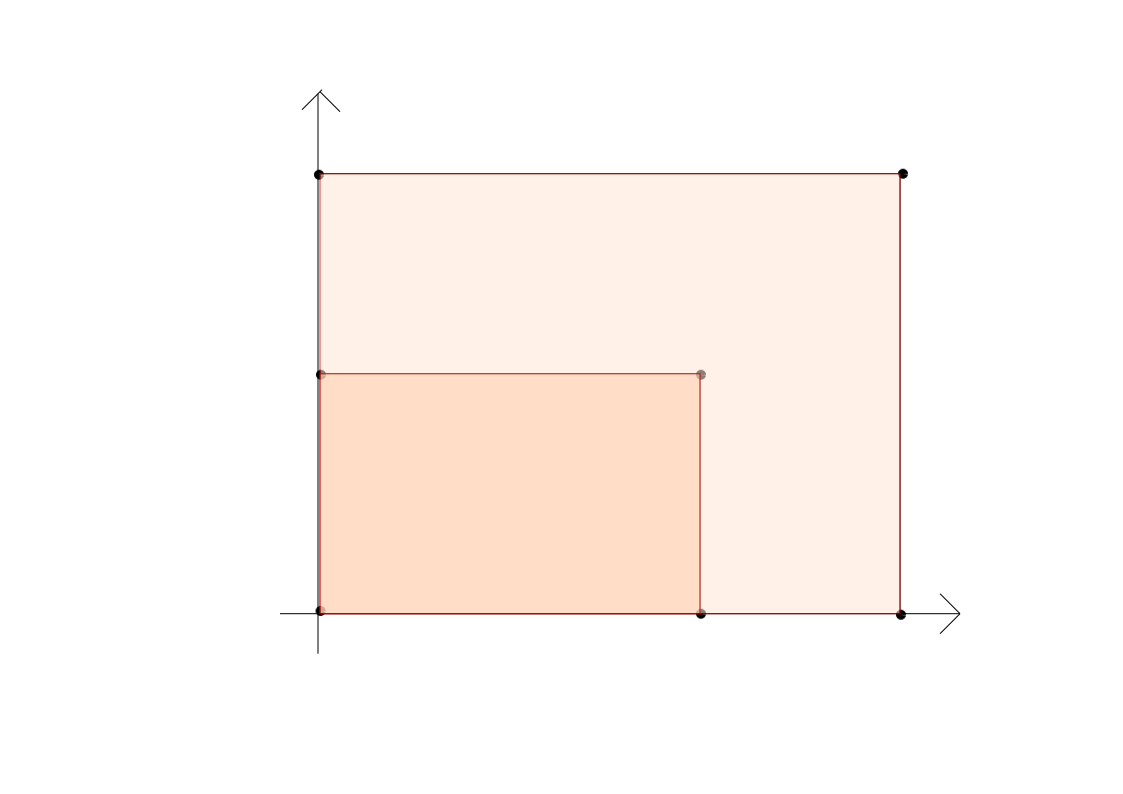}
				\put (22,53) {$S$}
				\put (42,25.5) {$A_1$}
				\put (52,42.5) {$A_2$}				
				\put (79,10.5) {$A$}
				\put (13,36.5) {$S-\alpha$}
				\put (55,10.5) {$A-\alpha$}	
			\end{overpic}
			\subcaption{Here is the section $(t=\alpha)$ of $Q'$ where $\alpha\in (0,\min\{A,S\}).$}
		\end{subfigure}
	\end{figure}
	\quad
	\begin{center}
		\begin{figure}[H]
			\begin{overpic}[scale=0.25]{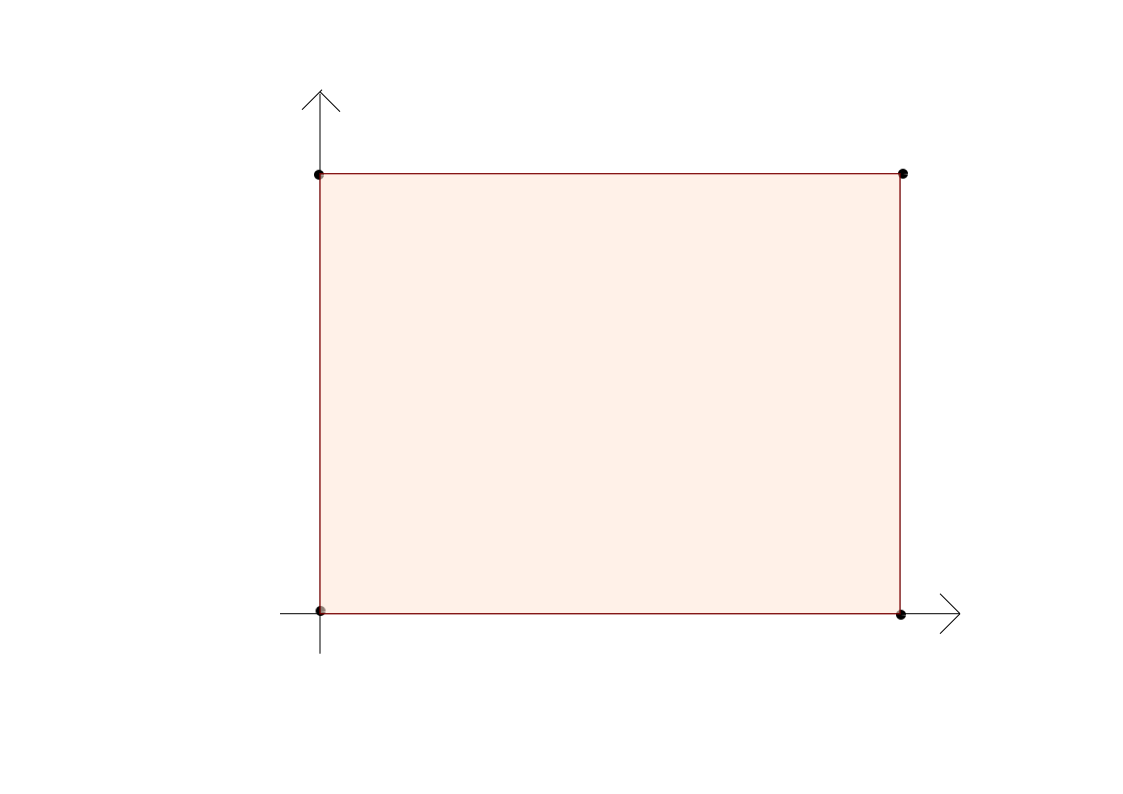}
				\put (20,53) {$S$}	
				\put (50,35) {$A_2$}
				\put (79.5,10.5) {$A$}	
			\end{overpic}
			\subcaption{Here is the section $(t=\alpha)$ where $\alpha\geq\min\{A,S\}$ of $Q'.$}
		\end{figure}
	\end{center}
	Let $L$ the operator define in
	\begin{equation}
		D(\mathcal{A}^*) \text{ by } L(a,s,x)\psi=(-\mu_1(a)-\mu_2(s)+\Delta_x)\psi.\label{33}
	\end{equation}
	The operator $L$ is a infinitesimal generator of strongly continuous semigroup in $K.$\\
	We have the following result:\\
	\begin{proposition}
		For every $q_0\in K, $ under the assumptions $(H_1)-(H_2)-(h1),$ the system $(\ref{2})$ admits a unique solution $q.$ Moreover integrating along the characteristic lines, the solution $q$ of $(\ref{112})$ is given by:
		\begin{align}
			q(t) =\left\lbrace\begin{array}{l}
				q_0(.,a+t,s+t)e^{t L}+\displaystyle\int\limits_{0}^{t}\left(e^{(t-l)L}\displaystyle\int\limits_{0}^{S}\beta(a+t-l,s+t-l,\hat{s})q(x,0,\hat{s},l)d\hat{s})\right)dl \hbox{ in } A_1, \\
				\displaystyle\int\limits_{\max\{t-S+s,t-A+a\}}^{t}\left(e^{(t-l)L}\displaystyle\int\limits_{0}^{S}\beta(a+t-l,s+t-l,\hat{s})q(x,0,\hat{s},l)d\hat{s})\right)dl \hbox{ in } A_2;\label{alp}
			\end{array}\right.
		\end{align}
	\end{proposition}
	where \[e^{tL}=\dfrac{\pi_1(a)}{\pi_1(a-t)}\dfrac{\pi_2(s)}{\pi_2(s-t)}e^{t\Delta} \hbox{ with } \pi_1(a)=\exp\left(-\int\limits_{0}^{a}\mu_1(r)dr\right) \hbox{ and   }\pi_2(s)=\exp\left(-\int\limits_{0}^{s}\mu_2(r)dr\right).\]	
	\begin{proof}	
		The proof of existence is given by the Theorem 2.1\\
		We denoted by 
		\[w(x,\lambda)=q(x,a+t-\lambda,s+t-\lambda,\lambda)\] and 
		then the function $w$ verifies the following system
		\begin{align}
			\left\lbrace\begin{array}{l}
				w'(x,\lambda)=(-\mu_1(a+t-\lambda)-\mu_2(s+t-\lambda)+\Delta) w(\lambda,x)+ f(x,\lambda)\\
				\frac{\partial w}{\partial \nu}=0\\
				w(0,x)=q(x,a+t,s+t,0),
			\end{array}\right.
		\end{align}
		with \[f(\lambda,x)=\int\limits_{0}^{S}\beta(a+t-\lambda,s+t-\lambda,\hat{s})q(x,0,\hat{s},\lambda)d\hat{s}.\]\\
		The solution of the homogeneous equation is given by
		\[w_H(\lambda,x)=Ce^{\lambda L}.\]
		We notice that $$q(x,a,s,t)=w(x,t).$$\\
		\textbf{For the taking into account of the initial condition,} we consider the domain $A_1.$
		Using the Duhamel formula, we obtain 
		\[q(x,a,s,t)= w(x,t)=e^{tL}w(x,0)+\]\[\int\limits_{0}^{t}\left(e^{(t-\alpha)L}\int\limits_{0}^{S}\beta(a+t-\alpha,s+t-\alpha,\hat{s})q(x,0,\hat{s},\alpha)d\hat{s}\right)d\alpha.\]
		As
		\[w(0,x)=q(x,a+t,s+t,0)=q_0(x,a+t,s+t),\]
		then
		\[q(x,a,s,t)=q_{0}(x,a+t,s+t)e^{t L}+\int\limits_{0}^{t}\left(e^{(t-\alpha)L}\int\limits_{0}^{S}\beta(a+t-\alpha,s+t-\alpha,\hat{s})q(x,0,\hat{s},\alpha)d\hat{s}\right)d\alpha\]	
		in $A_1.$\\
		\textbf{Taking into account the boundary condition in $a.$}\\
		For the boundary condition in $\{a=A\}$ we use the set
		\[A'_1=\{(a,s,t)\in Q' \text{ such that } S-s>t>A-a> 0\text{ or } t>S-s>A-a> 0\},\]
		then using the Duhalmel formula (boundary condition in age is $q(x,A,s,t)$) we obtain 
		\begin{equation}
			q(x,a,s,t)=e^{(A-a)L}w(t-(A-a),x)+
			\int\limits_{t-A+a}^{t}\left(e^{(t-l)L}\int\limits_{0}^{S}\beta(a+t-l,s+t-l,\hat{s})q(x,0,\hat{s},l)d\hat{s})\right)dl. \label{souff}
		\end{equation}
		But $$w(t-(A-a),x)=q(x,A,s-a+A,a+t-A)=0,$$ then 
		\[q(x,a,s,t)=\int\limits_{t-A+a}^{t}\left(e^{(t-l)L}\int\limits_{0}^{S}\beta(a+t-l,s+t-l,\hat{s})q(x,0,\hat{s},l)d\hat{s})\right)dl .\]	
		\textbf{Taking into account of the boundary condition in $s.$}\\
		For the boundary condition in $\{s=S\}$ we use the set 
		\[A'_2=\{(a,s,t)\in Q' \text{ such that } A-a>t>S-s> 0\text{ or } t>A-a>S-s> 0\},\]	then
		using the Duhamel formula, we obtain:
		\[q(x,a,s,t)=e^{(S-s)L}w(t-(S-s),x)+\int\limits_{t+s-S}^{t}\left(e^{(t-l)L}\int\limits_{0}^{S}\beta(a+t-l,s+t-l,\hat{s})q(x,0,\hat{s},l)d\hat{s})\right)dl .\]
		As before $w(t+s-S,x)=0$ (boundary condition in size), then 
		\begin{equation}
			q(x,a,s,t)=\int\limits_{t+s-S}^{t}\left(e^{(t-l)L}\int\limits_{0}^{S}\beta(a+t-l,s+t-l,\hat{s})q(x,0,\hat{s},l)d\hat{s})\right)dl \hbox{ in }A'_2. \label{souf}
		\end{equation} 
	\end{proof}
	In the rest of the paper we will adopt the following representation of the solution
	\begin{align}
		q(t) = 
		\left\lbrace\begin{array}{l}
			q_0(.,a+t,s+t)e^{t L}+\displaystyle\int\limits_{0}^{t}\left(e^{(t-l)L}\int\limits_{0}^{S}\beta(a+t-l,s+t-l,\hat{s})q(x,0,\hat{s},l)d\hat{s})\right)dl \text{ in } A_1, \\
			\displaystyle\int\limits_{\max\{t-A+a,  t-S+s\}}^{t}\left(e^{(t-l)L}\int\limits_{0}^{S}\beta(a+t-l,s+t-l,\hat{s})q(x,0,\hat{s},l)d\hat{s})\right)dl \text{ in } A_2.\label{alp}
		\end{array}\right.
	\end{align}
	Indeed, we have:
	\[(a,s,t)\in A'_1\Leftrightarrow 0<t+s-S< t+a-A<t\text{ or } t+s-S<0<t+a-A,\] and
	\[(a,s,t)\in A'_2\Leftrightarrow 0<t+a-A<t+s-S<t  \text{ or } t+a-A<0<t+s-S,\] So we notice that
	\[\max\{t-A+a,  t-S+s\}=t-A+a \quad in\quad A^{'}_{1}.\] 
	and 
	\[\max\{t-A+a,  t-S+s\}=t-S+s \quad in\quad A^{'}_{2}.\] Then we obtain:  
	\begin{equation}
		q(x,a,s,t)=\displaystyle\int\limits_{\max\{t-A+a,  t-S+s\}}^{t}\left(e^{(t-l)L}\int\limits_{0}^{S}\beta(a+t-l,s+t-l,\hat{s})q(x,0,\hat{s},l)d\hat{s})\right)dl \text{ in } A_2
	\end{equation}
	\section{An Observability Inequality}
	As mentioned above, the null-controllability of a pair $(\mathcal{A},\mathcal{B})$ is equivalent to the final state observability of the pair $(\mathcal{A}^*,\mathcal{B}^*),$ see \cite{b8}. Recall that the final-state observability of $(\mathcal{A}^*,\mathcal{B}^*)$ is defined as:
	\begin{definition}
		\cite[\hbox{ Definition }6.1.1] {b8}\\
		The pair $(\mathcal{A}^*,\mathcal{B}^*)$ is final observable in time $T$ if there exists a $K_T>0$ such that 
		\begin{equation}
			\displaystyle\int\limits_{0}^{T}\|\mathcal{B}^*\textbf{U}_{t}^{*}q_{0}\|^2\geq K_{T}^{2}\|\textbf{U}_{T}^{*}q_0\|^2 \text{ } (q_{0}\in D(\mathcal{A}^*)).
		\end{equation}
	\end{definition}
	\subsection{Proof of the Theorem 1.1}
	We consider the following adjoint system of $(\ref{2})$ given by:
	\begin{equation}
		\left\lbrace
		\begin{array}{ll}
			\dfrac{\partial q}{\partial t}-\dfrac{\partial q}{\partial a}-\dfrac{\partial q}{\partial s}-\Delta q+(\mu_1(a)+\mu_2(s))q=\int\limits_{0}^{S}\beta(a,s,\hat{s})q(x,0,\hat{s},t)d\hat{s}&\hbox{ in }Q ,\\ 
			\dfrac{\partial q}{\partial\nu}=0&\hbox{ on }\Sigma,\\ 
			q\left( x,A,s,t\right) =0&\hbox{ in } Q_{S,T} \\
			q(x,a,S,t)=0&\hbox{ in } Q_{S,T}\\
			q\left(x,a,s,0\right)=q_0(x,a,s)&\hbox{ in }Q_{A,S}.
		\end{array}\right.\label{3}
	\end{equation} 
	We recall that $$T_1=\max\{a_1+S-s_2,s_1\}\hbox{ and }T_0=\max\{S-s_2,s_1\}.$$
	In view of \cite[\hbox{ Theorem } 11.2.1]{b8} ,
	the result of the Theorem 1.1 is then reduced to the following theorem which will be proved later.
	\begin{theorem}
		Under the assumption of the Theorem 1.1, for every $q_0\in D(\mathcal{A}^*)$
		the pair $(\mathcal{A}^{*},\mathcal{B}^{*})$ is final-state observable for every $T>A-a_2+T_1+T_0.$\\
		In other words, for every $T>A-a_2+T_1+T_0$ there exist $K_T>0$ such that the solution $q$ of (\ref{3}) satisfies \begin{equation}\displaystyle\int\limits_{0}^{S}\int\limits_{0}^{A}\int_{\Omega}q^2(x,a,s,T)dxdads\leq K_T\displaystyle\int\limits_{0}^{T}\int\limits_{s_1}^{s_2}\int\limits_{a_1}^{a_2}\int_{\omega}q^2(x,a,s,t)dxdadsdt.\end{equation}
	\end{theorem}
	For the proof we proceed as in \cite{abst}.
	The principle is based on the estimation of the non-local term $\displaystyle\int\limits_{0}^{S}\beta(a,s,\hat{s})q(x,0,\hat{s},t)d\hat{s}$. Hence the following Proposition:
	\begin{proposition}
		Let us assume the assumption $(H_1)-(H_2)$ and let  $$a_1<\hat{a}\hbox{, }T_0<\min\{\hat{a}-a_1,a_2-a_1\}\hbox{ and }T_1<\eta<T.$$ Then there exists a constant $C>0$ such that for every $q_{0}\in K,$ the solution $q$ of the system $(\ref{3})$ verifies the following inequality:
		\begin{equation}
			\displaystyle\int\limits_{\eta}^{T}\int\limits_{0}^{S}\int_{\Omega}q^2(x,0,s,t)dxdsdt\leq C\displaystyle\int\limits_{0}^{T}\int\limits_{s_1}^{s_2} \int\limits_{a_1}^{a_2}\int_{\omega}q^2(x,a,s,t)dxdadsdt \label{Bq}.
		\end{equation}
	\end{proposition}
	For the proof of the Proposition 3.1, we first recall the following observability inequality for parabolic equation (see, for instance, Imanuvilov and Fursikov \cite{b10}):
	\begin{proposition}
		Let $T>0,$ $t_0$ and $t_1$ such that $0<t_0<t_1<T.$ Then for every $w_0\in L^2(\Omega),$ the solution $w$ of of the initial and boundary problem \
		\begin{equation}
			\left\lbrace
			\begin{array}{l}
				\dfrac{\partial w(x,\lambda)}{\partial \lambda}-\Delta w(x,\lambda)=0\text{ in } \Omega\times(t_0,T)\\
				\dfrac{\partial w}{\partial \nu}=0 \text{ on } \partial\Omega\times(t_0,T)	\\
				w(x,t_0)=w_0(x)\text{ in }\Omega
			\end{array}\right.,
		\end{equation}
		satisfies the estimate
		\[\displaystyle\int_{\Omega}w^2(T,x)dx\leq\int_{\Omega}w^2(x,t_1)dx\leq c_1e^{\dfrac{c_2}{t_1-t_0}}\int\limits_{t_0}^{t_1}\int_{\omega}w^2(x,\lambda)dxd\lambda,\]
		where the constant $c_1$ and $c_2$ depend on $T$ and $\Omega.$\\
	\end{proposition}
	\begin{proof}{Proof of the Proposition 3.1}\\
		We recall that $T_0=\max\{s_1,S-s_2\}$ and $T_1=\max\{a_1+S-s_2,s_1\}.$ We have two scenarios:\\
		\textbf{first scenario} \[	T_1<a_1+T_0 \hbox{ iff } T_0=s_1\hbox{ and }a_1>0;\]
		\textbf{second scenario}
		\[T_1=a_1+T_0 \hbox{ if } T_0=S-s_2\hbox{ and }a_1\geq 0.\]
		
		For the first scenario, the Proposition 3.7 of \cite{abst} gives the result. We explain it in what follows.\\
		Let $$\mathbf{A}\psi=\partial_{s}\psi-\Delta\psi+\mu_2(s)\psi$$ the operator define in \[D(\mathbf{A})=\{\varphi/ \mathbf{A}\varphi\in L^{2}(\Omega\times (0,S)), \quad \varphi(x,0)=0\quad  \dfrac{\partial\varphi}{\partial\nu}|_{\partial\Omega}=0 \},\], \[\mathbf{B}=\mathbf{1}_{(s_1,s_2)}\mathbf{1}_{\omega} \hbox{ and } \quad \mathcal{B}=\mathbf{1}_{(a_1,a_2)}\mathbf{B}.\]\\
		The operator $\mathcal{A}$ can be rewritten by: \[\mathcal{A}\varphi=-\partial_{a}\varphi-\mathbf{A}\varphi-\mu_1(a)\varphi\hbox{ for } \varphi\in D(\mathcal{A}),\]
		and the adjoint $\mathcal{A}^*$ of the operator $\mathcal{A}$  is defined by:
		\[\mathcal{A}^*\varphi=\partial_{a}\varphi-\mathbf{A}^*\varphi-\mu_1(a)\varphi +\int\limits_{0}^{S}\beta(a,s,\hat{s})\varphi(x,0,\hat{s})d\hat{s}\hbox{ for }  \varphi\in D(\mathcal{A^*})\]
		We can prove that the operator $(\mathbf{A}^*,\mathbf{B}^*)$ is final state observable for every $T>T_0$ (see for instance \cite{dy} and \cite[\hbox{ Proposition 4 }]{yac1}).\\
		Using the operators thus defined, we have the result of the Proposition 3.1 with $$a_1+\max\{s_1,S-s_2\}<\eta<T$$ by applying the result of Proposition 3.7 of \cite{abst}.\\
		Let's establish the result in the case where $a_1>0\hbox{ and }T_0=s_1$ that mean $T_1<a_1+T_0.$\\
		For $a\in (0,\hat{a})$ (assumption $H_3$) we have $\beta(a,s,\hat{s})=0$, therefore the system (\ref{3}) is can be written by
		\begin{equation}
			\left\lbrace \begin{array}{l}
				\dfrac{\partial q}{\partial t}-\dfrac{\partial q}{\partial a}-
				\dfrac{\partial q}{\partial s}-\Delta q+(\mu_1(a)+\mu_2(s)) q=0\text{ in }\Omega\times(0,\hat{a})\times(0,S)\times (0,T),\\ 
				q(x,a,s,0)=q_0(x,a,s) \text{ in }\Omega\times(0,\hat{a})\times(0,S).
			\end{array}
			\right.
			\label{ad11}
		\end{equation}	
		We denote by \[\tilde{q}(x,a,s,t)=q(x,a,s,t)\exp\left(-\int\limits_{0}^{a}\mu_1(\alpha)d\alpha-\int\limits_{0}^{s}\mu_2(r)dr\right).\] then $\tilde{q}$ satisfies
		\begin{align}
			\dfrac{\partial \tilde{q}}{\partial t}-\dfrac{\partial \tilde{q}}{\partial a}-\dfrac{\partial \tilde{q}}{\partial s}-\Delta\tilde{q}=0\text{ in }\Omega\times(0,\hat{a})\times(0,S)\times (0,T). \label{ad1}
		\end{align}
		Let $S^*<S;$ (the real $S^*$ verifying $q(x,0,s,t)=0\hbox{ in }\Omega\times(S^*,S)\times(\eta,T)$ to be explained later) proving the inequality (\ref{Bq}) lead also to show that,
		there exits a constant $C>0$ such that the solution $\tilde{q}$ of (\ref{ad11}) satisfies 
		\begin{equation}
			\int\limits_{\eta}^{T}\int\limits_{0}^{S^*}\int\limits_{\Omega}q^2(x,0,s,t)dxdsdt\leq C\int\limits_{0}^{T}\int\limits_{s_1}^{s_2}\int\limits_{a_1}^{a_2}\int_{\omega}\tilde{q}^2(x,a,s,t)dxdadsdt. \label{cci}
		\end{equation}
		Indeed, we have 
		\[
		\int\limits_{\eta}^{T}\int\limits_{0}^{S^*}\int_{\Omega}q^2(x,0,s,t)dxdsdt\leq e^{2\int\limits_{0}^{\hat{a}}\mu_1(r)dr+2\int\limits_{0}^{S^*}\mu_2(r)dr}\int\limits_{\eta}^{T}\int\limits_{0}^{S}\int_{\Omega}\tilde{q}^2(x,0,s,t)dxdsdt\]
		and
		\[\int\limits_{0}^{T}\int\limits_{s_1}^{s_2}\int\limits_{a_1}^{a_2}\int_{\omega}\tilde{q}^2(x,a,s,t)dxdadsdt\leq C(e^{2\|\mu_1\|_{L^1(0,\hat{a})}+2\|\mu_2\|_{L^1(0,S^*)}})\int\limits_{0}^{T}\int\limits_{s_1}^{s_2}\int\limits_{a_1}^{a_2}\int_{\omega}q^2(x,a,s,t)dxdadsdt.
		\]
		We consider the following characteristics trajectory $\gamma(\lambda)=(t-\lambda,t+s-\lambda,\lambda).$ If $t-\lambda=0$ the backward characteristics starting from $(0,s,t).$ If $T<T_1,$ we can not have information about all the characteristics (see Figure 5 6). So we choose $T>T_1.$\\
		Without loss the generality, let us assume here $\eta< T\leq \min\{\hat{a},a_2\}.$\\
		The proof will be in two steps:\\
		\textbf{Step 1 : In this step, we show that there exists $\delta>0$ such that the non local term $$q(x,0,s,t)=0\text{ for all }s\in (s_2-a_1,S)\text{ and }t\in(a_1+S-s_2,T).$$}\\
		First, we suppose $s_2>a_1.$
		According to the assumptions $T_0<\min\{\hat{a},a_2\}-a_1$, 
		it's easy to proof (using the representation of $q$ by the semi group method) that $q(x,0,s,t)=0$ for all $$s\in (s_2-a_1,S)\hbox{ and } t\in (a_1+S-s_2, T)$$ (see Figure 4).\\
		Indeed, for every $$(s,t)\in (s_2-a_1,S)\times (a_1+S-s_2,T)\hbox{ and }T_0<\min\{a_2,\hat{a}\}-a_1$$ we have, $$t-A<t-S+s\hbox{ and }S-s<S-s_2+a_1<t$$ then,
		\[q(x,0,s,t)=\int\limits_{t-S+s}^{t}\left(e^{(t-l)L}\int\limits_{0}^{S}\beta(t-l,s+t-l,\hat{s})q(x,0,\hat{s},l)d\hat{s})\right)dl.\]
		Moreover, we have $$0<t-l<S-s\hbox{ and as }s\in (s_2-a_1,S)$$ then $$0<t-l<S-s<S-s_2+a_1<\min\{a_2,\hat{a}\}.$$ The fertility $\beta$ being assumed to be zero on $(0,\min\{a_2,\hat{a}\}),$ then $q(x,0,s,t)=0.$\\
		Likewise, if $s_2<a_1,$ the non local term verifies $$q(x,0,s,t)=0\hbox{ for all }s\in (0,S)\hbox{ and } t\in (S-s_2+a_1, T).$$\\
		\textbf{Step 2: Estimation of the non local term $$q(x,0,s,t)\hbox{ for }(s,t)\in (0,s_2-a_1-\delta)\times (\eta,T)\text{ with }T_1<T$$ where $\delta>0$ and $\max\{s_1,a_1\}<\eta<T$}.\\
		Notice that the case $s_2<a_1$ is irrelevant because in this case \[q(x,0,s,t)=0\hbox{ a.e. in }\Omega\times(0,S)\times(S-s_2+a_1,T).\]
		Let $\delta>0$ can be as small as you want and $s\in (0,s_2-a_1-\delta)$ and $t\in (\eta,T).$ From the Figure 4 we can see that all the characteristic starting at $(0,s,t)$ goes through by the observation domain if $t>\sup\{a_1,s_1\}$ and $0\leq s<s_2-a_1.$ \\ 
		We will now focus on the estimate for $s\in (0,s_2-a_1-\delta)$ and $t\in (\eta,T).$\\
		Two situations arise: \\
		\textbf{Case 1: $\sup\{s_1,a_1\}=a_1$}\\
		We denote by:\\
		\[w(x,\lambda)=\tilde{q}(x,t-\lambda,s+t-\lambda,\lambda) \text{ ; }(x\in\Omega\hbox{, }\lambda\in (0,t))\]
		Then $w$ satisfies:
		\begin{align}
			\left\lbrace
			\begin{array}{l}
				\dfrac{\partial w(x,\lambda)}{\partial \lambda}-\Delta w(x,\lambda)=0\text{ in } ( \Omega\times (0,t))\\
				\dfrac{\partial w}{\partial \nu}=0 \text{ on } \partial\Omega\times (0,t)	\\
				w(x,0)	=\tilde{q}(x,t,s+t,0)\text{ in }\Omega
			\end{array}
			\right.,
		\end{align}
		Using the Proposition 3.2  with $0<t_0<t_1<t$ we obtain:\\
		\[\int_{\Omega}w^2(x,t)dx\leq\int_{\Omega}w^2(x,t_1)dx\leq c_1e^{\dfrac{c_2}{t_1-t_0}}\int\limits_{t_0}^{t_1}\int_{\omega}w^2(x,\lambda)dxd\lambda.\]
		That is equivalent to
		\[\int_{\Omega}\tilde{q}^2(x,0,s,t)dx\leq c_1e^{\frac{c_2}{t_1-t_0}}\int\limits_{t_0}^{t_1}\int_{\omega}\tilde{q}^2(x,t-\lambda,s+t-\lambda,\lambda)dxd\lambda=c_1e^{\frac{c_2}{t_1-t_0}}\int\limits_{t-t_1}^{t-t_0}\int_{\omega}\tilde{q}^2(x,\alpha,s+\alpha,t-\alpha)dxd\alpha.\]	
		Then for $t_0=t-a_1-\delta$ and $t_1=t-a_1,$ we obtain\\
		\[\int_{\Omega}\tilde{q}^2(x,0,s,t)dx\leq c_1e^{\frac{c_2}{\delta}}\int\limits_{a_1}^{a_1+\delta}\int_{\omega}\tilde{q}^2(x,\alpha,s+\alpha,t-\alpha)dxd\alpha.\]
		Integrating with respect $s$ over $(0,s_2-a_1-\delta)$ we get
		\[\int\limits_{0}^{s_2-a_1-\delta}\int_{\Omega}\tilde{q}^2(x,0,s,t)dxds\leq c_1e^{\frac{c_2}{\delta}}
		\int\limits_{a_1}^{a_1+\delta}\int\limits_{a_1}^{s_2}\int_{\omega}\tilde{q}^2(x,a,l,t-a)dxdlda.\]
		Finaly, integrating with respect $t$ over $(\eta,T)$, we obtain
		\[\int\limits_{\eta}^{T}\int\limits_{0}^{s_2-a_1-\delta}\int_{\Omega}\tilde{q}^2(x,0,s,t)dxdsdt\leq c_1e^{\frac{c_2}{\delta}} \int\limits_{a_1}^{a_1+\delta}\int\limits_{s_1}^{s_2}\int\limits_{\eta-a}^{T-a}\int_{\omega}\tilde{q}^2(x,a,s,t)dxdtdsda\]\[\leq C(\delta)\int\limits_{0}^{T} \int\limits_{a_1}^{a_2}\int\limits_{s_1}^{s_2}\int_{\omega}\tilde{q}^2(x,a,s,t)dxdsdadt.\]
		Then
		\begin{equation}
			\int\limits_{\eta}^{T}\int\limits_{0}^{s_2-a_1-\delta}\int_{\Omega}\tilde{q}^2(x,0,s,t)dxdsdt\leq C(\delta)\int\limits_{0}^{T} \int\limits_{s_1}^{s_2}\int\limits_{a_1}^{a_2}\int_{\omega}\tilde{q}^2(x,a,s,t)dxdsdadt.
		\end{equation}
		\textbf{Case 2: $\sup\{s_1,a_1\}=s_1$}\\ 
		In this case, we split $(0,s_2-a_1-\delta)$ in two sub intervals $(0,s_1-a_1)\cup (s_1-a_1,s_2-a_1-\delta).$\\
		For $s\in (s_1-a_1,s_2-a_1-\delta)$, we denote by:\\
		\[w(x,\lambda)=\tilde{q}(x,t-\lambda,s+t-\lambda,\lambda) \text{ ; }(x\in\Omega\text{ , }\lambda\in(0,t))\]
		Then $w$ satisfies:
		\begin{equation}
			\left\lbrace
			\begin{array}{l}
				\dfrac{\partial w(x,\lambda)}{\partial \lambda}-\Delta w(x,\lambda)=0\text{ in } \Omega\times (0,t)\\
				\dfrac{\partial w}{\partial \nu}=0 \text{ on } \partial\Omega\times (0,t)	\\
				w(x,0)	=\tilde{q}(x,t,s+t,0)\text{ in }\Omega
			\end{array}\right.,
		\end{equation}
		Using the Proposition 3.2  with $0<t_0<t_1<t$ we obtain:\\
		\[\int_{\Omega}w^2(x,t)dx\leq\int_{\Omega}w^2(x,t_1)dx\leq c_1e^{\dfrac{c_2}{t_1-t_0}}\int\limits_{t_0}^{t_1}\int_{\omega}w^2(x,\lambda)dxd\lambda.\]
		That is equivalent to
		\[\int_{\Omega}\tilde{q}^2(x,0,s,t)dx\leq c_1e^{\frac{c_2}{t_1-t_0}}\int\limits_{t_0}^{t_1}\int_{\omega}\tilde{q}^2(x,t-\lambda,s+t-\lambda,\lambda)dxd\lambda=c_1e^{\frac{c_2}{t_1-t_0}}\int\limits_{t-t_1}^{t-t_0}\int_{\omega}\tilde{q}^2(x,\alpha,s+\alpha,t-\alpha)dxd\alpha.\]	
		and we denote by $t_0=t-a_1-\delta$ and $t_1=t-a_1,$ we obtain\\
		\[\int_{\Omega}\tilde{q}^2(x,0,s,t)dx\leq C(\delta)\int\limits_{a_1}^{a_1+\delta}\int_{\omega}\tilde{q}^2(x,\alpha,s+\alpha,t-\alpha)dxd\alpha.\]
		Integrating with respect $s$ over $(s_1-a_1,s_2-a_1-\delta)$ we get
		\[\int\limits_{s_1-a_1}^{s_2-a_1-\delta}\int_{\Omega}\tilde{q}^2(x,0,s,t)dxds\leq C(\delta)
		\int\limits_{a_1}^{a_1+\delta}\int\limits_{s_1}^{s_2}\int_{\omega}\tilde{q}^2(x,a,l,t-a)dxdlda.\]
		Finally, integrating with respect $t$ over $(\eta,T)$, we obtain
		\[\int\limits_{\eta}^{T}\int\limits_{s_1-a_1}^{s_2-s_1-\delta}\int_{\Omega}\tilde{q}^2(x,0,s,t)dxdsdt\leq C(\delta) \int\limits_{a_1}^{a_1+\delta}\int\limits_{s_1}^{s_2}\int\limits_{\eta-a}^{T-a}\int_{\omega}\tilde{q}^2(x,a,s,t)dxdtdsda\]\[\leq C(\delta)\int\limits_{0}^{T} \int\limits_{a_1}^{a_2}\int\limits_{s_1}^{s_2}\int_{\omega}\tilde{q}^2(x,a,s,t)dxdsdadt.\]
		Then
		\begin{equation}
			\int\limits_{\eta}^{T}\int\limits_{s_1-a_1}^{s_2-a_1-\delta}\int_{\Omega}\tilde{q}^2(x,0,s,t)dxdsdt\leq C(\delta)\int\limits_{0}^{T} \int\limits_{s_1}^{s_2}\int\limits_{a_1}^{a_2}\int_{\omega}\tilde{q}^2(x,a,s,t)dxdsdadt.\label{rr}
		\end{equation}
		\textbf{The case $s\in (0,s_1-a_1)$} \\ 
		Here again we have two situations:
		\[s_2-s_1> s_1-a_1\hbox{ and }s_2-s_1<s_1-a_1.\]
		If $s_2-s_1<s_1-a_1,$ we split $(0,s_1-a_1)$ as the following
		\[(0,s_2-s_1)\cup(s_2-s_1,s_1-a_1).\]
		Here we will do the calculations only in the case $s_2-s_1>s_1-a_1.$\\
		We denote by:\\
		\[w(x,\lambda)=\tilde{q}(x,t-\lambda,s+t-\lambda,\lambda) \text{ ; }(x\in\Omega\hbox{, }\lambda\in (0,t))\]
		Then $w$ satisfies:
		\begin{align}
			\left\lbrace
			\begin{array}{l}
				\dfrac{\partial w(x,\lambda)}{\partial \lambda}-\Delta w(x,\lambda)=0\text{ in } ( \Omega\times (0,t))\\
				\dfrac{\partial w}{\partial \nu}=0 \text{ on } \partial\Omega\times (0,t)	\\
				w(x,0)	=\tilde{q}(x,t,s+t,0)\text{ in }\Omega
			\end{array}
			\right.,
		\end{align}
		Using the Proposition 3.2  with $0<t_0<t_1<t$ we obtain:\\
		\[\int_{\Omega}w^2(x,t)dx\leq\int_{\Omega}w^2(x,t_1)dx\leq c_1e^{\dfrac{c_2}{t_1-t_0}}\int\limits_{t_0}^{t_1}\int_{\omega}w^2(x,\lambda)dxd\lambda.\]
		That is equivalent to
		\[\int_{\Omega}\tilde{q}^2(x,0,s,t)dx\leq c_1e^{\dfrac{c_2}{t_1-t_0}}\int\limits_{t_0}^{t_1}\int_{\omega}\tilde{q}^2(x,t-\lambda,s+t-\lambda,\lambda)dxd\lambda=c_1e^{\frac{c_2}{t_1-t_0}}\int\limits_{t-t_1}^{t-t_0}\int_{\omega}\tilde{q}^2(x,\alpha,s+\alpha,t-\alpha)dxd\alpha.\]	
		Then for $t_0=t-s_1-\kappa$ where $\kappa>0$ and $t_1=t-s_1,$ we obtain\\
		\[\int_{\Omega}\tilde{q}^2(x,0,s,t)dx\leq C(\kappa)\int\limits_{s_1}^{s_1+\kappa}\int_{\omega}\tilde{q}^2(x,\alpha,s+\alpha,t-\alpha)dxd\alpha.\]
		Integrating with respect $s$ over $(0,s_1-a_1)$ we get
		\[\int\limits_{0}^{s_1-a_1}\int_{\Omega}\tilde{q}^2(x,0,s,t)dxds\leq C(\kappa)
		\int\limits_{s_1}^{s_1+\kappa}\int\limits_{l}^{l+s_1-a_1}\int_{\omega}\tilde{q}^2(x,a,l,t-a)dxdlda.\] Then
		\[\int\limits_{0}^{s_1-a_1}\int_{\Omega}\tilde{q}^2(x,0,s,t)dxds\leq C(\kappa)
		\int\limits_{s_1}^{s_1+\kappa}\int\limits_{s_1}^{\kappa+2s_1-a_1}\int_{\omega}\tilde{q}^2(x,a,l,t-a)dxdlda.\]
		Finaly, integrating with respect $t$ over $(\eta,T)$, we obtain
		\[\int\limits_{\eta}^{T}\int\limits_{0}^{s_1-a_1}\int_{\Omega}\tilde{q}^2(x,0,s,t)dxdsdt\leq C(\kappa) \int\limits_{s_1}^{s_1+\kappa}\int\limits_{s_1}^{2\kappa+s_1-a_1}\int\limits_{\eta-a}^{T-a}\int_{\omega}\tilde{q}^2(x,a,s,t)dxdtdsda.\]
		We choose $\kappa$ small enough such that \[s_1+\kappa<T<\min\{a_2,\hat{a}\}\]
		Then, we get
		\begin{equation}
			\int\limits_{\eta}^{T}\int\limits_{0}^{s_1-a_1}\int_{\Omega}\tilde{q}^2(x,0,s,t)dxdsdt\leq C(\kappa)\int\limits_{0}^{T} \int\limits_{s_1}^{s_2}\int\limits_{a_1}^{a_2}\int_{\omega}\tilde{q}^2(x,a,s,t)dxdsdadt.
		\end{equation}
		Finally,
		combining $(\ref{rr})$ and the fact that $$q(x,0,s,t)=0\hbox{ for } t\in (S-s_2+a_1+\delta,+\infty)\hbox{ } s\in (s_2-a_1-\delta,S)$$, we obtain:
		\[\int\limits_{\eta}^{T}\int\limits_{0}^{S}\int_{\Omega}\tilde{q}^2(x,0,s,t)dxdsdt\leq C(\kappa)\int\limits_{0}^{T} \int\limits_{s_1}^{s_2}\int\limits_{a_1}^{a_2}\int_{\omega}\tilde{q}^2(x,a,s,t)dxdsdadt\] where $\max\{S-s_2+a_1,s_1\}<\eta<T.$
		\begin{remark}
			In all the cases, when $\delta \longrightarrow 0\hbox{ or }\kappa \longrightarrow0 \hbox{ or }\eta \longrightarrow T_1$ we have $C(\eta,\kappa,\delta)\longrightarrow +\infty.$ 
		\end{remark}
		\begin{remark}
			For the case $\eta<\min\{\hat{a},a_2\}\leq T,$ we split the interval $(\eta,T)$ in two sub-intervals $(\eta,\min\{\hat{a},a_2\})\cup  (\min\{\hat{a},a_2\},T)$.\\
			In $(\eta,\min\{\hat{a},a_2\}),$ we proceed in the same way. In $(\min\{\hat{a},a_2\},T)$ we proceed as in the proof of the Proposition 3.7 in \cite{abst}.
		\end{remark}
	\end{proof}
	\textbf{Illustration of cases where the non-local term $q(x,0,s,t)$ cannot be estimated\\Assumption $T_0<\min\{a_2,\hat{a}\}$}
	\begin{figure}[H]
		\begin{overpic}[scale=0.6]{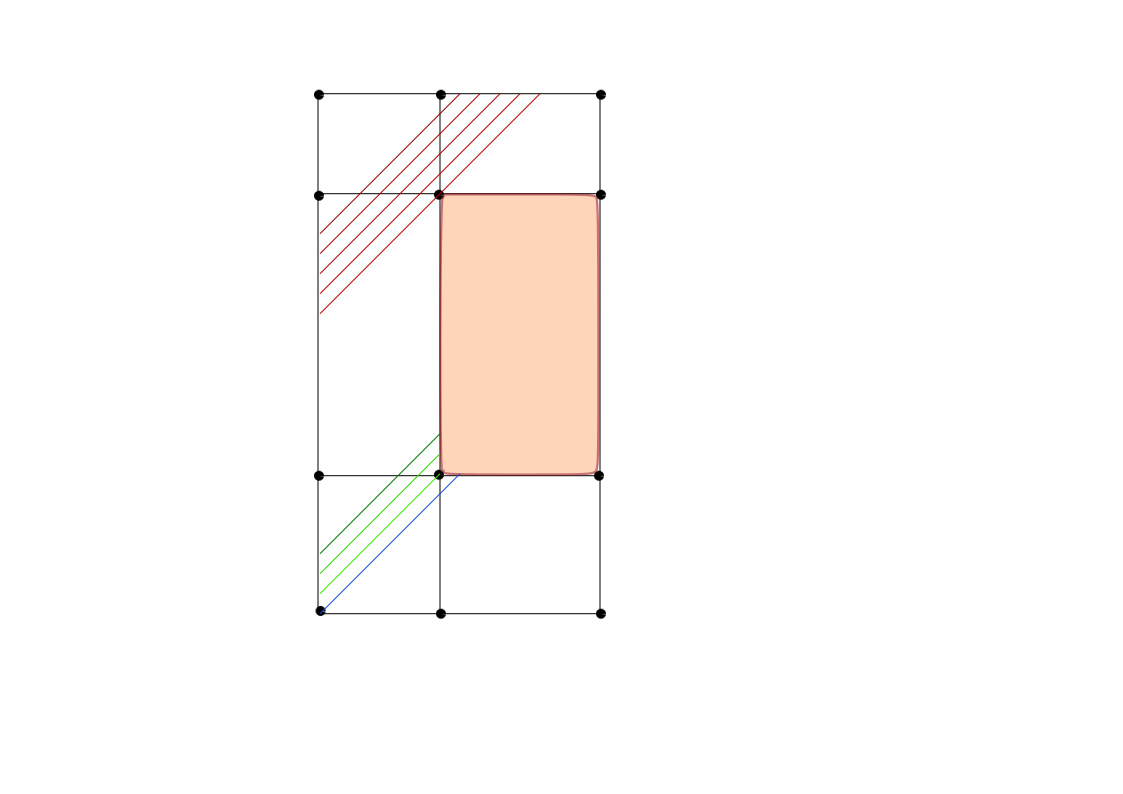}
			\put (25,62) {$S$}	
			\put (53.5,14) {$\hat{a}$}
			\put (25,28) {$s_1$}
			\put (25,53) {$s_2$}				
			\put (38.7,14) {$a_1$}
			\put (42.5,35.5) {$Q_{1}$}		
		\end{overpic}
		\caption{An illustration of the estimate of $q(x,0,s,t)$\\Here we have chosen $a_2=\hat{a}$. Since $t>T_1$ all the backward characteristics starting from $(0,s,t)$ enters the observation domain (the green and blue lines), or without the domain by the boundary $s=S$ (red line).}
	\end{figure}
	\begin{figure}[H]
		\begin{overpic}[scale=0.4]{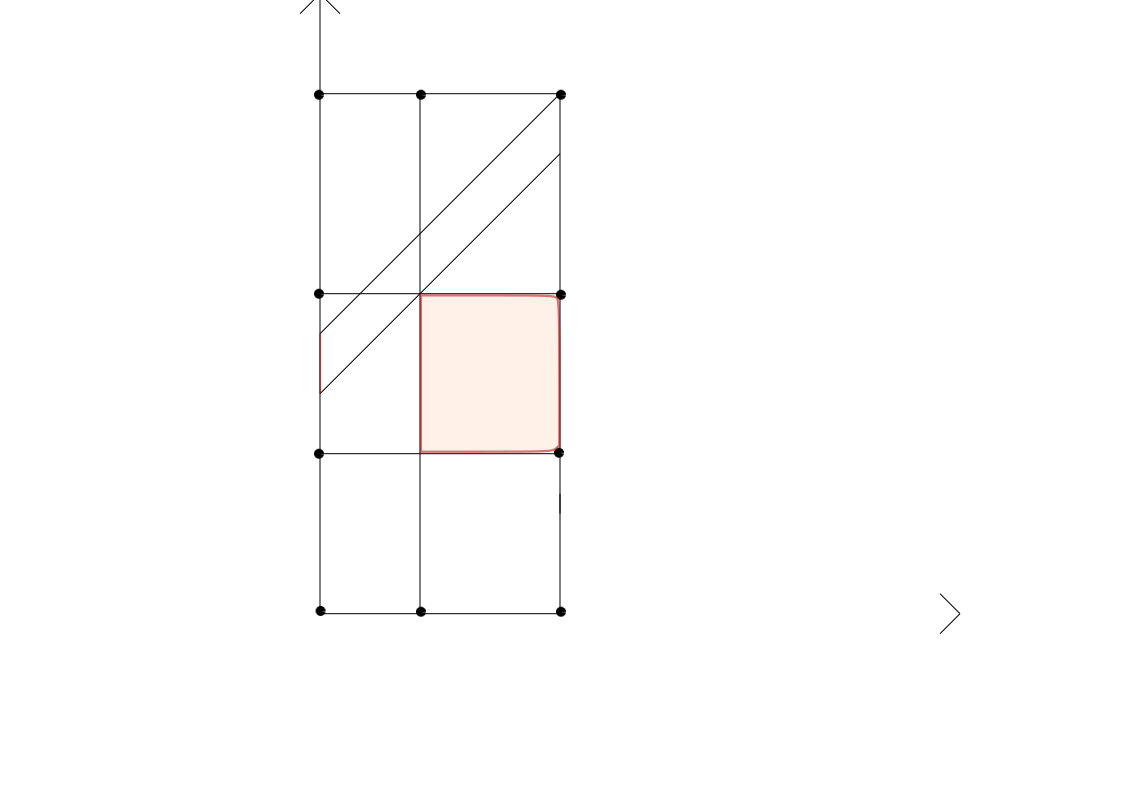}
			\put (25,62) {$S$}	
			\put (69.5,12) {$A$}
			\put (49.5,12.5) {$\hat{a}$}
			\put (25,30) {$s_1$}
			\put (25,44) {$s_2$}
			\put (25.5,35) {$c_1$}
			\put (25.5,41) {$c_2$}				
			\put (36.5,13) {$a_1$}
			\put (42.5,35.5) {$Q_{1}$}		
		\end{overpic}
		\caption{For $S-s_2>\hat{a}-a_1,$ we can not estimate $q (x,0,s,t)$ for $s\in(c_1,c_2)$ by the characteristic method. Indeed for even $t>a_1+S-s_2$ the characteristics starting at $(0,s,t)$ without the domain by the boundary $t=0 \hbox{ or enter in the region } a>\hat{a},$ without going through the observation domain.}
	\end{figure}
	\begin{figure}[H]
		\begin{overpic}[scale=0.4]{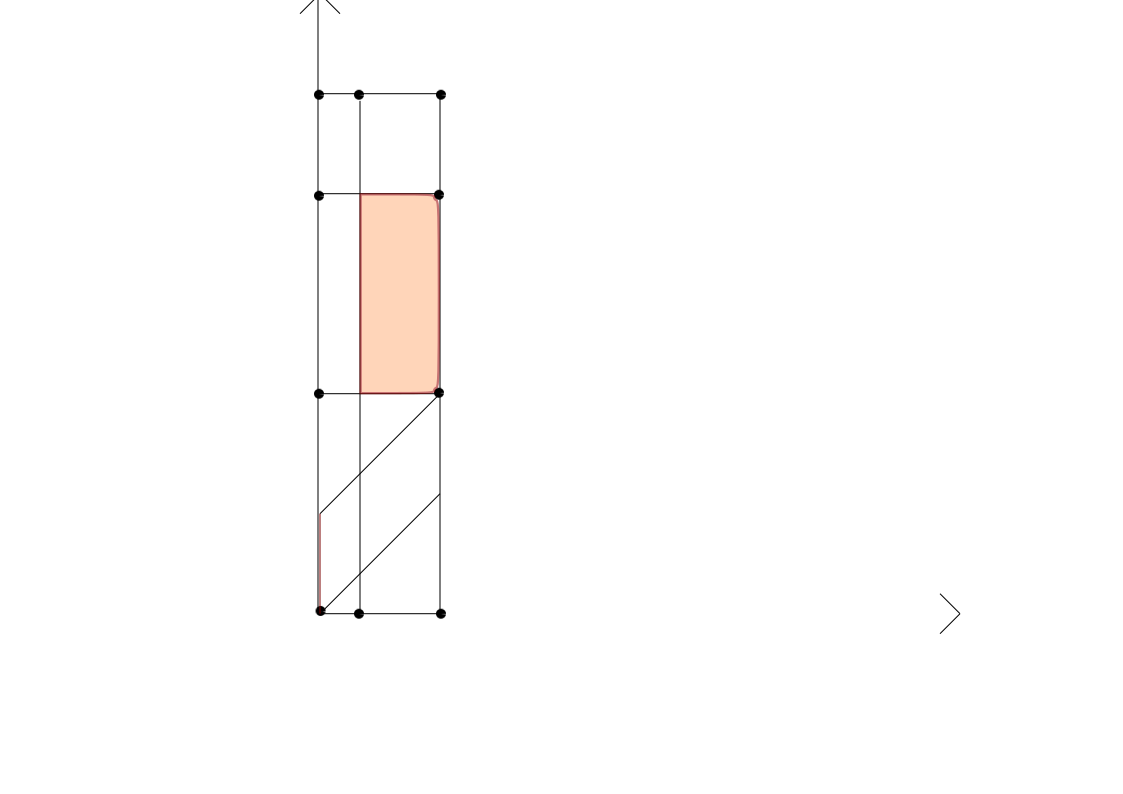}
			\put (25,62) {$S$}	
			\put (38,12.5) {$\hat{a}$}
			\put (25,35) {$s_1$}
			\put (25,53) {$s_2$}				
			\put (31.5,13) {$a_1$}
			\put (34.5,44.5) {$Q_{1}$}	
			\put (25.5,16) {$c_1$}
			\put (25.5,24) {$c_2$}
		\end{overpic}
		\caption{For the second case if $s_1>\hat{a}-a_1,$ we can not estimate $q (x,0,s,t)$ for $s\in (c_1,c_2)$ by the characteristic method. Indeed for even $t>a_1+S-s_2$ the characteristics starting at $(0,s,t)$ without the domain by the boundary $t=0 \hbox{ or enter in the region } a>\hat{a},$ without going through the observation domain.}
	\end{figure}
	\begin{proposition}
		Let us assume the hypothesis of Theorem 3.1. Let, $T>a_1+T_0.$ Then for every $q_0\in L^2(\Omega\times (0,A)\times (0,S)),$ the solution $q$ of the system $(\ref{3}),$ obeys
		\begin{equation}
			\displaystyle\int\limits_{0}^{S}\int\limits_{0}^{a_1}\int_{\Omega}q^2(x,a,s,T)dxdads\leq C_T\displaystyle\int\limits_{0}^{T}\int\limits_{s_1}^{s_2}\int\limits_{a_1}^{a_2}\int_{\omega}q^2(x,a,s,t)dxdadsdt \label{C81}
		\end{equation}
	\end{proposition}
	\begin{remark}
		The result of the previous Proposition remains true for $T>T_1.$ But does not improve the main result for the strategy used.
	\end{remark}
	Now, let consider the following cascade system:\\
	\begin{equation}
		\left\lbrace
		\begin{array}{ll}
			\dfrac{\partial q}{\partial t}-\dfrac{\partial q}{\partial a}-\dfrac{\partial q}{\partial s}-\Delta q+(\mu_1(a)+\mu_2(s))q=0&\hbox{ in }Q ,\\ 
			\dfrac{\partial q}{\partial\nu}=0&\hbox{ on }\Sigma,\\ 
			q\left( x,A,s,t\right) =0&\hbox{ in } Q_{S,T} \\
			q(x,a,S,t)=0&\hbox{ in } Q_{S,T}\\
			q\left(x,a,s,0\right)=q_0(x,a,s)&\hbox{ in }Q_{A,S}.
		\end{array}\right.\label{322}
	\end{equation} 
	We also need the following result for the Proof of the Theorem 2.1.
	\begin{proposition}
		Let us assume the assumption of Theorem 3.1. Let $T>T_0$ and $a_1<a_0<a_2-T_0.$ There exists $C_T>0$ such that the solution $q$ of the system $(\ref{322})$ verifies the following inequality
		\begin{equation}
			\displaystyle\int\limits_{0}^{S}\int\limits_{a_1}^{a_0}\int_{\Omega}q^2(x,a,s,T)dxdads\leq C_T\displaystyle\int\limits_{0}^{T}\int\limits_{s_1}^{s_2}\int\limits_{a_1}^{a_2}\int_{\omega}q^2(x,a,s,t)dxdadsdt \label{C8}
		\end{equation}
	\end{proposition}
	\begin{proof}{Proposition 3.3. and Proposition 3.4. }\\
		We denote by $$\mathbf{A}\psi=\partial_{s}\psi-\Delta\psi+\mu_2(s)\psi$$ the operator define on \[D(\mathbf{A})=\{\varphi/ \mathbf{A}\varphi\in L^{2}(\Omega\times (0,S)), \quad \varphi(x,0)=0\quad  \dfrac{\partial\varphi}{\partial\nu}|_{\partial\Omega}=0 \}\], $$\mathbf{B}=\mathbf{1}_{(s_1,s_2)}\mathbf{1}_{\omega}\hbox{ and } \mathcal{B}=\mathbf{1}_{(a_1,a_2)}\mathbf{B}.$$\\
		The operator $\mathcal{A}$ can be rewritten by: \[\mathcal{A}\varphi=-\partial_{a}\varphi-\mathbf{A}\varphi-\mu(a)\varphi\quad for \quad \varphi\in D(\mathcal{A}),\]
		and the adjoint $\mathcal{A}^*$ of the operator $\mathcal{A}$  is defined by:
		\[\mathcal{A}^*\varphi=\partial_{a}\varphi-\mathbf{A}^*\varphi-\mu(a)\varphi +\int\limits_{0}^{S}\beta(a,s,\hat{s})\varphi(x,0,\hat{s})d\hat{s}\hbox{ for } \varphi\in D(\mathcal{A^*}).\]
		We can prove that the operator $(\mathbf{A}^*,\mathbf{B}^*)$ is final state observable for every $T>T_0$ (see for instance \cite{yac1}). Finally, the results of the Proposition 3.5. of \cite{abst}, gives the inequality (\ref{C81}) and the result of the Proposition 3.6. of \cite{abst}, gives the inequality (\ref{C8}) (see Figure 7).
	\end{proof}
	\begin{figure}[H]
		\begin{subfigure}{0.5\textwidth}
			\begin{overpic}[scale=0.35]{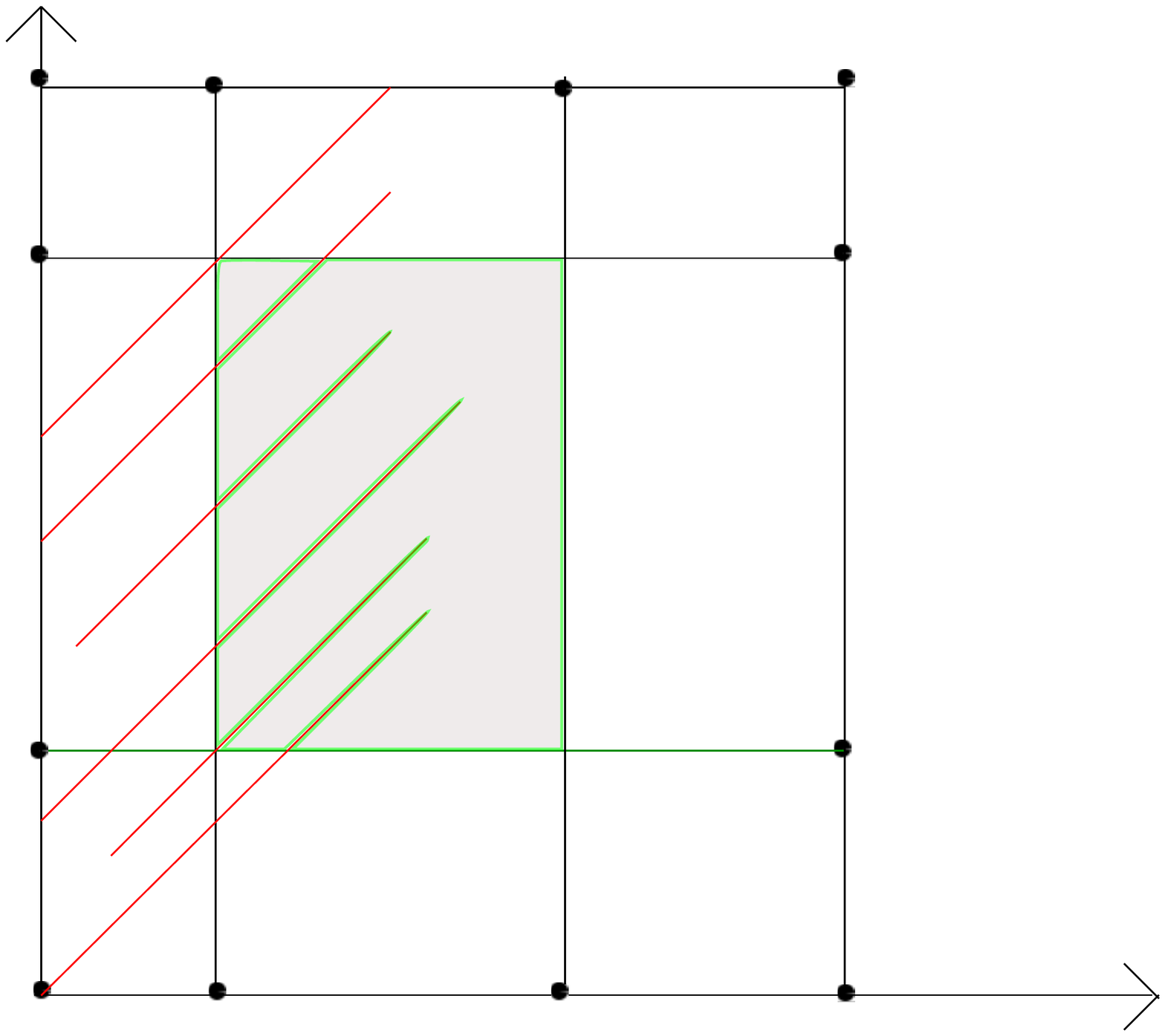}
				\put (25,60.5) {$S$}
				\put (25,52.5) {$s_2$}
				\put (40,34.5) {$Q_1$}	
				\put (25,27) {$s_1$}
				\put (36,12.5) {$a_1$}
				\put (53,12.5) {$a_2$}
				\put (67,11.5) {$A$}
				\put(28.5,24.1){\makebox[0pt]{\Huge\textcolor{black}{.}}}
				\put(32.3,22.5){\makebox[0pt]{\Huge\textcolor{black}{.}}}
				\put(30.2,32.85){\makebox[0pt]{\Huge\textcolor{black}{.}}}
				\put(28.5,43.8){\makebox[0pt]{\Huge\textcolor{black}{.}}}
				\put(28.5,38.3){\makebox[0pt]{\Huge\textcolor{black}{.}}}
			\end{overpic}
			\caption{Estimate of $q(x,a,s,T)$ on $(0,a_1)$\\
				For $T>T_1,$ the backward characteristics starting from $(a,s,T)$ with $a\in (0,a_1)$ enter the observation domain or without all the domain by the boundary $s=S.$ }
		\end{subfigure}\quad
		\begin{subfigure}{0.5\textwidth}
			\begin{overpic}[scale=0.35]{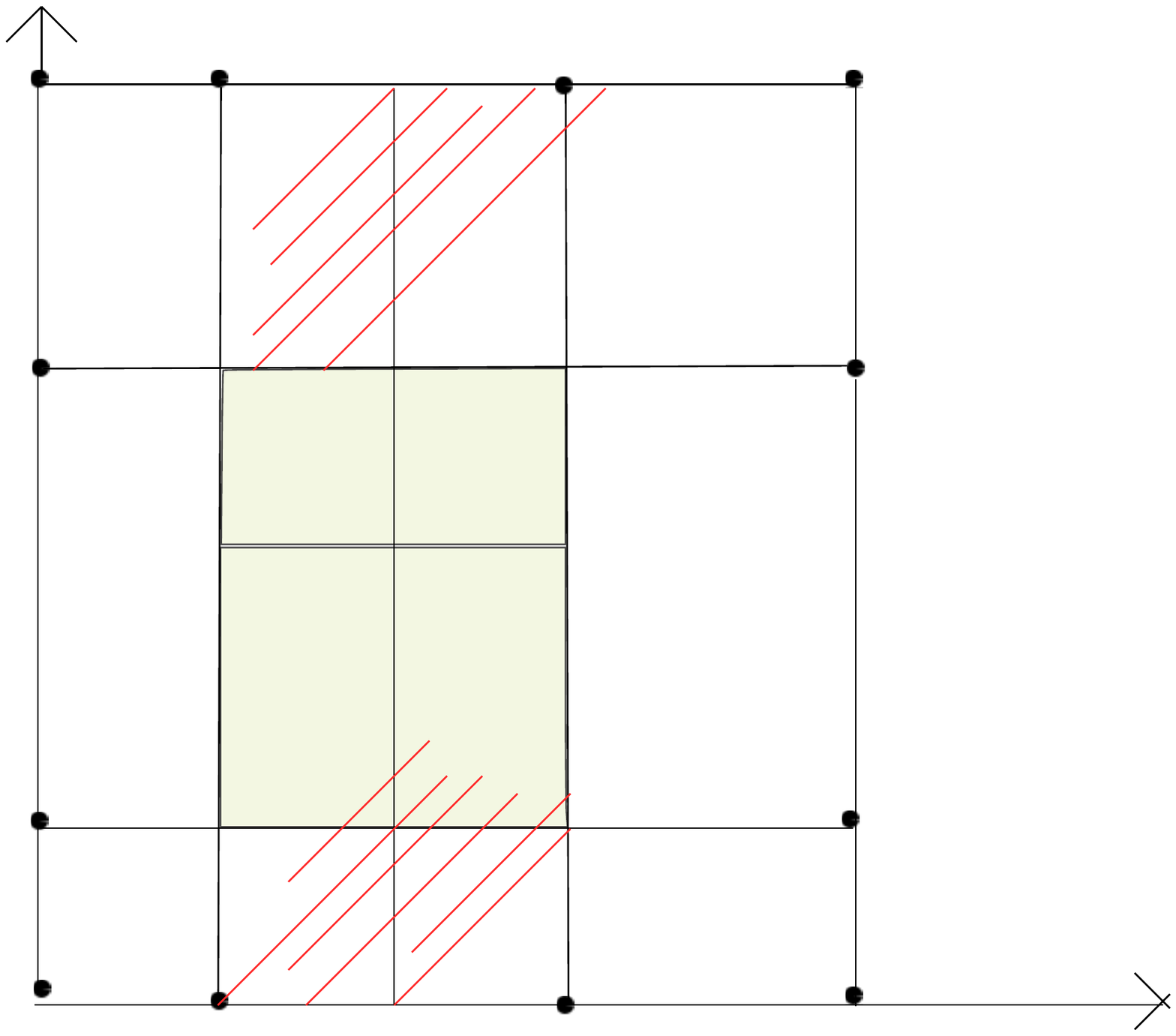}
				\put (25,60.5) {$S$}	
				\put (25,47) {$s_2$}
				\put (25,23) {$s_1$}
				\put (45,13) {$c$}
				\put (42,32.5) {$Q_1$}
				\put (36,12.5) {$a_1$}
				\put (54,12.5) {$a_2$}
				\put (67,11.5) {$A$}
				\put (42,8) {$c=a_2-s_1$}
				\put(46.2,15.2){\makebox[0pt]{\Huge\textcolor{black}{.}}}
				\put(41.2,21.5){\makebox[0pt]{\Huge\textcolor{black}{.}}}
				\put(42,15.2){\makebox[0pt]{\Huge\textcolor{black}{.}}}
				\put(41,16.9){\makebox[0pt]{\Huge\textcolor{black}{.}}}
				\put(39.3,54.3){\makebox[0pt]{\Huge\textcolor{black}{.}}}
				\put(39.9,52.3){\makebox[0pt]{\Huge\textcolor{black}{.}}}
				\put(39.5,49){\makebox[0pt]{\Huge\textcolor{black}{.}}}
				\put(39.6,47.5){\makebox[0pt]{\Huge\textcolor{black}{.}}}
				\put(43,47.5){\makebox[0pt]{\Huge\textcolor{black}{.}}}
			\end{overpic}
			\caption{Estimate of $q(x,a,s,T)$ on $(a_1,a_0)$\\For $T>T_0,$ the backward characteristics starting from $(a,s,T)$ with $a\in (a_1,a_2-T_0)$ enter the observation domain in the case where $s\in(0,s_2)$ or without all the domain by the boundary  $s=S$ if $s\in (s_2,S)$ }
		\end{subfigure}
		\caption{Illustration of the time to estimate $q(x,a,s,T)$ on $(0,a_0)$ }
	\end{figure}
	Let us now gives a preliminary results for the proof of the Theorem 3.1.\\
	Let $q=u_1+u_2$ where $u_1$ and $u_2$ verify
	\begin{equation}
		\left\lbrace
		\begin{array}{ll}
			\dfrac{\partial u_1}{\partial t}-\dfrac{\partial u_1}{\partial a}-\dfrac{\partial u_1}{\partial s}-\Delta u_1+(\mu_1(a)+\mu_2(s))u_1=0&\hbox{ in }Q_{A.S}\times (\eta,T) ,\\ 
			\dfrac{\partial u_1}{\partial\nu}=0&\hbox{ on }\partial\Omega\times (0,A)\times(0,S)\times (\eta,T),,\\ 
			u_1\left( x,A,s,t\right) =0&\hbox{ in } \Omega\times(0,S)\times (\eta,T) \\
			u_1(x,a,S,t)=0&\hbox{ in } Q_{S,T}\\
			u_1\left(x,a,s,\eta\right)=q_{\eta}&\hbox{ in }Q_{A,S}.
		\end{array}\right.\label{3220}
	\end{equation} 
	where $q_{\eta}=q(x,a,s,\eta)$ in $Q_{A.S}$ and
	\begin{equation}
		\left\lbrace
		\begin{array}{ll}
			\dfrac{\partial u_2}{\partial t}-\dfrac{\partial u_2}{\partial a}-\dfrac{\partial u_2}{\partial s}-\Delta u_2+(\mu_1(a)+\mu_2(s))u_2=\int\limits_{0}^{S}\beta(a,s,\hat{s})q(x,0,\hat{s},t)d\hat{s}&\hbox{ in }Q_{A.S}\times (\eta,T) ,\\ 
			\dfrac{\partial u_2}{\partial\nu}=0&\hbox{ on }\partial\Omega\times (0,A)\times(0,S)\times (\eta,T),\\ 
			u_2\left( x,A,s,t\right) =0&\hbox{ in } \Omega\times(0,S)\times (\eta,T) \\
			u_2(x,a,S,t)=0&\hbox{ in } Q_{S,T}\\
			u_2\left(x,a,s,\eta\right)=0&\hbox{ in }Q_{A,S}.
		\end{array}\right.\label{3221}
	\end{equation}
	where $V(x,a,s,t)=\int\limits_{0}^{S}\beta(a,s,\hat{s})q(x,0,\hat{s},t)d\hat{s}.$\\ 
	Using Duhamel’s formula we can write \[u_2(x,a,s,t)=\int\limits_{\eta}^{t}\mathbb{T}_{t-f}V(.,.,.,f)df\] where $\mathbb{T}$ is the semigroup generates by the operator $\mathcal{A}^{*}.$
	Moreover, the solution $u_2$ of the system $(\ref{3221})$ verifies the following estimates :
	\begin{proposition}
		Under the assumptions $(H_1)$, $(H_2)$ and $(h_1)$, there exist $C=e^{\frac{3}{2}T}\|\beta\|^{2}_{\infty}A$ such that the solution $u_2$ of the system $(\ref{3221})$ verifies the following estimate
		\begin{equation}
			\int\limits_{\eta}^{T}\int\limits_{s_1}^{s_2}\int\limits_{a_1}^{a_2}\int_{\omega}u_{2}^2(x,a,s,t)dxdadsdt\leq \int\limits_{\eta}^{T}\int\limits_{0}^{S}\int\limits_{0}^{A}\int_{\Omega}u_{2}^2(x,a,s,t)dxdadsdt\leq C \int\limits_{\eta}^{T}\int\limits_{0}^{S}\int_{\Omega}q^2(x,0,s,t)dxdsdt.	
		\end{equation}
	\end{proposition}
	\begin{proof}{Proof of Proposition 3.5.}
		We denote by $$u_2=\hat{u}_{2}e^{\lambda t}\hbox{ with }\lambda>0.$$ The function $u_2$ verifies \begin{equation}\dfrac{\partial \hat{u}_2}{\partial t}-\dfrac{\partial \hat{u}_2}{\partial a}-\dfrac{\partial \hat{u}_2}{\partial s}-\Delta \hat{u}_2+(\mu_1(a)+\mu_2(s)+\lambda)\hat{u}_2=e^{-\lambda t}\int\limits_{0}^{S}\beta(a,s,\hat{s})q(x,0,\hat{s},t)d\hat{s}\hbox{ in }Q_{A.S}\times (\eta,T). \label{humm}
		\end{equation}
		Multiplying $(\ref{humm})$ by $\hat{u}_2$ and integrating over $Q_{A,S}\times (\eta,T)$, we get 
		\[\int\limits_{0}^{S}\int\limits_{0}^{A}\int_{\Omega}\hat{u}_{2}^2(x,a,s,T)dxdads+\int\limits_{0}^{T}\int\limits_{0}^{S}\int_{\Omega}\hat{u}_{2}^2(x,0,s,t)dxdsdt+\int\limits_{0}^{T}\int\limits_{0}^{A}\int_{\Omega}\hat{u}_{2}^2(x,a,0,t)dxdadt\]\[+\int\limits_{\eta}^{T}\int_{Q_{A,S}}|\nabla\hat{u}_{2}|^2dxdadsdt+\int\limits_{\eta}^{T}\int_{Q_{A,S}}(\mu_1(a)+\mu_2(s)+\lambda)\hat{u}_{2}^2dxdadsdt\]\[=\int\limits_{\eta}^{T}\int_{Q_{A,S}}\left(e^{-\lambda t}\int\limits_{0}^{S}\beta(a,s,\hat{s})q(x,0,\hat{s},t)d\hat{s}\right)\hat{u}_2dxdadsdt.\]
		Using Young inequality and choosing $\lambda=\frac{3}{2},$
		we obtain
		\[ \int\limits_{\eta}^{T}\int_{Q_{A,S}}\hat{u}_{2}^2dxdadsdt\leq \|\beta\|^{2}_{\infty}A \int\limits_{\eta}^{T}\int\limits_{0}^{S}\int_{\Omega} q^2(x,0,s,t)dxdsdt.\]
		Finally, we get 
		\[ \int\limits_{\eta}^{T}\int_{Q_{A,S}}u_{2}^2dxdadsdt\leq e^{\frac{3}{2}T}\|\beta\|^{2}_{\infty}A \int\limits_{\eta}^{T}\int\limits_{0}^{S}\int_{\Omega} q^2(x,0,s,t)dxdsdt.\]
	\end{proof}
	\begin{proof}{Proof of the Theorem 3.1}
		We split the term to be estimated as the following
		\[\int\limits_{0}^{S}\int\limits_{0}^{A}\int_{\Omega}q^2(x,a,s,T)dxdads=\int\limits_{0}^{S}\int\limits_{0}^{a_1}\int_{\Omega}q^2(x,a,s,T)dxdads+\int\limits_{0}^{S}\int\limits_{a_1}^{A}\int_{\Omega}q^2(x,a,s,T)dxdads.\]
		Using the Proposition 3.3. we obtain the estimate 
		\begin{equation}\int\limits_{0}^{S}\int\limits_{0}^{a_1}\int_{\Omega}q^2(x,a,s,T)dxdads\leq C\int\limits_{0}^{A}\int\limits_{0}^{S}\int\limits_{0}^{A}\int_{\Omega}q^2(x,a,s,t)dxdadsdt.\label{coupp}\end{equation}
		We are now left with the estimate of \[\int\limits_{0}^{S}\int\limits_{a_1}^{A}\int_{\Omega}q^2(x,a,s,T)dxdads.\]
		But since $q=u_1+u_1$ we must therefore estimate 
		\[\int\limits_{0}^{S}\int\limits_{a_1}^{A}\int_{\Omega}u_{1}^{2}(x,a,s,T)dxdads+\int\limits_{0}^{S}\int\limits_{a_1}^{A}\int_{\Omega}u_{2}^{2}(x,a,s,T)dxdads.\]
		We have \[\int\limits_{0}^{S}\int\limits_{a_1}^{A}\int_{\Omega}u_{2}^{2}(x,a,s,T)dxdads\leq C(A-a_2)\int\limits_{\eta}^{T}\int\limits_{0}^{S}\int_{\Omega}q^2(x,0,s,t)dxdsdt.\]
		And then, using the Proposition 3.1., we obtain
		\begin{equation}
			\int\limits_{0}^{S}\int\limits_{a_1}^{A}\int_{\Omega}u_{2}^{2}(x,a,s,T)dxdads\leq C_1\int\limits_{0}^{T}\int\limits_{s_1}^{s_2}\int\limits_{a_1}^{a_2}\int_{\omega}q^2(x,a,s,t)dxdadsdt.\label{debay}
		\end{equation}
		See Figure for the Illustration of the Observability inequality.
		\begin{figure}[H]
			\begin{overpic}[scale=0.65]{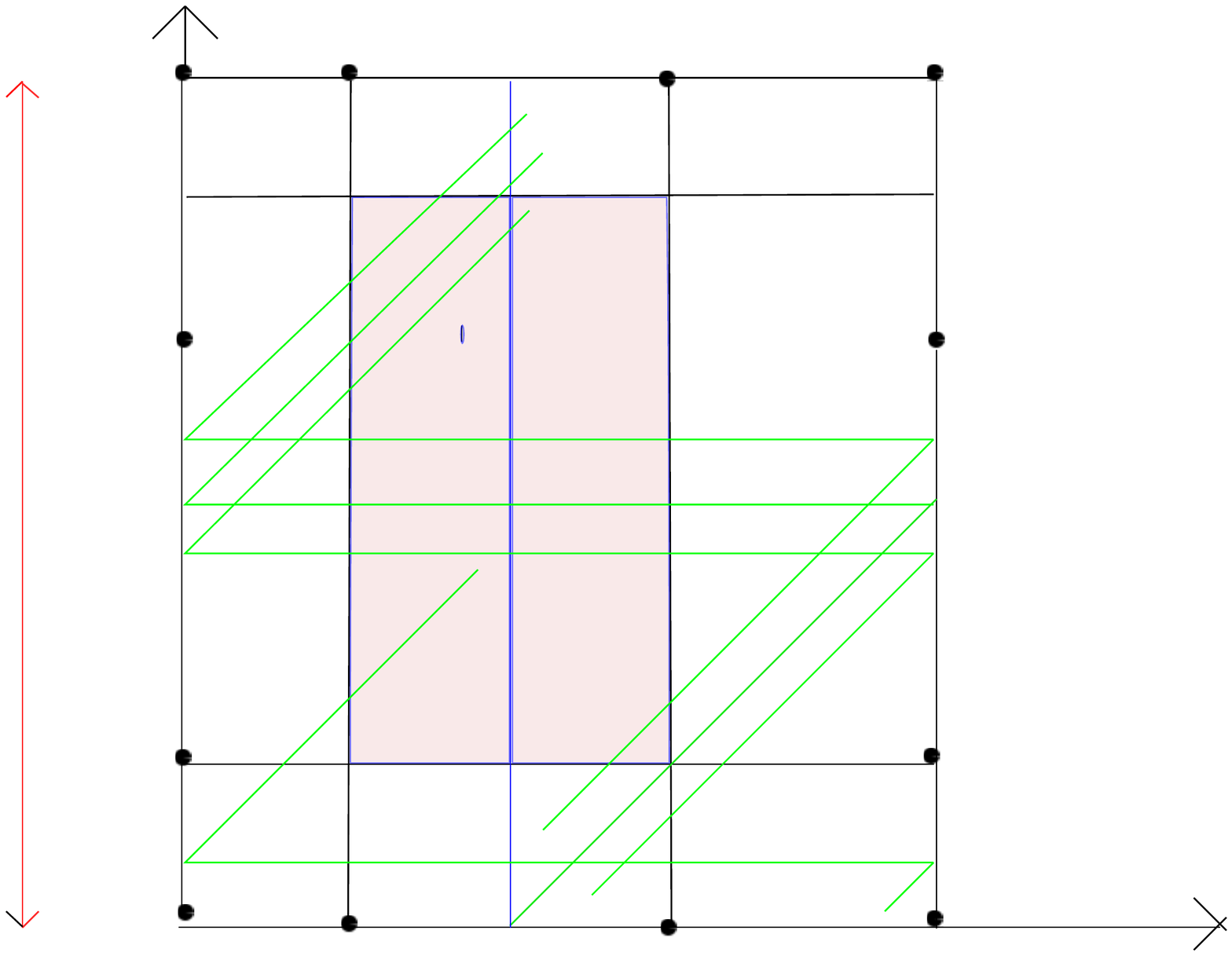}
				\put (26,61) {$S$}	
				\put (69.5,13) {$A$}
				\put (26.5,13.5) {$0$}
				\put (54.5,13) {$a_2$}
				\put (26,24.5) {$s_1$}
				\put (22,47) {$s_2-a_1$}
				\put (26,55) {$s_2$}
				\put (36.5,13) {$a_1$}
				\put (42,35.6) {$Q_{1}$}	
				\put(0,35){$\int\limits_{0}^{S}\beta(a,s,l)q(x,0,l,t)dl$}
				\put(50.7,17){\makebox[0pt]{\Huge\textcolor{black}{.}}}
				\put(66.8,16.1){\makebox[0pt]{\Huge\textcolor{black}{.}}}
				\put(48.1,20.7){\makebox[0pt]{\Huge\textcolor{black}{.}}}
				\put(46.3,15.4){\makebox[0pt]{\Huge\textcolor{black}{.}}}	
			\end{overpic}
			\caption{ The backward characteristics starting from $(a,s,T)$ with
				$a\in (a_0,A)$ (red lines or green lines) hits the boundary $(a=A)$, gets renewed
				by the renewal condition $\int\limits_{0}^{S}\beta(a,\hat{s},s)q(x,0,s,t)ds$ and then 
				enters the observation domain (green lines) or without by the boundary $s=S$.
				So, with the conditions $T>A-a_2+T_1+T_0$ all the characteristics starting at $(a,s,T)$ with  $a\in (a_0,A)$ get renewed by the renewal condition $\int\limits_{0}^{S}\beta(a,s,s')q(x,0,s',t)ds'$ with $t>T_1$ and enter the observation domain or without all the domain by the boundary $s=S$ with $a<\hat{a}.$ }
		\end{figure}
		As, $T>A-a_2+T_1+T_0$, there exists $\delta>0$ such that $$T>A-a_2+T_1+T_0+2\delta;$$ therefore $$T-(T_1+\delta)>A-(a_2-T_0+\delta).$$
		Moreover for 
		\[a\in (a_2-T_0-\delta,A)\hbox{ we have }T-(T_1+\delta)>A-(a_2-T_0-\delta)>A-a.\]
		Then \[u_1(x,a,s,T)=0\hbox{ a.e. }x\in \Omega\hbox{, }a\in(a_2-(T_0-\delta),A).\]
		Therefore 
		\[\int\limits_{0}^{S}\int\limits_{a_1}^{A}\int_{\Omega}u_{1}^2(x,a,s,T)dxdads=\int\limits_{0}^{S}\int\limits_{a_1}^{a_2-T_0+\delta}\int_{\Omega}u_{1}^{2}(x,a,s,T)dxdads.\]
		As $$T-(T_1+\delta)>A-(a_2-(T_0+\delta)),$$ then using the Proposition 3.4., we obtain
		
		\begin{equation}
			\int\limits_{0}^{S}\int\limits_{a_1}^{A}\int_{\Omega}u_{1}^{2}(x,a,s,T)dxdads\leq K\int\limits_{\eta}^{T}\int\limits_{s_1}^{s_2}\int\limits_{a_1}^{a_2}\int_{\omega}u_{1}^{2}(x,a,s,t)dxdadsdt.\label{deba}
		\end{equation}
		As $$q=u_1+u_2\Longleftrightarrow u_1=q-u_2$$ then 
		\begin{equation}
			\int\limits_{0}^{S}\int\limits_{a_1}^{A}\int_{\Omega}u_{1}^{2}(x,a,s,T)dxdads\leq 2\left(\int\limits_{\eta}^{T}\int\limits_{s_1}^{s_2}\int\limits_{a_1}^{a_2}\int_{\omega}q^{2}(x,a,s,t)dxdadsdt+\int\limits_{\eta}^{T}\int\limits_{s_1}^{s_2}\int\limits_{a_1}^{a_2}\int_{\omega}u_{2}^{2}(x,a,s,t)dxdadsdt\right)..\label{deba1}
		\end{equation}
		From the Proposition 3.5.; we get 
		\begin{equation}
			\int\limits_{0}^{S}\int\limits_{a_1}^{A}\int_{\Omega}u_{1}^{2}(x,a,s,T)dxdads\leq K\int\limits_{0}^{T}\int\limits_{s_1}^{s_2}\int\limits_{a_1}^{a_2}\int_{\omega}q^{2}(x,a,s,t)dxdadsdt..\label{deba11}
		\end{equation}
		Combining the inequalities $(\ref{debay})$ and $(\ref{deba11})$, we obtain
		\begin{equation}
			\int\limits_{0}^{S}\int\limits_{a_1}^{A}\int_{\Omega}q^{2}(x,a,s,T)dxdads\leq K_1 \int\limits_{0}^{T}\int\limits_{s_1}^{s_2}\int\limits_{a_1}^{a_2}\int_{\Omega}q^{2}(x,a,s,t)dxdadsdt.\label{coup1}
		\end{equation}
		Finally, $(\ref{coupp})$ and $(\ref{coup1})$ give
		\begin{equation}
			\int\limits_{0}^{S}\int\limits_{0}^{A}\int_{\Omega}q^{2}(x,a,s,T)dxdads\leq K_T \int\limits_{0}^{T}\int\limits_{s_1}^{s_2}\int\limits_{a_1}^{a_2}\int_{\Omega}q^{2}(x,a,s,t)dxdadsdt.\label{coup}
		\end{equation}
	\end{proof}
	In this part we are interested in a different control domain. Indeed, the idea is to be able to act on the same individuals until there $a_2$ age. To do this we need an additional requirement on the kernel $\beta.$ Indeed, we assume more than	 $\beta(a,\hat{s},s)=0\hbox{ for all } s\in (s_e,S),$ this means that the size of the new borns is always lower than $s_e.$ 
	\subsection{Proof of the Theorem 1.2}
	We still consider in this part the same problem of null controllability of the system $\ref{2},$ but by acting now on the system through the control domain \[Q_{2}=\{(x,a,s)\hbox{ such that } (x,a,s)\in \omega\times(a_1,a_2)\times(0,S)\hbox{ and } a-a_0<s<a+s_e\}.\]
	We suppose that the assumption $(H3)$ and $(H_4)$ holds.
	Under the assumption  $(H_4)$ the adjoint system of $(\ref{2}) $ becomes:
	\begin{equation}
		\left\lbrace\begin{array}{ll}
			\dfrac{\partial q}{\partial t}-\dfrac{\partial q}{\partial a}-\dfrac{\partial q}{\partial s}-\Delta q+(\mu_1(a)+\mu_2(s))q=\int\limits_{0}^{s_e}\beta(a,s,\hat{s})q(x,0,\hat{s},t)dad\hat{s}&\hbox{ in }Q ,\\ 
			\dfrac{\partial q}{\partial\nu}=0&\hbox{ on }\Sigma,\\ 
			q\left( x,A,s,t\right)=0&\hbox{ in } Q_{S,T} ,\\
			q\left(x,a,s,0\right)=q_0(x,a,s)&
			\hbox{ in }Q_{A,S};\\
			q(x,a,S,t)=0& \hbox{ in } Q_{S,T}.
		\end{array}
		\right.
		\label{45}
	\end{equation}
	Consequently, the estimation of the renewal term with respect to the variable $s$ will be done on $(0,s_e).$\\
	In view of \cite[\quad Theorem\quad 11.2.1]{b8} ,
	the result of the Theorem 1.2 is then reduced to the following Theorem.
	\begin{theorem}
		Under the assumption of the theorem 1.2, the pair $(\mathcal{A}^*,\mathcal{B}^*)$ is final-state observable for every $T>\max\{S-s_e,a_1+A-a_0\}$. In other words, for every $T>\max\{S-s_e,a_1+A-a_0\}$ , there exist $K_T>0$ such that the solution $q$ of (\ref{45}) satisfies
		\begin{equation}
			\int\limits_{0}^{A}\int\limits_{0}^{S}\int_{\Omega}q^2(x,a,s,T)dxdsda\leq K_T\int\limits_{0}^{T}\int_{Q_{2}}q^2(x,a,s,t)dxdadsdt \label{A}
		\end{equation}
	\end{theorem}
	For the proof, we need the following results:
	\begin{proposition}
		Let us assume the assumption $(H_1)-(H_3)$ and $(H_5)$. If $a_1<\hat{a}$ and $a_1<\eta<T$ , then there exists a constant $C>0$ such that for every $q_{0}\in K,$ the solution $q$ of the system (\ref{45}) verifies the following inequality:
		\begin{equation}
			\int\limits_{\eta}^{T}\int\limits_{0}^{s_e}\int_{\Omega}q^2(x,0,s,t)dxdsdt\leq C\int_{0}^{T} \int\limits_{a_1}^{a_2}\int\limits_{a}^{a+s_e}\int_{\omega}q^2(x,a,s,t)dxdsdadt \label{B}.
		\end{equation}
	\end{proposition}
	\begin{proof}{Proof of the Proposition 3.4}\\
		For $a\in (0,\hat{a})$ we have $\beta(a,\hat{s},s)=0$, therefore the system (\ref{45}) is can be written by
		\begin{equation}
			\left\lbrace\begin{array}{ll}
				\dfrac{\partial q}{\partial t}-\dfrac{\partial q}{\partial a}-\dfrac{\partial q}{\partial s}-\Delta q+(\mu_1(a)+\mu_2(s))q=0&\hbox{ in }Q_{\hat{a}},\\ 
				\dfrac{\partial q}{\partial\nu}=0&\hbox{ on }\Sigma',\\ 
				q(x,a,s,0)=q_0(x,a,s)& \hbox{ in } Q_{\hat{a},S}.
			\end{array}
			\right.
			\label{ad}
		\end{equation}	
		where $Q_{\hat{a}}=\Omega\times(0,\hat{a})\times (0,S)\times (0,T)$ and $\Sigma '=\partial\Omega\times(0,\hat{a})\times (0,S)\times(0,T).$\\
		\\
		We denote by $\tilde{q}(x,a,s,t)=q(x,a,s,t)\exp\left(-\int_{0}^{a}\mu_1(r)dr-\int_{0}^{s}\mu_2(r)dr\right).$ then $\tilde{q}$ satisfies
		\begin{equation}
			\left\lbrace 
			\begin{array}{l}
				\dfrac{\partial \tilde{q}}{\partial t}-\dfrac{\partial \tilde{q}}{\partial a}-\dfrac{\partial \tilde{q}}{\partial s}-\Delta \tilde{q}=0\text{ in }Q_{\hat{a}}\\
				\dfrac{\partial \tilde{q}}{\partial\nu}=0\text{ on }\Sigma '	\end{array}
			\right.
			\label{ad1}
		\end{equation}
		Proving the inequality (\ref{B}) lead also to show that,
		there exits a constant $C>0$ such that the solution $\tilde{q}$ of (\ref{ad1}) satisfies 
		\begin{equation}
			\int\limits_{\eta}^{T}\int\limits_{0}^{s_e}\int_{\Omega}\tilde{q}^2(x,0,s,t)dxdsdt\leq C\int\limits_{0}^{T}\int\limits_{a_1}^{a_2}\int\limits_{a}^{a+s_e}\int_{\omega}\tilde{q}^2(x,\alpha,s,t)dxdsdadt. \label{cci}
		\end{equation}
		Indeed, we have 
		\[
		\int\limits_{\eta}^{T}\int\limits_{0}^{s_e}\int_{\Omega}q^2(x,0,s,t)dxdsdt\leq e^{\left(2\int\limits_{0}^{\hat{a}}\mu_1(r)dr+2\int\limits_{0}^{s_e}\mu_2(r)dr\right)}\int\limits_{\eta}^{T}\int\limits_{0}^{s_e}\int_{\Omega}\tilde{q}^2(x,0,s,t)dxdsdt\]
		then
		\[\int\limits_{\eta}^{T}\int\limits_{0}^{s_e}\int_{\Omega}q^2(x,0,s,t)dxdsdt\leq e^{\left(2\|\mu_1\|_{L^{1}(0,\hat{a})}+2\|\mu_2\|_{L^{1}(0,s_e)}\right)}\int\limits_{\eta}^{T}\int\limits_{0}^{s_e}\int_{\Omega}\tilde{q}^2(x,0,s,t)dxdsdt.
		\]	
		
		We consider the following characteristics trajectory $\gamma(\lambda)=(t-\lambda,s+t-\lambda,\lambda).$ If $t-\lambda=0$ the backward characteristics starting from $(0,s,t).$ If $T<a_1$ the trajectory $\gamma(\lambda)$ never reaches the observation region $Q_2$ (see Figure 9). So we choose $T>a_1.$\\
		Without loss the generality, let us assume here $\eta< a_2=\hat{a}<T.$\\
		The proof will done in two steps:\\
		\[\textbf{step 1: Estimation for } t\in (\eta,\hat{a}) \]
		We denote by:\\
		\[w(x,\lambda)=\tilde{q}(x,t-\lambda,s+t-\lambda,\lambda) \text{ ; }(x\in\Omega\text{, }\lambda\in (0,t))\]
		Then $w$ satisfies:
		\begin{equation}
			\left\lbrace\begin{array}{ll}
				\dfrac{\partial w(x,\lambda)}{\partial \lambda}-\Delta w(x,\lambda)=0&\hbox{ in } \Omega\times (0,t)\\
				\dfrac{\partial w}{\partial \nu}=0& \hbox{ on } \partial\Omega\times (0,t)	\\
				w(x,0)	=\tilde{q}(x,t,s+t,0)&\hbox{ in }\Omega
			\end{array}
			\right.,
		\end{equation}
		Using the Proposition 3.2  with $0<t_0<t_1<t$ we obtain:\\
		\[\int_{\Omega}w^2(x,t)dx\leq\int_{\Omega}w^2(x,t_1)dx\leq c_1e^{\dfrac{c_2}{t_1-t_0}}\int\limits_{t_0}^{t_1}\int_{\omega}w^2(x,\lambda)dxd\lambda.\]
		That is equivalent to
		\[\int_{\Omega}\tilde{q}^2(x,0,s,t)dx\leq c_1e^{\dfrac{c_2}{t_1-t_0}}\int\limits_{t_0}^{t_1}\int_{\omega}\tilde{q}^2(x,t-\lambda,s+t-\lambda,\lambda)dxd\lambda=C\int\limits_{t-t_1}^{t-t_0}\int_{\omega}\tilde{q}^2(x,\alpha,s+\alpha,t-\alpha)dxd\alpha.\]	
		Then for $t_0=0$ and $t_1=t-a_1,$ we obtain\\
		\[\int_{\Omega}\tilde{q}^2(x,0,s,t)dx\leq C\int\limits_{a_1}^{t}\int_{\omega}\tilde{q}^2(x,\alpha,s+\alpha,t-\alpha)dxd\alpha.\]
		Integrating with respect $s$ over $(0,s_e)$ we get
		\[\int\limits_{0}^{s_e}\int_{\Omega}\tilde{q}^2(x,0,s,t)dxds\leq C
		\int\limits_{a_1}^{t}\int\limits_{\alpha}^{s_e+\alpha}\int_{\omega}\tilde{q}^2(x,\alpha,l,t-\alpha)dxdld\alpha.\]
		Finaly, integrating with respect $t$ over $(\eta,\hat{a})$, we obtain
		\[\int\limits_{\eta}^{\hat{a}}\int\limits_{0}^{s_e}\int_{\Omega}\tilde{q}^2(x,0,s,t)dxdsdt\leq C \int\limits_{\eta}^{\hat{a}}\int\limits_{a_1}^{t}\int\limits_{a}^{s_e+a}\int_{\omega}\tilde{q}^2(x,a,s,t-a)dxdsdadt.\]
		Then
		\[\int\limits_{\eta}^{\hat{a}}\int\limits_{0}^{s_e}\int_{\Omega}\tilde{q}^2(x,0,s,t)dxdsdt\leq C \int\limits_{a_1}^{\hat{a}}\int\limits_{\nu}^{\hat{a}}\int\limits_{a}^{s_e+a}\int_{\omega}\tilde{q}^2(x,a,s,t-\nu)dxdtdsd\nu\]
		\[\leq C\int_{0}^{T} \int\limits_{a_1}^{a_2}\int\limits_{a}^{a+s_e}\int_{\omega}\tilde{q}^2(x,a,s,t)dxdsdadt.\]
		Then
		\begin{equation}
			\int\limits_{\eta}^{\hat{a}}\int\limits_{0}^{s_e}\int_{\Omega}\tilde{q}^2(x,0,s,t)dxdsdt\leq C\int_{0}^{T} \int\limits_{a_1}^{a_2}\int\limits_{a}^{a+s_e}\int_{\omega}\tilde{q}^2(x,a,s,t)dxdsdadt \label{z}
		\end{equation}
		\[\textbf{step 2: Estimation for } t\in(\hat{a},T)\]
		We denote by:\\
		\[w(x,\lambda)=\tilde{q}(x,t-\lambda,s+t-\lambda,\lambda) \text{ ; }(\lambda\in (t-\hat{a},t)\text{ , }x\in\Omega)\]Then $w$ satisfies:
		\begin{equation}
			\left\lbrace
			\begin{array}{l}
				\dfrac{\partial w(x,\lambda)}{\partial \lambda}-\Delta w(x,\lambda)=0\text{ in }  \Omega\times(t-\hat{a},t)\\
				\dfrac{\partial w}{\partial \nu}=0 \text{ on } \partial\Omega\times(t-\hat{a},t)	\\
				w(x,0)	=\tilde{q}(x,\hat{a},s+\hat{a},t-\hat{a})\text{ in }\Omega
			\end{array}
			\right.,
		\end{equation}
		Using the Proposition 3.2  with $t-\hat{a}<t_0<t_1<t$ we obtain:\\
		\[\int_{\Omega}w^2(x,t)dx\leq\int_{\Omega}w^2(x,t_1)dx\leq c_1e^{\dfrac{c_2}{t_1-t_0}}\int\limits_{t_0}^{t_1}\int_{\omega}w^2(x,\lambda)dxd\lambda.\]
		That is equivalent to
		\[\int_{\Omega}\tilde{q}^2(x,0,s,t)dx\leq c_1e^{\dfrac{c_2}{t_1-t_0}}\int\limits_{t_0}^{t_1}\int_{\omega}\tilde{q}^2(x,t-\lambda,s+t-\lambda,\lambda)dxd\lambda=C\int\limits_{t-t_1}^{t-t_0}\int_{\omega}\tilde{q}^2(x,\alpha,s+\alpha,t-\alpha)dxd\alpha.\]	
		Then for $t_0=t-\hat{a}$ and $t_1=t-a_1,$ we obtain\\
		\[\int_{\Omega}\tilde{q}^2(x,0,s,t)dx\leq C\int\limits_{a_1}^{\hat{a}}\int_{\omega}\tilde{q}^2(x,\alpha,s+\alpha,t-\alpha)dxd\alpha.\]
		Integrating with respect $s$ over $(0,s_e)$ we get
		\[\int\limits_{0}^{s_e}\int_{\Omega}\tilde{q}^2(x,0,s,t)dxds\leq C
		\int\limits_{a_1}^{\hat{a}}\int\limits_{\alpha}^{s_e+\alpha}\int_{\omega}\tilde{q}^2(x,\alpha,l,t-\alpha)dxdld\alpha.\]
		Finally, integrating with respect $t$ over $(\hat{a},T)$, we obtain
		\[\int\limits_{\hat{a}}^{T}\int\limits_{0}^{s_e}\int_{\Omega}\tilde{q}^2(x,0,s,t)dxdsdt\leq C \int\limits_{a_1}^{a_2}\int\limits_{a}^{s_e+a}\int\limits_{\hat{a}-a}^{T-a}\int_{\omega}\tilde{q}^2(x,a,s,t)dxdtdsda\]\[\leq C\int\limits_{0}^{T} \int\limits_{a_1}^{a_2}\int\limits_{a}^{a+s_e}\int_{\omega}\tilde{q}^2(x,a,s,t)dxdsdadt.\]
		Then
		\begin{equation}
			\int\limits_{\hat{a}}^{T}\int\limits_{0}^{s_e}\int_{\Omega}\tilde{q}^2(x,0,s,t)dxdsdt\leq C\int\limits_{0}^{T} \int\limits_{a_1}^{a_2}\int\limits_{a}^{a+s_e}\int_{\omega}\tilde{q}^2(x,a,s,t)dxdsdadt \label{zz}
		\end{equation}
		Combining $(\ref{z})$ and $(\ref{zz}),$ we obtain the result.
	\end{proof}
	\begin{proposition}
		Let us assume the assumption $(H_1)-(H_3)$ and $(H_5).$ If $a_1<a_0<\hat{a},$ there exists $C_T>0$ such that the solution $q$ of the system $(\ref{45})$ verifies the following inequality
		\begin{equation}
			\int\limits_{0}^{a_0}\int\limits_{0}^{s_e}\int_{\Omega}q^2(x,a,s,T)dxdsda\leq C\int\limits_{0}^{T} \int\limits_{a_1}^{a_2}\int\limits_{a-a_0}^{a+s_e}\int_{\omega}q^2(x,a,s,t)dxdsdadt \label{C}
		\end{equation}
	\end{proposition}
	\begin{proof}{Proof of the Proposition 3.5}\\
		We consider in this proof the characteristics $\gamma (\lambda)=(a+T-\lambda,s+T-\lambda,\lambda).$
		For $\lambda=T$ the characteristics starting from $(a,s,T).$\\
		We have three cases.\\
		\textbf{Case 1}: $T<a_2$ and $a_2\leq \hat{a}.$\\
		Two situations can arise: \\
		$\bullet$ $b_0=a_2-T<a_1<a_0$ in this situation we split the interval $(0,a_0)$ as \begin{equation}(0,a_0)=(0,b_0)\cup(b_0,a_1)\cup(a_1,a_0)\label{iu}.\end{equation}\\
		$\bullet$ $a_1<b_0<a_0$, in this situation we split the interval $(0,a_0)$ as $(0,a_0)=(0,b_0)\cup (b_0,a_0).$	\\ 
		\textbf{Case 2}: $T\geq a_2$ and $a_2\leq\hat{a}.$\\ In this case we split the interval $(0,a_0)$ as $(0,a_0)=(0,a_1)\cup (a_1,a_0).$\\
		\textbf{Case 3}: $a_2>\hat{a}$:\\ In this case we proof similarly the observability in $(a_1,a_0)$ and to expand to $(a_1,a_2)$.\\
		Here we give the proof in the only situation where 
		\begin{equation}
			T<a_2 \text{ and } b_0=a_2-T<a_1. \label{el}
		\end{equation}
		In the remaining part of the proof we give upper bounds for $\int_{0}^{S}\int_{I}\int_{\Omega}\tilde{q}(x,a,s,T)dxdadt$
		where $I$ is successively each one of the intervals appearing in the decomposition $(\ref{iu})$\\
		\textbf{Upper bound on $(0,b_0)$}:\\
		For $a\in(0,b_0)$ we first set 
		$w(\lambda,x)=\tilde{q}(x,a+T-\lambda,s+T-\lambda,\lambda)\text{    } (\lambda\in (0,T)\text{ and } x\in\Omega)$ where $$\tilde{q}=\exp\left(-\int\limits_{0}^{a}\mu(r)dr-\int\limits_{0}^{s}\mu_2(r)dr\right)q.$$
		Then $w$ verifies 
		\begin{equation}
			\left\lbrace
			\begin{array}{l}
				\dfrac{\partial w(x,\lambda)}{\partial \lambda}-\Delta w(x,\lambda)=0\text{ in } (0,T)\times \Omega\\
				\dfrac{\partial w}{\partial \nu}=0 \text{ on }(0,T)\times \partial\Omega	\\
				w(x,0)	=\tilde{z}(x,a+T,s+T,0)\text{ in }\Omega
			\end{array}
			\right.,\label{ooo}
		\end{equation}
		By applying the Proposition 3.2 with $t_0=0$ and $t_1=a+T-a_1,$ we obtain:\\
		\[\int_{\Omega}w^2(x,T)dx\leq c_1e^{\dfrac{c_2}{a+T-a_1}}\int\limits_{0}^{a+T-a_1}\int_{\omega}w^2(x,\lambda)dxd\lambda.\]
		
		Then we have 
		\[\int_{\Omega}\tilde{q}^2(x,a,s,T)dx\leq c_1e^{\dfrac{c_2}{a+T-a_1}}\int\limits_{0}^{a+T-a_1}\int_{\omega}\tilde{q}^2(x,a+T-\lambda,s+T-\lambda,\lambda)dxd\lambda\]\[=C\int\limits_{a_1}^{a+T}\int_{\omega}\tilde{q}^2(x,\alpha,s+\alpha-a,a+T-\alpha)dxd\alpha\leq C\int\limits_{a_1}^{a_2}\int_{\omega}\tilde{q}^2(x,\alpha,s+\alpha-a,a+T-\alpha)dxd\alpha.\]	
		Integrating with respect $s$ over $(0,s_e)$ we get
		\[\int\limits_{0}^{s_e}\int_{\Omega}\tilde{q}^2(x,a,s,T)dxds\leq C\int\limits_{0}^{s_e}\int\limits_{a_1}^{a_2}\int_{\omega}\tilde{q}^2(x,\alpha,s+\alpha-a,a+T-\alpha)dxd\alpha.\]
		As
		\[\int\limits_{0}^{s_e}\int\limits_{a_1}^{a_2}\int_{\omega}\tilde{q}^2(x,\alpha,s+\alpha-a,a+T-\alpha)dxd\alpha ds=\int\limits_{a_1}^{a_2}\int\limits_{0}^{s_e}\int_{\omega}\tilde{q}^2(x,\alpha,s+\alpha-a,a+T-\alpha)dxd\alpha ds,\]
		then
		\[\int\limits_{0}^{s_e}\int_{\Omega}\tilde{q}^2(x,a,s,T)dxds\leq C\int\limits_{a_1}^{a_2}\int\limits_{\alpha-a}^{s_e+a}\int_{\omega}\tilde{q}^2(x,\alpha,l,a+T-\alpha)dxdld\alpha .\]
		So, integrating with respect $a$ over $(0,b_0)$ we get
		\[\int\limits_{0}^{b_0}\int\limits_{0}^{s_e}\int_{\Omega}\tilde{q}^2(x,a,s,T)dxdsda\leq C \int\limits_{0}^{b_0} \int\limits_{a_1}^{a_2}\int\limits_{\alpha-b_0}^{s_e+\alpha}\int\limits_{\omega}\tilde{q}^2(x,\alpha,l,a+T-\alpha)dxdl d\alpha da\]
		\[=C \int\limits_{a_1}^{a_2}\int\limits_{a-b_0}^{s+a}\int\limits_{T-\alpha}^{T-\alpha+b_0}\int_{\omega}\tilde{q}^2(x,a,s,t)dxdtdsda\leq C\int_{0}^{T} \int\limits_{a_1}^{a_2}\int\limits_{a-b_0}^{a+s_e}\int_{\omega}\tilde{q}^2(x,a,s,t)dxdsdadt. \]
		Finally
		\begin{equation}
			\int\limits_{0}^{b_0}\int\limits_{0}^{s_e}\int_{\Omega}\tilde{q}^2(x,a,s,T)dxdsda\leq C\int\limits_{0}^{T} \int\limits_{a_1}^{a_2}\int\limits_{a-b_0}^{a+s_e}\int_{\omega}\tilde{q}^2(x,a,s,t)dxdsdadt \label{7}
		\end{equation}
		\textbf{Upper bound $(b_0,a_1)$}:\\
		For $a\in (b_0,a_1)$	we consider always the system $(\ref{ooo})$ but $\lambda\in (T+a-a_2,T).$\\
		Applying the Proposition 3.2 with $t_0=T+a-a_2$ and $t_1=T+a-a_1$, we obtain
		\begin{equation*}
			\displaystyle\int_{\Omega}\tilde{q}^2(x,a,s,T)dx\leq C\int\limits_{T+a-a_2}^{a+T-a_1}\int_{\omega}\tilde{q}^2(x,a+T-\lambda,s+T-\lambda,\lambda)dxd\lambda= C\int\limits_{a_1}^{a_2}\int_{\omega}\tilde{q}^2(x,\alpha,s+\alpha-a,a+T-\alpha)dxd\alpha.
		\end{equation*}	
		And as before we get 
		\begin{equation}
			\displaystyle\int\limits_{b_0}^{a_1}\int\limits_{0}^{s_e}\int_{\Omega}\tilde{q}^2(x,a,s,T)dxdsda\leq C\int\limits_{0}^{T}\int\limits_{a_1}^{a_2}\int\limits_{a-a_1}^{a+s_e}\int_{\omega}\tilde{q}^2(x,a,s,t)dxdsdadt \label{9}
		\end{equation}
		\textbf{Upper bound $(a_1,a_0)$}:\\
		For $a\in (a_1,a_0)$	 we use again the system $(\ref{ooo})$ but $\lambda\in (T+a-a_2,T).$\\
		Applying the Proposition 3.2 with $t_0=T+a-a_2$ and $t_1=T$, we obtain
		\[\int_{\Omega}\tilde{q}^2(x,a,s,T)dx\leq C\int\limits_{T+a-a_0}^{a+T-a_1}\int_{\omega}\tilde{q}^2(x,a+T-\lambda,s+T-\lambda,\lambda)dxd\lambda= C\int\limits_{a_1}^{a_0}\int_{\omega}\tilde{q}^2(x,\alpha,s+\alpha-a,a+T-\alpha)dxd\alpha.\]	
		Integrating with respect $s$ over $(0,S)$ and $a$ over $(a_1,a_0)$ we get
		\begin{equation}
			\int\limits_{a_1}^{a_0}\int\limits_{0}^{s_e}\int_{\Omega}\tilde{q}^2(x,a,s,T)dxdsda\leq C\int\limits_{0}^{T}\int\limits_{a_1}^{a_1}\int\limits_{a-a_0}^{a+s_e}\int_{\omega}\tilde{q}^2(x,a,s,t)dxdsdadt \label{10}
		\end{equation}
		Consequently, combining $(\ref{7})$, $(\ref{9})$ and $\ref{10}$ we obtain:
		\[\int\limits_{0}^{s_e}\int\limits_{0}^{a_0}\int_{\Omega}q^2(x,a,s,T)dxdads\leq C\int\limits_{0}^{T}\int\limits_{a_1}^{a_2}\int\limits_{a-a_0}^{a+s_e}\int_{\omega}q^2(x,a,s,t)dxdadsdt.\]	
	\end{proof} 
	For it, we consider the section $S(T)$ define by $$S(T)=\{(a,s,t)\in Q'\hbox{ such that }t=T\}$$
	and split this surface into two parts $$V_1(T)=[0,a_0)\times [0,S] \hbox{ and } V_2(T)=[a_0,A]\times [0,S].$$ From $T>\max\{a_1+A-a_0,a_1+S-s_2\},$ we get the following result:	
	\begin{lemma}
		Let us suppose that $T>\max\{a_1+A-a_0,S-s_e\}$ where $a_0\in [a_1,\hat{a}].$ Then $A_1\cap V_2(T)=\emptyset$	
		and $\max\{T-A+a,  T-S+s\}>a_1$ in $V_2(T).$ Moreover \[q(x,a,s,T)=\int\limits_{\max\{T-A+a,  T-S+s\}}^{T}\left(e^{(T-l)L}\int\limits_{0}^{s_e}\beta(a+T-l,s+T-l,\hat{s})q(x,0,\hat{s},l)d\hat{s})\right)dl\] in $\Omega\times V_2(T).$
	\end{lemma}
	\begin{proof}{Proof of Lemma}\\
		Let suppose that $T>\max\{a_1+A-a_0,S-s_e\}.$ In $V_2(T),$ we have $T+a-A>T+a_0-A>a_1\hbox{ because }a\in (a_0,A)$ Moreover, we have $A-T<a_0-a_1<a_0,$ as $T>S-s_e,$ we have also $S-T<s_e.$ Finally, if $T>\max\{ A-a_0+a_1,S-s_e\},$ we have $A-T<a_0$ and $S-T<s_e$ then $V_2(T)\cap A_1=\emptyset$ (see Figure 1 and Figure 9), therefore we have:
		\[q(x,a,s,T)=\int\limits_{\max\{T-A+a,  T-S+s\}}^{T}\left(e^{(T-l)L}\int\limits_{0}^{s_e}\beta(a+T-l,s+T-l,\hat{s})q(x,0,\hat{s},l)d\hat{s})\right)dl\] in $\Omega\times V_2(T)$ (see Figure 9).		 
	\end{proof}
	\begin{proof}{Proof of the Theorem 3.2}\\
		We suppose that $T>\sup\{A-a_0+a_1,S-s_e\}$ and we 
		consider $a_0$, $s_e$ as in the previous Lemma. then
		\[\int\limits_{0}^{S}\int\limits_{0}^{A}\int_{\Omega}q^2(x,a,s,T)dxdads\]		\[=\int\limits_{0}^{s_e}\int\limits_{0}^{a_0}\int_{\Omega}q^2(x,a,s,T)dxdads+\int\int_{V}\int_{\Omega}q^2(x,a,s,T)dxdads\]	
		From the Proposition 3.2 we have 
		\begin{equation}
			\int\limits_{0}^{s_e}\int\limits_{0}^{a_0}\int_{\Omega}q^2(x,a,s,T)dxdads\leq C_{T}\int\limits_{0}^{T}\int\limits_{a_1}^{a_2}\int\limits_{a-a_0}^{a+s_e}\int_{\omega}q^2(x,a,s,t)dxdsdadt.\label{x1}
		\end{equation}	
		Using the result of the Lemma 3.2, we obtain:\\
		\[q(x,a,s,T)=\int\limits_{\max\{T-A+a,  T-S+s\}}^{T}\left(e^{(T-l)L}\int\limits_{0}^{s_e}\beta(a+T-l,s+T-l,\hat{s})q(x,0,\hat{s},l)d\hat{s})\right)dl\]
		Using the hypothesis on $\beta$ the regularity of the semi-group and the inequality of Cauchy Schwartz, we obtain:
		\[\int\int_{V}\int_{\Omega}q^2(x,a,s,T)\leq  C(\|\beta\|_{\infty},T)\int\int_{V}\int_{\Omega}\int\limits_{\max\{T-A+a,  T-S+s\}}^{T}\int\limits_{0}^{s_e}q^2(x,0,s,t)dsdt.\]
		And finally the Lemma 3.2 gives:
		\begin{equation}
			\int\int_{V}\int_{\Omega}q^2(x,a,s,T)dxdads\leq AS C(\|\beta\|_{\infty},T)\int\limits_{\eta}^{T}\int\limits_{0}^{s_e}\int_{\Omega}q^2(x,0,s,t)dxdsdt.
		\end{equation}
		Using the Proposition 3.1
		\begin{equation}
			\int\int_{V}\int_{\Omega}q^2(x,a,s,T)dxdads\leq K_T\int\limits_{0}^{T}\int\limits_{a_1}^{a_2}\int\limits_{a-a_0}^{a+s_e}\int_{\omega}q^2(x,a,s,t)dxdadsdt.\label{x2}
		\end{equation}
		Combining the inequality $(\ref{x1})$ and $(\ref{x2})$  we obtain.
		\begin{equation}
			\int\limits_{0}^{S}\int\limits_{0}^{A}\int_{\Omega}q^2(x,a,s,T)dxdads\leq K_T\int\limits_{0}^{T}\int\limits_{a_1}^{a_2}\int\limits_{a-a_0}^{a+s_e}\int_{\omega}q^2(x,a,s,t)dxdsdadt.
		\end{equation}
		\begin{figure}[H]
			\begin{overpic}[scale=0.45]{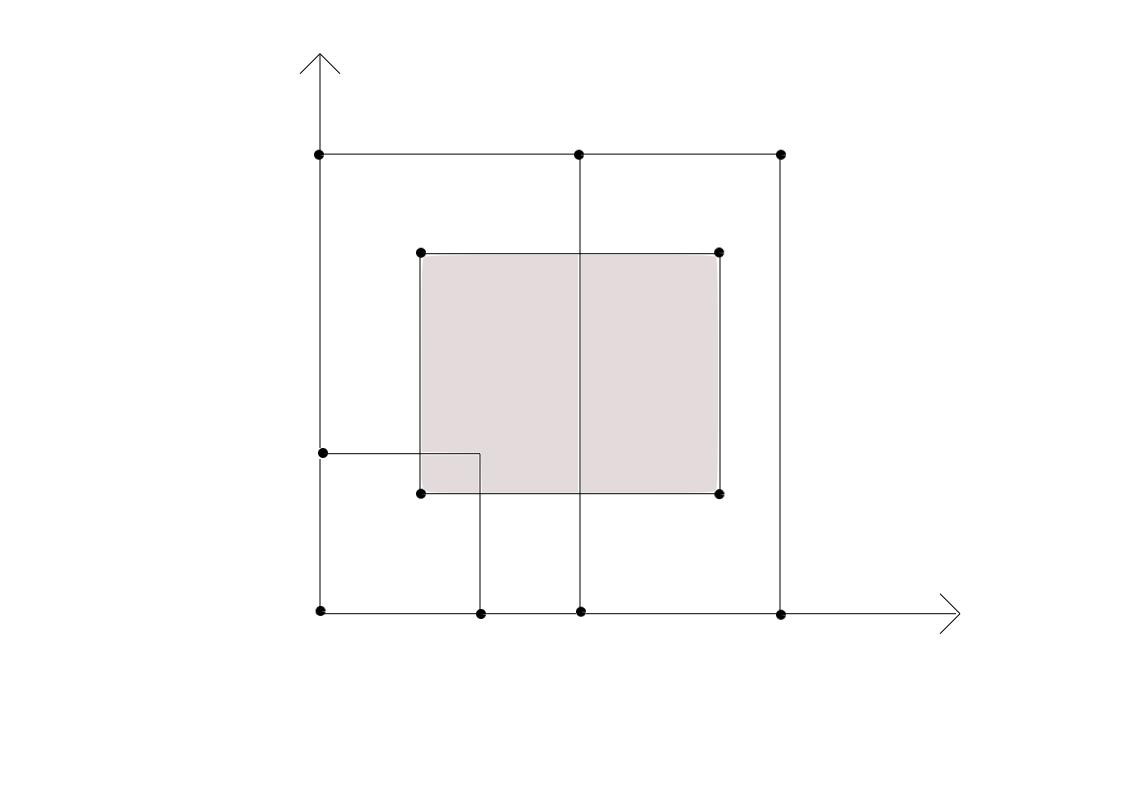}
				\put (25,55) {$S$}
				\put (35,50) {$S_1(T)$}
				\put (60,50) {$S_2(T)$}
				\put (69,12) {$A$}
				\put (52,14) {$a_0$}
				\put (21,30) {$S-T$}
				\put (39,34) {$Observation$}
				\put (32,21) {$A_1$}               
				\put (40,12) {$A-T$}     
			\end{overpic}
			\caption{This figure illustrates the result of the Lemma 3.1. Indeed, this figure corresponds to the section $S(T),$ and for $T>A-a_2+T_0+T_1$, we have $A_1\cap S_1(T)=\emptyset $}
		\end{figure}
		\begin{figure}[H]
			\begin{overpic}[scale=0.35]{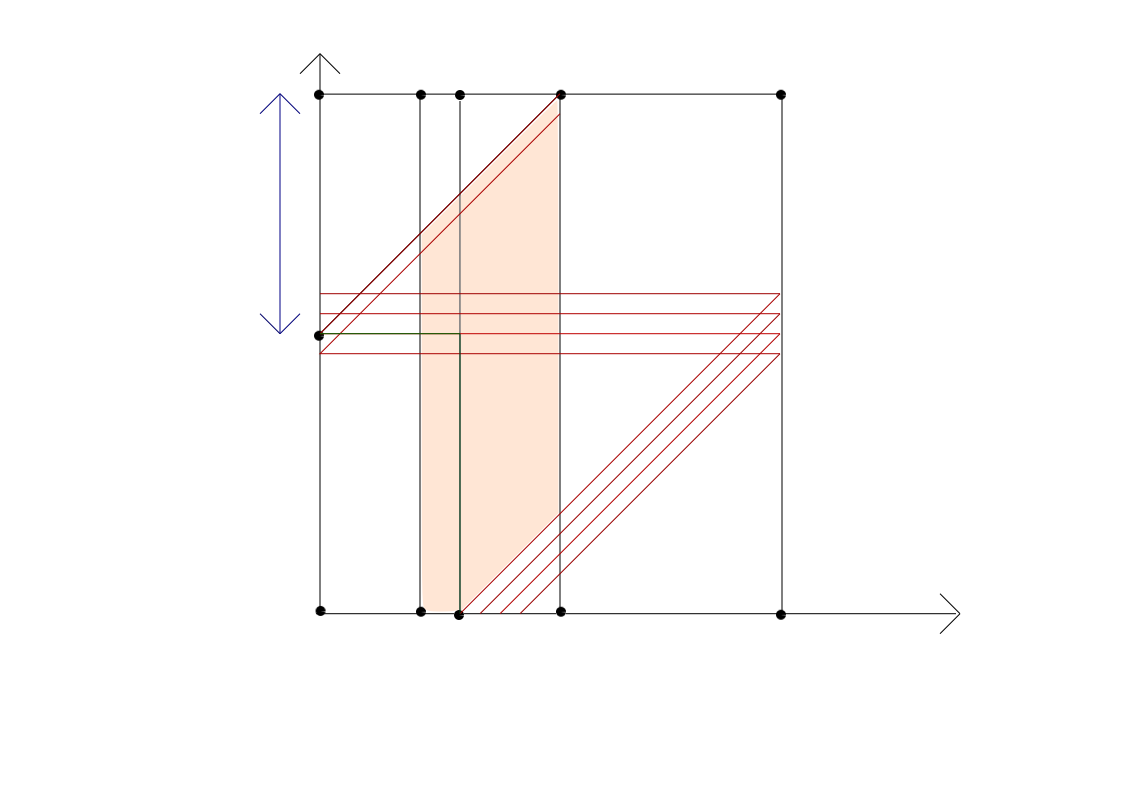}
				\put (25,61) {$S$}
				\put (-10,50) {$\beta(a,\hat{s},s)q(x,0,s,t)=0$}	
				\put (69,12) {$A$}
				\put (49,12) {$a_2$}
				\put (25,40) {$s_e$}
				\put (39,34) {$Obs$}
				\put (36,12) {$a_1$}               
				\put (40,64) {$a_0$}     
			\end{overpic}
			\caption{Illustration of the Lemma 3.2\\ The backward characteristics starting at $(a,s,T)$ with $a\in (a_0,A)$ hit the boundary condition $a=A$ (the maximal time to hit the boundary $a=A$ is $t=A-a_0$) and gets renewal by the renewal term $\int\limits_{0}^{s_e}\beta(a,s,\hat{s})q(x,0,\hat{s},t)d\hat{s}$ with $s\in (0,s_e)\text{ or }s\in (s_e,S).$\\ If $s\in (0,s_e)$ the characteristics need again $a_1$ time to enter in observation domain else the renewal term equal to zero }
		\end{figure}
		Then, (see Figure 10)
		\begin{equation}
			\int\limits_{0}^{S}\int\limits_{0}^{A}\int_{\Omega}q^2(x,a,s,T)dxdads\leq K_T\int\limits_{0}^{T}\int_{Q_{2}}q^2(x,a,s,t)dxdsdadt.
		\end{equation}
	\end{proof}
	\section{Controls preserving positivity}
	In this part we are interested in a problem of controllability with constraint of positivity on the state. We first establish the existence of steady solution for this model.
	And we study the existence of controls such that the corresponding state trajectories join twice different non-negative stationary states in some $T$, while preserving the positivity of the controlled trajectory for $t\in [0,T]$. This type of result has been proved in \cite{dy,dd} for the Lotka-McKendrick system with diffusion and by Loh\'eac, Tr\'elat and Zuazua for purely parabolic problems \cite{b5} (in a time depending on an appropriate norm of the difference of the two stationary states). We prove below that the situation encountered in the latter case also applies to the problem considered in the present work. An essential ingredient in obtaining this type of result is proving the null controllability of the system by means of $L^{\infty}$ controls and then slowly (s.t. positivity is preserved) driving, the initial state towards the desired target.\\
	To state our result on controllability with positivity constraints, we first define the concept
	of non-negative steady state for the system (\ref{2}). \\
	\subsection{Existence of steady states to (\ref{2})}
	In this subsection, we suppose that $\mu_2=0$ for simplify the calculus, Moreover, we suppose in this part that $\beta $ and $\mu_1$ verify $(H_1)$ and $(H_3).$
	Let $\overline{Q}=\Omega\times(0,A)\times(0,S)\hbox{, }\overline{\Sigma}=\partial\Omega\times(0,A)\times(0,S)\hbox{, }\overline{Q}_A=\Omega\times(0,A)\hbox{, }\overline{Q}_S=\Omega\times(0,S).$\\
	We consider the following system
	\begin{equation}
		\left\lbrace\begin{array}{ll}
			\dfrac{\partial y}{\partial t}+\dfrac{\partial y}{\partial a}+\dfrac{\partial y }{\partial s}-\Delta y+\mu_1(a)y=u &\hbox{ in }\overline{Q}\times(0,+\infty) ,\\ 
			\dfrac{\partial y}{\partial\nu}=0&\hbox{ on }\overline{\Sigma}\times(0,+\infty),\\ 
			y\left( x,0,s,t\right) =\displaystyle\int\limits_{0}^{A}\int\limits_{0}^{S}\beta(a,\hat{s},s)y(x,a,\hat{s},t)da d\hat{s}&\hbox{ in } \overline{Q}_{S}\times(0,+\infty) \\
			y\left(x,a,s,0\right)=0&
			\hbox{ in }\overline{Q};\\
			y(x,a,0,t)=0& \hbox{ in } \overline{Q}_{A}\times(0,+\infty).
		\end{array}\right..
		\label{222}
	\end{equation}
	where $u\in L^{\infty}(\overline{Q}\times(0,+\infty)).$\\
	\begin{definition}
		Let $u\in L^{\infty}(\overline{Q})$ be a steady interior control such that
		\[u\geq 0 \hbox{ a.e. in } \overline{Q}.\]
		A non-negative function $p\in L^{\infty}(\overline{Q})$ satisfying the equations	
		\begin{equation}
			\left\lbrace
			\begin{array}{ll}
				\dfrac{\partial p}{\partial a}+\dfrac{\partial p}{\partial s}-\Delta p+\mu_1(a)p=u &\text{ in }\overline{Q} ,\\ 
				\dfrac{\partial p}{\partial\nu}=0&\text{ on }\overline{\Sigma},\\ 
				p\left( x,0,s\right) =\int\limits_{0}^{A}\int\limits_{0}^{S}\beta(a,\hat{s},s)p(x,a,\hat{s})da d\hat{s}&\text{ in } \overline{Q}_S \\
				p\left(x,a,0\right)=0&
				\text{ in }\overline{Q}_A;\\
				
			\end{array}
			\right.\label{www}
		\end{equation}	  
		is said to be a non-negative steady state for the system $(\ref{222}).$
	\end{definition}
	The following gives the existence of a steady solution of the system $(\ref{222}).$
	\begin{proposition}
		Denote by $$\mathcal{R}(s)=\int\limits_{0}^{A}\int\limits_{0}^{S}\beta(a,\hat{s},s)\exp\left(-\int\limits_{0}^{a}\mu_1(r)dr\right)d\hat{s}da$$	 the reproductive number. 
		\begin{enumerate}
			\item If $\|\mathcal{R}\|_{L^{\infty}(0,S)}<1$ and $u\geq 0 \hbox{ a.e. in }\overline{Q},$ then there exists a unique non-negative solution to $(\ref{www}).$ Moreover, there exists four constants $\rho_0>0$, $a^*\in (0,A)$, $s_{1}^{*}\hbox{, }s_{2}^{*}\in (0,S)\hbox{ with }s_{1}^{*}<s_{2}^{*}$ such that $$p_s(x,a,s)>\rho_0\hbox{ a.e. in }\Omega \times [0,a^*]\times [s_{1}^{*},s_{2}^{*}].$$
			\item If $\mathcal{R}(s)=1\text{ and }u\equiv 0$ then there exist infinitely many solutions to $(\ref{www})$, which satisfy $(\ref{222}).$
		\end{enumerate}
	\end{proposition}
	For the proof of the  previous Proposition we need this result:
	\begin{lemma}
		The operator Laplacian verify the following inequality:
		\[\|e^{t\Delta}\psi\|_{L^{\infty}(\Omega)}\leq \|\psi\|_{L^{\infty}(\Omega)}\hbox{ for all } t\geq 0, \hbox{ } \psi\in L^{\infty}(\Omega).\]
	\end{lemma}
	For the proof of this result we refer to Ouhabaz \cite[\hbox{ Corollary } 4.10]{qu},
	\begin{proof}{Proof of the Proposition 4.1}\\
		Let $\eta\in L^{2}(\overline{Q})$ and $q=p\exp\left(\int\limits_{0}^{a}\mu_1(r)dr\right).$
		The function $q$ solve:
		\begin{equation}
			\left\lbrace
			\begin{array}{ll}
				\dfrac{\partial q}{\partial a}+\dfrac{\partial q}{\partial s}-\Delta q=u &\text{ in }\overline{Q} ,\\ 
				\dfrac{\partial q}{\partial\nu}=0&\text{ on }\overline{\Sigma},\\ 
				q\left( x,0,s\right) =\int\limits_{0}^{A}\int\limits_{0}^{S}\beta(a,\hat{s},s)\exp\left(-\int\limits_{0}^{a}\mu_1(r)dr\right)q(x,a,\hat{s})da d\hat{s}&\text{ in } \overline{Q}_S \\
				q\left(x,a,0\right)=0&
				\text{ in }\overline{Q}_A;\\
			\end{array}
			\right.
		\end{equation} 
		We denote by $q_{\eta}$ the solution of the following system
		\begin{equation}
			\left\lbrace 
			\begin{array}{ll}
				\dfrac{\partial q}{\partial a}+\dfrac{\partial  q}{\partial s}-\Delta  q=u &\text{ in }\overline{Q} ,\\ 
				\dfrac{\partial  q}{\partial\nu}=0&\text{ on }\overline{\Sigma},\\ 
				q\left( x,0,s\right) =\int\limits_{0}^{A}\int\limits_{0}^{S}\beta(a,\hat{s},s)\exp\left(-\int\limits_{0}^{a}\mu_1(r)dr\right)\eta(x,a,\hat{s})  da d\hat{s}&\text{ in } \overline{Q}_S \\
				q\left(x,a,0\right)=0&
				\text{ in }\overline{Q}_A;\\
			\end{array}
			\right.\label{cg}
		\end{equation}
		The system (\ref{cg}) can be considered as the Lotka-McKendrick system with spatial diffusion, but the renewal term not depend of the state of system. Therefore, if $u\geq 0$ and $\eta\geq 0,$ we have a non-negative solution.
		Let $\Lambda$ the operator define by:
		\[L^{\infty}(\overline{Q})\longrightarrow L^{\infty}(\overline{Q}):\eta\longmapsto\Lambda(\eta)=q_{\eta};\] where 
		We denote by $\hat{\eta}=\eta-\eta'$ and $q_{\eta}-q_{\eta'}=q_{\hat{\eta}},$ where $(\eta,\eta')\in \left(L^{\infty}(\overline{Q})\right)^2$.   The function $q_{\hat{\eta}}$ is the solution of (\ref{cg}) where $\eta$ is replaced by $\hat{\eta}$ and $u=0.$\\
		Using the characteristic method, we obtain:
		\begin{equation}
			q_{\hat{\eta}}(x,a,s)=	
			\begin{cases}
				\int\limits_{0}^{A}\int\limits_{0}^{S}\beta(\alpha,\hat{s},s-a)\exp\left(-\int\limits_{0}^{\alpha}\mu(r)dr\right)e^{(s-a)\Delta}\hat{\eta}(x,\alpha,\hat{s})d\hat{s}d\alpha\quad if \quad 0<a<s\\
				0 \quad if\quad 0<s<a.
			\end{cases}
		\end{equation}
		From result of the Lemma 4.1, we obtain
		\[\displaystyle|q_{\hat{\eta}}|\leq \int\limits_{0}^{A}\int\limits_{0}^{S}\beta(\alpha,\hat{s},s-a)\exp\left(-\int_{0}^{\alpha}\mu_1(r)dr\right)d\alpha d\hat{s}\|\hat{\eta}\|_{L^{\infty}(\overline{Q})}\]
		Then
		\[\displaystyle\|q_{\hat{\eta}}\|_{L^{\infty}(\overline{Q})}\leq \displaystyle \|\mathcal{R}\|_{L^\infty(0,S)}\|\hat{\eta}\|_{L^{\infty}(\overline{Q})}.\]
		Finally, if $\|\mathcal{R}\|_{L^{\infty}(0,S)}<1,$ the operator $\Lambda$ is a contracting operator. By the Banach fixed point $\Lambda$ admits a fixed point $\hat{q}.$ By the maximum principle, we conclude that $\hat{q}$ is non-negative. \\ 
		Now, we denote by $p_s(x,a,s)$ the unique non negative solution associated to the control $u$ as in the proposition 4.1. By the comparison method, (the positivity result of \cite{b9} give us the comparison principle) we get that:
		\[p_{s}(x,a,s)\geq p_{i}(x,a,s,t)\quad a.e \hbox{ } (x,a,s,t)\in Q,\]
		where $p_{i}$ solve 
		\begin{equation}
			\left\lbrace
			\begin{array}{ll}
				\dfrac{\partial p_{i}}{\partial t}+\dfrac{\partial p_{i}}{\partial a}+\dfrac{\partial p_{i}}{\partial s}-\Delta  p_{i}+\mu_1(a)p_{i}=u_0 &\text{ in }Q ,\\ 
				\dfrac{\partial p_{i}}{\partial\nu}=0&\text{ on }\Sigma,\\ 
				p_{i}\left( x,0,s,t\right) =\int\limits_{0}^{A}\int\limits_{0}^{S}\beta(a,\hat{s},s)p_{i}(x,a,\hat{s},t)da d\hat{s}&\text{ in } Q_{S,T} \\
				p_{i}\left(x,a,0,t\right)=0&
				\text{ in }Q_{A,T};\\
				p_{i}(x,a,s,0)=0&\text{ in }Q_{A,S}
			\end{array}
			\right..
		\end{equation}
		As the boundary condition in space and  initial condition are zero, then the function $p_i$ doesn't explicitly depend on $x.$ So, we will write $p_i(a,s,t)$ instead of $p_i(x,a,s,t)$ and 
		\[p_s(x,a,s)\geq  p_i(a,s,t)\hbox{ for all } t\geq 0\hbox{ a.e } (x,a,s)\in \overline{Q}.\]
		where $p_i(a,s,t)$ verifies:
		\begin{equation}
			\left\lbrace 
			\begin{array}{ll}
				\dfrac{\partial p_i}{\partial t}+\dfrac{\partial p_i}{\partial a} +\dfrac{\partial p_i}{\partial s}+\mu_1(a)p_i=u_0 &\text{ in }(0,A)\times (0,S)\times (0,T) ,\\  
				p_i\left( 0,s,t\right) =\int\limits_{0}^{A}\int\limits_{0}^{S}\beta(a,\hat{s},s)p_i(a,\hat{s},t)da d\hat{s}&\text{ in } (0,S)\times (0,T) \\
				p_i\left(a,0,t\right)=0&
				\text{ in }\ (0,A)\times (0,S);\\
				p_i\left(a,s,0\right)=0&
				\text{ in }\ (0,A)\times (0,S)\\
			\end{array}
			\right..\label{pre}
		\end{equation}
		The solution $p_i$ of $(\ref{pre})$ is given by:
		\begin{align*}
			p_i(a,s,t)=\left\lbrace
			\begin{array}{l}
				\int\limits_{0}^{t}\dfrac{u_0\pi(a)}{\pi(a-t+l)}dl\quad if \quad 0<t<a\quad 0<t<s\\
				\pi(a)p_i(0,s-a,t-a)+\int\limits_{t-a}^{t}\dfrac{u_0\pi(a)}{\pi(a-t+l)}dl\hbox{ if }0<a<t\hbox{ and }  0<a<s\\
				\int\limits_{t-s}^{t}\dfrac{u_0\pi(a)}{\pi(a-t+l)}dl	 \hbox{ if } 0<s<t<a\hbox{ or }  0<s<a<t.
			\end{array}
			\right.
		\end{align*}
		We have (from the representation $(\ref{semi1})$)
		\[p_i(0,s,t)=\int\limits_{0}^{t}\int\limits_{a}^{S}\beta(a,\hat{s},s)\pi(a)p_i(0,\hat{s}-a,t-a)d\hat{s}da+L(s,t)\] with \[L(s,t)=u_0\int\limits_{0}^{t}\int\limits_{a}^{S}\beta(a,\hat{s},s)\left(\int\limits_{t-a}^{t}\dfrac{\pi(a)}{\pi(a-t-l)}dl\right)d\hat{s}da+\]\[u_0\int\limits_{0}^{A}\int\limits_{0}^{a}\beta(a,\hat{s},s)\left(\int\limits_{t-s}^{t}\dfrac{\pi(a)}{\pi(a-t-l)}dl\right)d\hat{s}da+u_0\int\limits_{t}^{A}\int\limits_{t}^{S}\beta(a,\hat{s},s)\left(\int\limits_{0}^{t}\dfrac{\pi(a)}{\pi(a-t-l)}dl\right)d\hat{s}da.\] For $t>A\hbox{ and } t>S,$ the function $L(s,t)> 0,$ therefore, we have $p_i(0,s,t)>0$
		and $p_i(0,s,t)$ is continuous to respect to $t\text{ and } s$ (see \cite{b9}).\\
		Moreover, $$p_i(a,0,t)=0\hbox{ but }p_i(a,s,t)>0\hbox{ if }s\in (0,S)\hbox{ and }\lim_{a\longrightarrow A}p_i(a,s,t)=0.$$
		As consequence we obtain that there exists $\rho_0(a^*,s_{1}^{*},s_{2}^{*})>0$ such that, for $t$ large enough, and for any $(a,s)\in[0,a^*]\times[s_{1}^{*},s_{2}^{*}],$
		\[p_i(a,s,t)>\rho_0,\]
		and in conclusion we get the result.\\
		If $\mathcal{R}(s)=1\text{ a.e. in }(0,S)\text{ and }u\equiv 0,$ then any function defined by
		$$p_s(x,a,s)=H\exp\left(-\int\limits_{0}^{a}\mu(e)de\right)\text{, }(x,a,s)\in\overline{Q}$$ is a solution of $(\ref{www})$ (for any $H\in \mathbb{R}$). In fact these are all the solutions to $(\ref{www})$ in this case. Therefore there exist infinitely many solutions to $(\ref{www})$, which satisfy $(\ref{222})$.\\
	\end{proof}
	\subsection{Non existence of steady solution}
	In this subsection, we suppose that $\mu(a,s)=\mu_1(a)$ as in the system $(\ref{222})$ and we suppose that $\beta $ does not depend of the size of individuals.\\
	We denote by $$\gamma(a)=\int\limits_{0}^{S}\beta(a,s)ds.$$
	\begin{proposition}
		Denote by $$\mathcal{R}=\int\limits_{0}^{A}\gamma(a)\exp\left(-\int\limits_{0}^{a}\mu_1(r)dr\right)da=\int\limits_{0}^{A}\int\limits_{0}^{S}\beta(a,s)\exp\left(-\int\limits_{0}^{a}\mu_1(r)dr\right)dsda$$	 the reproductive number.\\ 
		\begin{itemize}
			\item Under the assumptions of $\beta\hbox{ and }\mu_1,$ if $u\geq 0\hbox{ a.e. in }\overline{Q}\times(0,+\infty)$ the solution $y$ of the system $(\ref{222})$ verifies \[y(x,a,s,t)\geq 0\hbox{ a.e. in }\overline{Q}\times(0,+\infty).\]
			\item Moreover, if 
			$\mathcal{R}>1$ then $$\lim_{t\longrightarrow+\infty}\|y(t)\|_{L^{2}(\overline{Q})}=+\infty$$ and
			there is no non-negative solution to $(\ref{www}),$ satisfying $(\ref{222}).$
		\end{itemize}
	\end{proposition}
	\begin{proof}
		An idea of Banach fixed point, allows us to show the existence and the unicity of a non-negative solution of $(\ref{222})$ and by the same token, we have the comparison principle ( see \cite{B1} ).\\We now turn to the second part of the proof; for sequel let $y$ verifies $(\ref{222})$
		where the fertility rate $\beta$ is independent of the size of the individuals.\\ We denote by $$Y=\int\limits_{0}^{S_1}yds \hbox{ where } S_1>S.$$ \\
		We can proof that $Y$ verifies
		\begin{equation}
			\left\lbrace 
			\begin{array}{ll}
				Y_t+Y_a-\Delta  Y+\mu_1(a)Y=U &\text{ in }\Omega\times(0,A)\times(0,+\infty) ,\\ 
				\dfrac{\partial  Y}{\partial\nu}=0&\text{ on }\partial\Omega\times(0,A)\times(0,+\infty),\\ 
				Y\left( x,0,t\right) =\int\limits_{0}^{A}\gamma(a)Yda&\text{ in } \Omega\times(0,+\infty) \\
				Y\left(x,a,0\right)=0&
				\text{ in }\Omega\times(0,+\infty);\\
			\end{array}
			\right.\label{cgg}
		\end{equation}
		where $U=\int\limits_{0}^{S}uds$ (here, $u=0\hbox{ if }s>S$).\\
		As $$ u\in L^{\infty}(\overline{Q}\times(0,+\infty))\hbox{ and }u\geq 0\hbox{ a.e. in }\overline{Q}\times(0,+\infty)\hbox{ then }$$ 
		$$U\in L^{\infty}(\Omega\times(0,A)\times(0,+\infty))\hbox{ and }U\geq 0\hbox{ a.e. in }\Omega\times(0,A)\times(0,+\infty).$$ Moreover $$\gamma(a)\in L^{\infty}(0,A)\hbox{ and }\gamma\geq 0\hbox{ in }(0,A).$$
		The system $(\ref{cgg})$ is a Lotka-Mckendrix type with diffusion; therefore, if $$\int\limits_{0}^{A}\gamma(a)\exp\left(-\int_{0}^{a}\mu_(s)ds\right)>1,\hbox{ }\|Y(t)\|_{L^{2}(\Omega\times(0,A))}\longrightarrow +\infty\hbox{ as }t\longrightarrow+\infty.$$
		That mean that
		\[\left\|\left(\int\limits_{0}^{S_1}y(t)ds\right)\right\|^{2}_{L^{2}(\Omega\times(0,A))}\longrightarrow +\infty\hbox{ as }t\longrightarrow+\infty\]
		By Cauchy Schwartz inequality, we obtain 
		\[\left\|\left(\int\limits_{0}^{S_1}y(t)ds\right)\right\|^{2}_{L^{2}(\Omega\times(0,A))}\leq S_1\left\|y(t)\right\|^{2}_{L^{2}(\Omega\times(0,A)\times(0,S_1))}\]
		As $$y(x,a,s,t)=0\hbox{ a.e. in }\Omega\times(0,A)\times(S,S_1)\times(0,+\infty)$$
		then
		\[\left\|\left(\int\limits_{0}^{S_1}y(t)ds\right)\right\|^{2}_{L^{2}(\Omega\times(0,A))}\leq\left\|y(t)\right\|^{2}_{L^{2}(\overline{Q})} \longrightarrow +\infty\hbox{ as }t\longrightarrow+\infty\]
		Therefore, \[\|y(t)\|_{L^2(\overline{Q})}\longrightarrow +\infty\hbox{ as }t\longrightarrow+\infty.\]
		Now, suppose that $\mathcal{R}>1$ and there exists a non-negative solution $p_s$ to $(\ref{www})$ satisfying $(\ref{222});$ that mean that there exists $p(x,a,s,t)$ verifying $(\ref{2})$ with the initial condition $y_0=p_s$ and $$p(x,a,s,t)=p_s(x,a,s)\text{ on }Q .$$ As $\mathcal{R}>1\text{ a.e. in }(0,S),$ the system $(\ref{222})$ is not exponentially stable and we would get  
		$$\|p_s\|_{L^2(\overline{Q})}=\|p(t)\|_{L^2(\overline{Q})}\longrightarrow +\infty\text{ as }t\longrightarrow +\infty,$$ which is absurd.\\
	\end{proof}
	\begin{proposition}	
		There exists a constant $C>0$ such that the solution of $(\ref{2})$ satisfies
		\begin{equation}\displaystyle\|y\|_{L^{\infty}(Q)}\leq C\left(\displaystyle\|u\|_{L^{\infty}(Q)}+\displaystyle\|y_{0}\|_{L^{\infty}(\overline{Q})}\right)\label{ertt}.\end{equation}
		for every $u\in L^{\infty}(Q)$ and $y_0\in L^{\infty}(\overline{Q})$
	\end{proposition}
	\begin{proof}
		Let $L$ be  the operator define by $L=\Delta-\mu_1(a)I,$ we have:\[e^{tL}=\dfrac{\pi_1(a)}{\pi_1(a-t)}e^{t\Delta}\]  is strongly continuous semigroup and verifies the following:
		\[\|e^{tL}\|_{L^{\infty}(\Omega\times(0,A))}\leq  C\|\psi\|_{L^{\infty}(\Omega\times(0,A)
			)}\hbox{ for all } t\geq0 \hbox{ } \psi\in L^{\infty}(\Omega\times(0,A)).\]
		see to Ouhabaz \cite[ \hbox{ Corollary } 4.10]{qu}, for the estimation.\\
		The solution of the system (\ref{2}) without the source term is given by:\\
		\begin{align*}
			y(x,a,s,t)=\left\lbrace
			\begin{array}{l}
				e^{tL}y_0(x,a-t, s-t)\text{ if } 0<t<a\text{ and }  0<t<s,\\
				e^{aL}b_{y_0}(x,s-a,t-a) \text{ if }0<a<t \text{ and } 0<a<s,\\
				0 \text{ otherwise }.
			\end{array},
			\right.
		\end{align*}
		where 
		\begin{equation}
			b_{y_0}(x,s,t)=\displaystyle\int\limits_{0}^{t}\int\limits_{a}^{S}\beta(a,\hat{s},s)e^{aL}b_{y_0}(\hat{s}-a,t-a) d\hat{s}da+\displaystyle\int\limits_{0}^{A-t}\int\limits_{0}^{S-t}e^{tL}\beta(a+t,\hat{s}+t,s)y_0(x,a,\hat{s}) d\hat{s}da.\label{semi}
		\end{equation}
		Let \[V=\{f\in C([0,S], L^{\infty}(\Omega))\},\] and \[W=C([0,T],V)\] where $T>0.$\\ Let
		\begin{equation}
			c_{y_0}(s,t)=\left\lbrace
			\begin{array}{l}
				\displaystyle\int\limits_{0}^{A-t}\int\limits_{0}^{S-t}e^{tL}\beta(a+t,\hat{s}+t,s)y_0(x,a,\hat{s})da d\hat{s}	\text{ if } A>t \text{ and } S>t	\\
				0 \text{ otherwise }
			\end{array}
			\right.
		\end{equation}
		Let $C_{y_0}(t)(s)=c_{y_0}(s,t) \text { if } 0<t<T$ and then $C_{y_0}\in W.$\\
		Define $K:[0,+\infty )\longrightarrow B(V)$ (the space of bounded linear operators in V) by \[(K(a)f)(s)=e^{aL}\int_{a}^{S}\beta(a,\hat{s},s)f(\hat{s}-a)d\hat{s} \text{,  }f\in V \text{ }a\in (0,A)\]
		Then, $K(a)$ is well-defined, and $\beta$ and $f$ are continuous. Equation $(\ref{semi})$ may now be written as an abstract linear Volterra integral equation in $V.$
		\begin{equation}
			B_{y_0}(t)=\int\limits_{0}^{t}K(a)B_{y_0}(t-a)da+C_{y_0}(t)\label{v}
		\end{equation}
		where $B_{\phi}\in W$ and $B_{\phi}=b_{\phi}.$\\
		The equation of Volterra thus defined admits a solution. Moreover
		\begin{equation}
			|B_{y_0}(t)|\leq\int_{0}^{t}|K(a)B_{y_0}(t-a)|da+C\|\beta\|_{\infty}\|y_0\|_{L^\infty(\overline{Q})}.
		\end{equation}
		Using the Gronwall Lemma and the regularity of $\beta$ and $y_0,$ we obtain:
		\[|B_{y_0}(t)|\leq C\|\beta\|_{\infty}\|y_0\|_{L^\infty(\overline{Q})}\exp(\int_{0}^{t}|K(a)|da)\hbox{ for all } t\geq 0\]
		and the  regularity of $K$ gives 
		\[\|B_{y_0}\|_{L^\infty(Q)}\leq C\|\beta\|_{\infty}\exp(CT\|\beta\|_{\infty})\|y_0\|_{L^\infty(\overline{Q})}\hbox{ } C>0\]
		Finally, using the Duhamel formula, we obtain that:
		\[\|y\|_{L^{\infty}(Q)}\leq C\left(\|u\|_{L^{\infty}(Q)}+\|y_{0}\|_{L^{\infty}(\overline{Q})}\right).\]
	\end{proof}
	\subsection{Positivity result}
	So we are in position to state the fourth main result.
	\begin{theorem} 
		Assume the hypothesis of Theorem 1.1. Let $g_s$ and $g_f$ are two non-negative steady states of the system (\ref{2}). Assume that there exist $a_*\in(0,A)\hbox{ and }s_{1}^{*},s_{2}^{*}\in (0,S)\hbox{ with }s_{1}^{*}<s_{2}^{*}$  and $\delta>0$ such that
		\[g_s(x,a,s)\hbox{ and } g_f(x,a,s)\geq\delta\hbox{ a.e on } \Omega\times[0,a^*]\times [s_{1}^{*},s_{2}^{*}].\] Then there exists $T>0$ and $u\in L^{\infty}(Q)$ such that the problem (\ref{2}) with $y_0(x,a,s)=g_s(x,a,s)$ admits a unique solution $y$ satisfying
		\[y(x,a,s,T)=g_f(x,a,s)\hbox{ a.e. }  (x,a,s)\in \overline{Q}\]
		Moreover, \[y(.,a,s,t)\geq 0\hbox{ for a.e. } (x,a,s)\in Q .\]
	\end{theorem}
	\begin{proposition}
		Under the assumption of the theorem 1.1, if $T_0<\min\{a_2-a_1,\hat{a}-a_1\},$ the pair $(\mathcal{A}^*,\mathcal{B}^*)$ is final-state observable for every $T>T_1+A-a_2+T_0$ . In other words, if $T_0<\min\{a_2-a_1,\hat{a}-a_1\},$ for every $T>T_1+A-a_2+T_0$, there exist $K_T>0$ such that the solution $q$ of (\ref{2}) satisfies
		\begin{equation}
			\left(\int\limits_{0}^{S}\int\limits_{0}^{A}\int_{\Omega}q^2(x,a,s,T)dxdads\right)^{\frac{1}{2}}\leq K_T\int\limits_{0}^{T}\int\limits_{s_1}^{s_2}\int\limits_{a_1}^{a_2}\int_{\omega}|q(x,a,s,t)|dxdadsdt.\label{lafin}
		\end{equation}
		\\\end{proposition}
	The proof of the above proposition is similar to that of Theorem 1.1.\\
	In the following Theorem we prove the null controllability of the system (\ref{2}) by means of $L^{\infty}$ controls. Besides the above ingredients, we use a classical duality argument, following closely the methodology in Micu, Roventa and Tucsnak (see \cite{dd}).
	\begin{theorem}
		With the notation and with the assumptions in Theorem 1.1, for every $T>T_1+A-a_2+T_0$ and for every $y_0\in L^{\infty}(\overline{Q})$ there exists a control  $u\in L^{\infty}(Q_1)$ such that the solution $y$ of $(\ref{2})$ satisfies
		$y(x,a,s,T)=0$ for all $(x,a,s)\in \overline{Q}.$ Moreover, there exists a positive constant $R(T)$ such that for $y_0\in L^{\infty}(\overline{Q})$ the control
		function $u$ and the corresponding state trajectory $y$ satisfy
		\begin{equation}\|y\|_{L^{\infty}(Q)}+\|u\|_{L^{\infty}(Q)}\leq R(T)\|y_{0}\|_{L^{\infty}(\overline{Q})}.\label{evell}
		\end{equation}
	\end{theorem}
	\begin{proof}
		We consider the pair $(\mathcal{A},\mathcal{B})$ defined in the section 2.
		As $proved$ in the section $\mathcal{A}$ is a infinitesimal generator of strongly semigroup $U(t)\quad t\geq 0$ of $K$.\\
		Consider the  subspace $\mathcal{X}$ of $L^{1}(Q)$ defined by:\[\mathcal{X}=\{\mathcal{B}^*U^*(t)q_0/q_0\in  K\}\]
		Given $y_0\in L^{\infty}(\overline{Q})$ consider the linear functional $\mathcal{F}$ on $\mathcal{X}$ defined by
		\[\mathcal{F}(\mathcal{B}^*U^*(t)q_0)=-\int\limits_{0}^{S}\int\limits_{0}^{A}\int_{\Omega}y_{0}(U^*(T)q_0).\]
		The fact that this functional is well defined follows from $(\ref{lafin})$ . Moreover we have 
		\[|\mathcal{F}(w)|\leq C\|y_{0}\|_{L^{\infty}(\overline{Q})}\|w\|_{L^{1}(Q)}\hbox{ for all } w\in \mathcal{X}\]
		By the Hahn-Banach Theorem, $\mathcal{F}$ can be extended to a bounded linear functional $\mathcal{\tilde{F}}\in L^{1}(Q)$ such that
		\[|\mathcal{\tilde{F}}(u)|\leq C\|y_{0}\|_{L^{\infty}(\overline{Q})}\|v\|_{L^{1}(Q)}\hbox{ for all } v\in L^{1}(Q).\]
		By the Riesz representation theorem it follows that there exists $u\in L^{\infty}(Q)$ and $R>0$ such that \begin{equation}\|u\|_{L^{\infty}(Q)}\leq R\|q_0\|_{L^{\infty}(\overline{Q})}\label{fim}\end{equation}
		and
		\[\int\limits_{0}^{T}\int\limits_{0}^{S}\int\limits_{0}^{A}\int_{\Omega}u(t-\tau)\mathcal{B}^*U^*(t)q_0+\int\limits_{0}^{S}\int\limits_{0}^{A}\int_{\Omega}y_0U^*(T)q_0=0\quad q_0\in K.\]
		which is equivalent to
		\[\int\limits_{0}^{T}<U(T-\tau)\mathcal{B}u(\tau),q_0>_K+<U(T)y_{0},q_0>_K=0 \quad q_0\in K,\]
		Since the above construction holds for every $q_0\in K$, we get 
		\[y(x,a,s,T)=\int\limits_{0}^{T}U(T-\tau)\mathcal{B}u(\tau)d\tau+U(T)y_{0}=0.\]
		By using $(\ref{ertt})$ and $(\ref{fim})$, we obtain the estimation $(\ref{evell})$
	\end{proof}
	\begin{proof}{Proof of the Theorem 4.1}
		The proof will be in three parts.\\
		\textbf{First part}:\\
		Let $g_s$ and $g_f$ be a two non negative steady states of the system $(\ref{2})$ and let $u_s$ and $u_f$ be  the corresponding steady controls. We set 
		\begin{equation}
			g_{s,j}=\left(1-\frac{j}{M}\right)g_s+\frac{j}{M}g_f \quad  and \quad u_{s,j}=\left(1-\frac{j}{M}\right)u_s+\frac{j}{M}u_f\quad j\in \{ 0,1,...,M\}
		\end{equation} 
		where $M\in \mathbb{N}.$ We assume that there exists $\delta>0$  such that the non negative steady verify
		\[g_s(.,a,s),\hbox{ } g_f(.,a,s)\geq\delta\hbox{ a.e. on }  \Omega\times[0,a_*]\times[s_{1}^{*},s_{2}^{*}]\hbox{ } j\in  \{ 0,1,...,M\}),\] where $a_1<a_2<a^*$ (without the generality) and $s_{1}^{*}< s_{2}^{*}.$ Then we have \[g_{s,j}\geq\delta\quad \hbox{ a.e. on } \Omega\times[0,a_*]\times[s_{1}^{*},s_{2}^{*}]\quad  j\in \{ 0,1,...,M\}.\]
		\textbf{Part 2}
		Using the result of the Theorem 4.2, for $a_2<a_*$ and $T^*>a_*-a_2+T_0+T_1,$ there exist a control $u_j\in L^{\infty}(\omega\times(a_1,a_2)\times(s_1,s_2)\times (0,T^*))$ such the $g_j(x,a,s,T^*)=0$ for all $j\in j\in \{ 0,1,...,M\},$ where $g_j$ is the solution of the following system:\\
		\begin{equation}
			\left\lbrace 
			\begin{array}{ll}
				\dfrac{\partial g_j}{\partial t}+\dfrac{\partial g_j}{\partial a}+\dfrac{\partial g_j}{\partial s}-\Delta g_j+\mu(a)g_j=u_j &\text{ in }Q ,\\ 
				\dfrac{\partial g_j}{\partial\nu}=0&\text{ on }\Sigma,\\ 
				g_j\left( x,0,s,t\right) =\int\limits_{0}^{A}\int\limits_{0}^{S}\beta(a,\hat{s},s)g_j(x,a,\hat{s},t)da d\hat{s}&\text{ in }Q_{S,T} \\
				g_j\left(x,a,0,t\right)=0&
				\text{ in }Q_{A,T};\\
				g_j(x,a,s,0)=(g_{s,j-1}-g_{s,j})&\text{ in }Q_{A,S}\end{array}
			\right..
		\end{equation}	 
		Moreover, there exists $R(T^*)$ such that:
		\[\|g_j\|_{L^{\infty}(\overline{Q}\times(0,T^*))}\leq R(T^*)\|g_{s,j-1}-g_{s,j}\|_{L^{\infty}(\overline{Q})}.\]
		We denoted by 
		\[G_j=g_j+g_{s,j}\quad  U_j=u_j+u_{s,j}\quad (x,a,s,t)\in \overline{Q}\times(0,T^*).\]
		We have:
		\[G_j(x,a,s,0)=g_{s,j-1}(x,a,s)\quad G_j(x,a,s,T^*)=g_{s,j}(x,a,s)\quad (x,a,s,t)\in \overline{Q}\times(0,T^*).\]
		Choosing $M$ sufficiently large large to have 
		\[\|g_{s,j-1}-g_{s,j}\|_{L^{\infty}(\overline{Q})}\leq\dfrac{\delta}{R(T^*)}.\]
		We obtain
		\[\|g_j\|_{L^{\infty}(\overline{Q}\times(0,T^*))}\leq \delta.\]
		Therefore, we get 
		\begin{equation}G_j(x,a,s,t)= g_j(x,a,s,t)+g_{s,j}(x,a,s)\geq 0\hbox{ a.e on } \Omega\times[0,a_*]\times[s_{1}^{*},s_{2}^{*}]\times (0,T^*).\label{vu}\end{equation}
		Let us now consider the sign of $G_j$ on the second part of $Q.$ For this we denote by
		\[\overline{Q}_1=\overline{Q}\backslash\left(\Omega\times[0,a_*]\times[s_{1}^{*},s_{2}^{*}]\right)\text{, } \overline{\Sigma}_1=\overline{\Sigma}\backslash\left(\partial\Omega\times[0,a_*]\times[s_{1}^{*},s_{2}^{*}]\right)\text{ and }\Theta=[0,s_{1}^{*})\cup (s_{2}^{*},S).\] Therefore, $G_j$ verifies follows
		\begin{equation}
			\left\lbrace 
			\begin{array}{ll}
				\dfrac{\partial G_j}{\partial t}+\dfrac{\partial G_j}{\partial a}+\dfrac{\partial G_j}{\partial s}-\Delta G_j+\mu(a)G_j=0 &\text{ in }\overline{Q}_1\times (0,T^*) ,\\ 
				\dfrac{\partial G_j}{\partial\nu}=0&\text{ on }\overline{\Sigma}_1\times(0,T^*),
			\end{array}
			\right.
		\end{equation}
		without the control term (Indeed the support of the control is limited only on the part $[0,a^*].$).	 
		Moreover, we have 
		$$
		G_j\left( x,a_*,s,t\right)\geq 0\text{ in }\Omega\times \Theta\times(0,T^*);$$ 
		$$ 
		G_j\left(x,a,0,t\right)=0\text{ and }G_j\left(x,a,s_{2}^{*},t\right)\geq 0
		\text{ in }\Omega\times(a^*,A)\times (0,T^*)$$
		\hbox{ and }
		$$
		G_j(x,a,s,0)=g_{s,j-1}\geq 0\text{ in }\Omega\times (a^*,A)\times\Theta\times (0,T^*).
		$$
		Using the comparison principle ( see for instance \cite{B1} in the case of a model dependent on age and spatial position), we have :
		\begin{equation}
			G_j(x,a,s,t)\geq 0\hbox{ a.e on } \overline{Q}_1\times (0,T^*).\label{vuu}
		\end{equation}
		Finally, combining $(\ref{vu})\hbox{ and }(\ref{vuu}),$ we get
		\begin{equation}
			G_j(x,a,s,t)\geq 0\hbox{ a.e on } \overline{Q}\times (0,T^*).
		\end{equation}
		\textbf{Part 3}
		We define 
		\begin{equation}
			y(x,a,s,t)=\left\lbrace
			\begin{array}{ll}
				G_1(.,.,.,t)\quad if \quad t\in (0,T^*)\\
				G_2(.,.,.,t-T^*)\quad if \quad t\in (T^*,2T^*)\\
				.\\
				.\\
				.\\
				G_M(.,.,.,t-(M-1)T^*)\quad if \quad t\in ((M-1)T^*,MT^*)
			\end{array}
			\right.
		\end{equation}
		and
		\begin{equation}
			u(x,a,s,t)=\left\lbrace
			\begin{array}{ll}
				U_1(.,.,.,t)\quad if \quad t\in (0,T^*)\\
				U_2(.,.,.,t-T^*)\quad if \quad t\in (T^*,2T^*)\\
				.\\
				.\\
				.\\
				U_M(.,.,.,t-(M-1)T^*)\quad if \quad t\in ((M-1)T^*,MT^*)
			\end{array}
			\right.
		\end{equation}
		The pair $(y,u)$ satisfies $(\ref{2})$ for $T=MT^*.$
		Hence the result.
	\end{proof}
	\section{Further comments and open problems}
	\subsection{Renewal term in the boundary condition of size}
	For size-dependent population dynamics models, we can also have a renewal term at the size boundary condition. Indeed, we consider the following system (see \cite{o}):
	\begin{equation}
		\left\lbrace\begin{array}{ll}
			\dfrac{\partial y}{\partial t}+\dfrac{\partial y}{\partial a}+\dfrac{\partial y }{\partial s}-\Delta y+\mu(a,s)y=mu &\hbox{ in }Q ,\\ 
			\dfrac{\partial y}{\partial\nu}=0&\hbox{ on }\Sigma,\\ 
			y\left( x,0,s,t\right) =\displaystyle\int\limits_{0}^{A}\beta_1(a)y(x,a,s,t)da &\hbox{ in } Q_{S,T} \\
			y(x,a,0,t)=\int\limits_{0}^{S}\beta_2(s)y(x,a,s,t)da ds& \hbox{ in } Q_{A,T}\\
			y\left(x,a,s,0\right)=y_{0}\left(x,a,s\right)&
			\hbox{ in }Q_{A,S}.
		\end{array}\right.
		\label{2gao}
	\end{equation}
	where,  $y_{0}$ is given in $K$, the positive function $\mu_1\text{ and }\mu_1$ denotes respectively the natural mortality rate of individuals of age $a$ and the natural mortality rate of individuals of size $s,$ supposed to be independent of the spatial position $x$ and of time $t$. The control function is $u$, depending on $x$, $a$, $s$ and $t$, where as $m$ is the characteristic function. \\
	We denote by $\beta_1\text{ and }\beta_2$ the positive functions describing respectively whereas the fertility rate age $a$ depending and fertility rate size $s$ depending. The fertility rate $\beta_1\text{ and }\beta_2$ is supposed to be independent of the spatial position $x$ and of time $t,$ so that the density of newly born individuals at the point $x$ at time $t$ of size $s$ privedeing of the boundary condition in age is given by
	\[\int_{0}^{A}\beta_1(a)y(x,a,s,t)da.\]
	and the density of newly born individuals at the point $x$ at time $t$ of size $s$ privedeing of the boundary condition in size is given by
	\[\int_{0}^{S}\beta_2(s)y(x,a,s,t)ds.\]
	We assume that the fertility rate $\beta_1, \text{, }\beta_2$ and the mortality rate $\mu_1\text{ and }\mu_2$ satisfy the demographic property:
	\[
	(H_1)=
	\left\{
	\begin{array}{c}
		\mu_1\left(a\right)\geq 0 \text{ for every } (a)\in (0,A)\\
		\mu_1\in L^{1}\left([0,A^*]\right)\quad\text{ for every }\quad A^*\in [0,A) \\
		\int_{0}^{A}\mu_1(a)da=+\infty
	\end{array} \right..
	\]
	,
	\[
	(H_2)=
	\left\{
	\begin{array}{c}
		\mu_2\left(s\right)\geq 0 \text{ for every } (a)\in (0,A)\\
		\mu_2\in L^{1}\left([0,S^*]\right)\quad\text{ for every }\quad S^*\in [0,A) \\
		\int_{0}^{S}\mu_2(s)ds=+\infty
	\end{array} \right..
	\]
	,
	\[
	(H_{3})=
	\left\{
	\begin{array}{c}
		\beta_1(a)\in C\left([0,A]\right)\text{ and }\beta_1\left(a\right)\geq0 \text{ in } [0,A]\\
		\beta_2\left(s\right)\in C\left([0,S]\right)
		\text{ and }\beta_2\left(s\right)\geq0 \text{ in } [0,S]
	\end{array}
	\right..
	\]
	We consider the following hypotheses:
	\[
	(H_4):
	\begin{array}{c}
		\beta_1(a)=0 \text{ } \forall a\in (0,\hat{a}) \text{ where } \hat{a}\in (0,A). 
	\end{array}
	\]
	\[
	(H_5):
	\begin{array}{c}
		\beta_2(s)=0 \text{ } \forall s\in (0,\hat{s}) \text{ where } \hat{s}\in (0,S). 
	\end{array}
	\]
	Also we denote by \[Q_{1}=\omega\times (a_1,a_2)\times (s_1,s_2),\]with $0\leq a_1<a_2\leq A\text{ and }0\leq s_1<s_2\leq S$ where $\omega\subset \Omega$ is an open set. 
	We denote by $T_2=s_1+S-s_2\text{ and }T_3=a_1+A-a_2.$ \\
	Using our main result of \cite{abst}, we obtain the following results: 
	\begin{theorem}
		Assume that $\beta_1\text{, }\beta_2$ and $\mu=\mu_1+\mu_2$ satisfy the conditions ($H_1)-(H_2)-(H_3)-(H_4)-(H_5)$ above.\\  Assume that $a_1< \hat{a}\text{, } s_1<\hat{s}$ and $T_2<\min\{a_2-a_1,\hat{a}-a_1\}.$	Then for every $T>a_1+A-a_2+2T_2$ and for every $y_0\in L^2(\Omega\times(0,A)\times (0,S))$ 
		there exists a control $u_1\in L^2(Q_1)$ such that the solution $y$ of $(\ref{2})$ satisfies 
		\[ y(x,a,s,T)=0\hbox{ a.e }  .x \in \Omega\text{ } a \in (0,A)\text{ and } s \in (0,S).\]	
	\end{theorem}
	and
	\begin{theorem}
		Assume that $\beta_1\text{, }\beta_2$ and $\mu=\mu_1+\mu_2$ satisfy the conditions ($H_1)-(H_2)-(H_3)-(H_4)-(H_5)$ above.\\  Assume that $a_1< \hat{a}\text{, } s_1<\hat{s}$ and $T_3<\min\{s_2-s_1,\hat{s}-s_1\}.$	Then for every $T>s_1+S-s_2+2T_3$ and for every $y_0\in L^2(\Omega\times(0,A)\times (0,S))$ 
		there exists a control $u_1\in L^2(Q_1)$ such that the solution $y$ of $(\ref{2})$ satisfies 
		\[ y(x,a,s,T)=0\hbox{ a.e. }  x \in \Omega\text{ } a \in (0,A)\text{ and } s \in (0,S).\]	
	\end{theorem}
	Here, the conditions become stronger because the adjoint system has two non-local second members, and the estimation of these non-local terms imposes additional assumptions.\\
	The question arises whether the control time can be improved. We may be tempted to consider the characteristics according to the three temporal variables.\\
	Moreover, some open issues and generalizations remain to be investigated. They are in order:
	\begin{itemize}
		\item\textbf{The growth Modulus depending of size $s$}: One is also interested to the case where the growth function is a variable function depending on $s$. In this case the characteristics lines are no longer vectors but curves.
		\item \textbf{The nonlinear case}: Other possible directions for future extensions of the results and methods in this work concern nonlinear problems (such as considering, for instance, mortality rates depending on the total population or the non-linearity in the newborns.)
		\item \textbf{Numerical implementation}: For a given fertility rate $\beta,$  the mortality rate $\mu,$ the initial condition $y_0$ and a positive parameter $\epsilon>0$, how to determine a numerical algorithm allowing to determine the $\epsilon$-approximate null control function $h?$ 
	\end{itemize}
	\subsection{Null controllability from measurable control support}
	Given a small subset $E\subset\Omega$ (of positive Lebesgue measure or at least not too small in a sense to be made precise later).
	Recall that the Hausdorff content of a set $E\subset \mathbb{R}^n\text{ with }n\in \mathbb{N}$ is 
	$$
	\mathcal{C}^{d}_{\mathcal{H}}(E)=\inf_{d}\{\sum\limits_{j}r_{j}^{d}, E\subset U_{j}B(x_j,r_j)\}
	$$ 
	and the Hausdorff dimension of E is defined as
	$$
	\dim_{\mathcal{H}}(E)=\inf_{d}\{d,\mathcal{C}^{d}_{\mathcal{H}}(E)\}.
	$$
	We shall denote by $|E|$ the Lebesgue measure of the set $|E|.$ We have a following recent result
	establish by Nicolas Burq and Iv\'an Moyano in \cite{ivan}.
	\begin{proposition}(See \cite{ivan})\label{propo1}
		(Null controllability from sets of positive measure). Let $F\subset \Omega\times(0,T)$ of positive Lebesgue measure. Then, there exists $C>0$ such that for any $u_0\in L^2(\Omega)$ the solution $u=\exp(t\Delta)u_0$ to the heat equation
		$$(\partial_t-\Delta)u=0\text{, }u\displaystyle|_{\partial\Omega}=0 \text{ (Dirichlet condition) or } \partial_{\nu}u\displaystyle|_{\partial\Omega}=0  \text{ (Neumann condition), } u\displaystyle|_{t=0}=u_0, $$
		satisfies
		$$\displaystyle\|e^{T\Delta}u_0\|_{L^2(\Omega)}\leq C\int_{F}|u|(x,t)dxdt.$$
		\begin{proposition}
			(Observability and exact controllability from zero measure sets). There exists $\delta\in (0,1)$ (depending only on $n$) which depends only of $n$, such that for any $E\subset \Omega$ of positive $d-\delta$ dimensional Hausdorff content,
			and any $J\subset (0,T)$ of positive Lebesgue measure, there exists $C>0$ such that for any $u_0\in L^2(\Omega)$ the solution $u=\exp(t\Delta)u_0$ to the heat equation
			$$(\partial_t-\Delta)u=0\text{, }u\displaystyle|_{\partial\Omega}=0 \text{ (Dirichlet condition) or } \partial_{\nu}u\displaystyle|_{\partial\Omega}=0  \text{ (Neumann condition), } u\displaystyle|_{t=0}=u_0, $$
			satisfies
			$$\displaystyle\|e^{T\Delta}u_0\|_{L^2(\Omega)}\leq C\int_{J}\sup_{x\in E}|u|(x,t)dt.$$
		\end{proposition}
		Using this previous proposition the condition $\omega$ an open subset of $\Omega$ can be replaced by $\omega$ measurable.
		Thus, by repeating all the preceding calculations by substituting the result of Proposition 3.1 by the result of Proposition 5.1, we obtain the same results with $\omega\subset \Omega$ the measurable (of positive Lebesgue measure) domain.
	\end{proposition}
	\section{Conclusion}
	Considering a Population dynamics with age, size structuring and diffusion, we modeled the issue of null controllability and null controllability with positivity constraints.
	We have first proved a null controllability where  control is localized in the space variable as well as with respect to the age, a null controllability where control is localized in the space variable as well as with respect to the size and the third case, localized is space and localized obliquely with respect to age and size. The second result is a null controllability preserving the positivity.\\
	\textbf{Acknowledgement}\\
	The authors wish to thank Prof. Enrique Zuazua for his comments, suggestions and for fruitful discussions.\\
	The first author would like to thank Debayan Maity for fruitful discussions.\\
	\textbf{Declarations}\\
	-\textbf{Ethical Approval}\\
	No particular ethical approval to state for this article. Authors all approve to participate and publish in your journal.\\
	-\textbf{Competing interests}\\
	The authors declare that they have no conflicts of interest.\\
	-\textbf{Authors contributions}\\
	Y.S. wrote the main manuscript text and prepared figures 1-10.  All authors reviewed the manuscript.\\
	-\textbf{Funding}\\
	\begin{enumerate}
		\item This work has been supported by the European Research Council (ERC) under the European Union’s Horizon 2020 research and innovation programme (grant agreement No 694126-DYCON)
		\item The work of U.B. is supported by the Grant PID2020-112617GB-C22 KILEARN of MINECO (Spain).
	\end{enumerate}
	-\textbf{Availability of data and materials}\\
	No data and materials were used to support this study.
	
\end{document}